\documentclass[11pt, letterpaper]{amsart}
\usepackage{amsthm,amssymb,mathrsfs,mathtools,empheq,mathabx}
\usepackage[shortlabels]{enumitem}
\usepackage[T1]{fontenc}

\usepackage[notref,notcite,final]{showkeys}
\mathtoolsset{showonlyrefs}
\usepackage{color}
\usepackage{geometry}
\allowdisplaybreaks[3]

\newtheorem{mainthm}{Theorem}
\newtheorem{theorem}{Theorem}[section]
\newtheorem*{theorem*}{Theorem}

\newtheorem{lemma}[theorem]{Lemma}
\newtheorem{proposition}[theorem]{Proposition}
\newtheorem*{proposition*}{Proposition}

\newtheorem*{conjecture*}{Conjecture}

\theoremstyle{definition}
\newtheorem{definition}[theorem]{Definition}
\newtheorem{remark}[theorem]{Remark}

\numberwithin{equation}{section}

\def\bN {\mathbb{N}}

\def\bR {\mathbb{R}}

\def\cE {\mathcal{E}}
\def\cF {\mathcal{F}}

\def\scrL{\mathscr{L}}

\def\grad {{\nabla}}

\def\la {\langle}
\def\ra {\rangle}

\newcommand{\tx}[1]{\mathrm{#1}}
\newcommand{\wto}{\rightharpoonup}

\newcommand{\wt}[1]{\widetilde{#1}}
\newcommand{\bs}[1]{\boldsymbol{#1}}
\newcommand{\conj}[1]{\overline{#1}}
\newcommand{\sh}[1]{#1^\sharp}

\newcommand{\shvec}[1]{{\vec #1}\,^\sharp}

\newcommand{\eee}{\mathrm e}

\newcommand{\ud}{\mathrm{\,d}}
\newcommand{\vd}{\mathrm{d}}

\newcommand{\vD}{\mathrm{D}}
\newcommand{\dd}[1]{{\frac{\vd}{\vd{#1}}}}

\newcommand{\qedno}{
\begin{flushright}
$\boxslash$
\end{flushright}
}

\newcommand{\lin}{\scalebox{0.6}{\textrm{\,L}}}

\title[Dynamics of kink clusters for scalar fields in dimension 1+1]
{Dynamics of kink clusters \\ for scalar fields in dimension 1+1}
\author{Jacek Jendrej}
\author{Andrew Lawrie}


\thanks{J. Jendrej is supported by  ANR-18-CE40-0028 project ESSED.  A. Lawrie is supported by NSF grant DMS-1954455 and the Solomon Buchsbaum Research Fund.}

\begin{document}

\begin{abstract}
We consider a real scalar field equation in dimension $1+1$
with an even positive self-interaction potential having two
non-degenerate zeros (vacua) $1$ and $-1$.
It is known that such a model
admits non-trivial static solutions called kinks and antikinks.
A kink cluster is a solution approaching, for large positive times, a superposition of alternating
kinks and antikinks whose velocities converge to $0$.
They can be equivalently characterised as the solutions of minimal possible energy
containing a given number of transitions between the vacua,
or as the solutions whose kinetic energy decays to $0$ for large time.

Our main result is a determination of the main-order asymptotic behaviour of any kink cluster.
Moreover, we construct a kink cluster for any prescribed initial positions of the
kinks and antikinks, provided that their mutual distances are sufficiently large.
Finally, we show that kink clusters are universal profiles for the formation/collapse of multi-kink configurations.
The proofs rely on a reduction,
using appropriately chosen modulation parameters,
to an $n$-body problem
with attractive exponential interactions.
\end{abstract}

\maketitle
\section{Introduction}
\label{sec:intro}
\subsection{Setting of the problem}
\label{ssec:setting}
We study scalar field equations in dimension $1+1$, which are associated to the Lagrangian action
\begin{equation}
\label{eq:lagrange}
\scrL(\phi) = \int_{-\infty}^{\infty}\int_{-\infty}^{\infty} \Big(\frac 12(\partial_t \phi)^2 - \frac 12 (\partial_x\phi)^2 - U(\phi)\Big)\, \ud  x \ud t,
\end{equation}
where the self-interaction potential $U: \bR \to [0, +\infty)$ is a given smooth function.
The unknown field $\phi = \phi(t, x)$ is assumed to be real-valued.
The resulting Euler-Lagrange equation is
\begin{equation}
\label{eq:csf-2nd}
\partial_t^2 \phi(t, x) - \partial_x^2 \phi(t, x) + U'(\phi(t, x)) = 0, \qquad (t, x) \in \bR\times \bR,\ \phi(t, x) \in \bR.
\end{equation}
We assume that
\begin{itemize}
\item
$U$ is an even function,
\item $U(\phi) > 0$ for all $\phi \in (-1, 1)$,
\item  $U(-1) = U(1) = 0$ and $U''(1) = U''(-1) = 1$.
\end{itemize}
The~zeros of $U$ are called the \emph{vacua}.
Linearisation of \eqref{eq:csf-2nd} around each of the vacua
$1$ and $-1$, $\phi = \pm 1 + g$, yields the free linear Klein-Gordon equation of mass $1$:
\begin{equation}
\label{eq:free-kg}
\partial_t^2 g_{\lin}(t, x) - \partial_x^2 g_{\lin}(t, x) + g_{\lin}(t, x) = 0.
\end{equation}

Two well-known examples of~\eqref{eq:csf-2nd} satisfying the hypotheses above are the \emph{sine-Gordon equation}
\begin{equation}
\label{eq:sg} 
\partial_t^2 \phi(t, x) - \partial_x^2 \phi(t, x) -\frac{1}{\pi}\sin(\pi\phi(t, x)) = 0,
\end{equation}
where we have taken $U(\phi) = \frac{1}{\pi^2}\big(1+ \cos(\pi\phi)\big)$, and the \emph{$\phi^4$ model}
\begin{equation}
\label{eq:phi4-m} 
\partial_t^2 \phi(t, x) - \partial_x^2 \phi(t, x) - \frac 12\phi(t, x) + \frac 12\phi(t, x)^3 = 0,
\end{equation}
for which $U(\phi) = \frac{1}{8}(1- \phi^2)^2$.

The equation \eqref{eq:csf-2nd} can be rewritten as a system of first order in $t$:
\begin{equation}
\label{eq:csf-1st}
\partial_t \begin{pmatrix} \phi(t, x) \\ \dot\phi(t, x)\end{pmatrix} = \begin{pmatrix} \dot\phi(t, x) \\ \partial_x^2 \phi(t, x) - U'(\phi(t, x))\end{pmatrix}.
\end{equation}
We denote $\bs \phi_0 = (\phi_0, \dot \phi_0)^\tx T$ an element of the phase space
(in the sequel, we omit the transpose in the notation).
The potential energy $E_p$, the kinetic energy $E_k$ and the total energy $E$
of a state are given by
\begin{align}
E_p(\phi_0) &= \int_{-\infty}^{+\infty}\Big(\frac 12 (\partial_x\phi_0(x))^2 + U(\phi_0(x))\Big)\ud x, \\
E_k(\dot \phi_0) &= \int_{-\infty}^{+\infty}\frac 12( \dot\phi_0(x))^2 \ud x, \\
E(\bs\phi_0) &= \int_{-\infty}^{+\infty}\Big(\frac 12(\dot \phi_0(x))^2+\frac 12 (\partial_x\phi_0(x))^2 + U(\phi_0(x))\Big)\ud x.
\end{align}
Denoting $\bs \phi(t, x) := (\phi(t, x), \dot \phi(t, x))$, the system \eqref{eq:csf-1st} can be reformulated in the Hamiltonian form as
\begin{equation}
\label{eq:csf}
\partial_t \bs \phi(t) = \bs J\vD E(\bs \phi(t)),
\end{equation}
where $\bs J := \begin{pmatrix}0 & 1 \\ {-1} & 0 \end{pmatrix}$ is the standard symplectic form
and $\vD$ is the Fr\'echet derivative for the $L^2\times L^2$ inner product.
In particular, $E$ is a conserved quantity, and we denote $E(\bs \phi)$ the energy of a solution $\bs \phi$ of \eqref{eq:csf}.

The set of finite energy states $\bs \phi_0 = (\phi_0, \dot \phi_0)$ contains the following affine spaces:
\begin{equation}
\begin{aligned}
\cE_{1, 1} &:= \{(\phi_0, \dot \phi_0): E(\phi_0, \dot \phi_0) < \infty\ \text{and}\ \lim_{x \to -\infty}\phi_0(x) = 1, \lim_{x \to \infty}\phi_0(x) = 1\}, \\
\cE_{-1, -1} &:= \{(\phi_0, \dot \phi_0): E(\phi_0, \dot \phi_0) < \infty\ \text{and}\ \lim_{x \to -\infty}\phi_0(x) = -1, \lim_{x \to \infty}\phi_0(x) = -1\}, \\
\cE_{1, -1} &:= \{(\phi_0, \dot \phi_0): E(\phi_0, \dot \phi_0) < \infty\ \text{and}\ \lim_{x \to -\infty}\phi_0(x) = 1, \lim_{x \to \infty}\phi_0(x) = -1\}, \\
\cE_{-1, 1} &:= \{(\phi_0, \dot \phi_0): E(\phi_0, \dot \phi_0) < \infty\ \text{and}\ \lim_{x \to -\infty}\phi_0(x) = -1, \lim_{x \to \infty}\phi_0(x) = 1\}.
\end{aligned}
\end{equation}
In the case of the $\phi^4$ model, these are all the finite-energy states,
but in general there can be other states of finite energy,
for example if $U$ has other vacua than $1$ and $-1$.
Here, the states which we consider will always belong to one of the four affine spaces listed above.

Equation \eqref{eq:csf} admits static solutions. They are the critical points
of the potential energy. The trivial ones are the vacuum fields $\phi(t, x) = \pm 1$.
The solution $\phi(t, x) = 1$ (resp. $\phi(t, x) = -1$) has zero energy and is the ground state in $\cE_{1, 1}$
(resp. $\cE_{-1, -1}$).

There are also non-constant static solutions $\phi(t, x)$
connecting the two vacua, that is
\begin{equation}
\label{eq:static-nn'}
\lim_{x \to -\infty}\phi(t, x) = \mp 1, \quad \lim_{x \to \infty}\phi(t, x) = \pm 1.
\end{equation}
One can describe all these solutions.
There exists a unique increasing odd function $H: \bR \to (-1, 1)$ such that
all the solutions of \eqref{eq:static-nn'} are given by $\phi(t, x) = \pm H(x - a)$ for some $a \in \bR$.
The basic properties of $H$ are given in Section~\ref{sec:static}.
Its translates are called the \emph{kinks} and are the ground states in $\cE_{-1, 1}$.
The translates of the function $-H$ are called the \emph{antikinks} and are the ground states in $\cE_{1, -1}$.



The condition~\eqref{eq:static-nn'} defines a \emph{topological class}, since for any continuous path of finite energy states, either none or all of them satisfy~\eqref{eq:static-nn'}. In general, minimizers of the energy in a topological class that does not contain vacua are called topological solitons. Topological solitons were introduced in the physics literature by Skyrme as candidates for particles in classical field theories; see~\cite{Skyrme, MS}. Kinks and antikinks are one dimensional examples, and in higher dimensions examples include vortices, harmonic maps,  monopoles, Skyrmions, and instantons. In this context, it is natural to investigate to what extent multikinks (defined below) can be effectively described as a system of interacting point particles. In other words, can their dynamical behavior be captured by means of a small number of parameters, for instance their positions and momenta. 


\subsection{Main results}
\label{ssec:results}
By the variational characterisation of $H$ and its translates as the ground states in $\cE_{-1, 1}$,
one can informally view them as the transitions between the two vacua $-1$ and $1$
having the minimal possible energy $E = E_p(H)$.
Given a natural number $n$, we are interested in solutions of \eqref{eq:csf-2nd} containing,
asymptotically as $t \to \infty$, $n$ such transitions.
Since energy $E_p(H)$ is needed for each transition, we necessarily have $E(\phi, \partial_t\phi) \geq nE_p(H)$. We call \emph{kink clusters} the solutions for which equality holds.

\begin{definition}[Kink $n$-cluster]
\label{def:cluster}
Let $n \in \{0, 1, \ldots\}$.
We say that a solution $\bs\phi$ of \eqref{eq:csf} is a \emph{kink $n$-cluster} if
there exist real-valued functions
$x_0(t) \leq x_1(t) \leq \ldots \leq x_n(t)$ such that
\begin{itemize}
\item  $\lim_{t\to\infty}\phi(t, x_k(t)) = (-1)^k$ for $k \in \{0, 1, \ldots, n\}$,
\item $E(\bs\phi) \leq nE_p(H)$.
\end{itemize}
\end{definition}
Note that the kink $0$-clusters are the constant solutions $\phi \equiv 1$ and the kink $1$-clusters are the antikinks.
We say that $\phi$ is a kink cluster if it is a kink $n$-cluster for some $n \in \{0, 1, \ldots\}$.

From the heuristic discussion above, the shape of each transition
in a kink cluster has to be close to optimal,
that is close to a kink or an antikink.
Before we give a precise statement of this fact,
we introduce the so-called \emph{multi-kink configurations}.
For $\vec a = (a_1, \ldots, a_n)$ such that $a_1 \leq \ldots \leq a_n$, we denote
\begin{equation}
H(\vec a) := 1 + \sum_{k=1}^{n} (-1)^k\big(H(\cdot - a_k) + 1\big)
\end{equation}
(we chose the ``additive ansatz'', see \cite[Section 1.7]{vachaspati}
for a comparison with a different ``product ansatz'', which we could also use
without introducing any changes in the statement of our results below).
\begin{proposition}
\label{prop:close-to-H}
A solution $\bs\phi$ of \eqref{eq:csf} is a kink $n$-cluster if and only if there exist continuous functions $a_1, \ldots, a_n: \bR \to \bR$
such that
\begin{equation}
\label{eq:close-to-H}
\begin{aligned}
\lim_{t \to \infty}\Big(\big\|\partial_t \phi(t)\big\|_{L^2}^2 + \big\|\phi(t) - H(\vec a(t))\big\|_{H^1}^2 + \sum_{k=1}^{n-1}\eee^{-(a_{k+1}(t) - a_k(t))}\Big) = 0.
\end{aligned}
\end{equation}
\end{proposition}
In other words, kink clusters can be equivalently defined as solutions approaching,
as $t \to \infty$, a superposition of a finite number of alternating kinks and antikinks,
whose mutual distances tend to $\infty$ and which travel with speeds converging to $0$.
In contrast to multikink solutions consisting of Lorentz-boosted kinks (travelling
with asymptotically non-zero speed) constructed in \cite{CJ2},
the dynamics of kink clusters are driven solely by interactions between the kinks and antikinks. Employing the term introduced by Martel and Rapha\"el in \cite{MaRa18},
we are dealing with multi-kinks in the regime of \emph{strong interaction}.
Proposition~\ref{prop:close-to-H} is proved in Section~\ref{ssec:close-to-H}.
It implies in particular that the energy of a kink $n$-cluster equals $nE_p(H)$.

In Section~\ref{ssec:asym-stat}, we provide another characterisation of kink clusters, namely as \emph{asymptotically
static solutions}, by which we mean solutions whose kinetic energy
converges to $0$ as $t \to \infty$.
For simplicity, we restrict our attention
to the $\phi^4$ self-interaction potential $U(\phi) := \frac 18(1-\phi^2)^2$.
\begin{proposition}
\label{prop:asym-stat}
Let $U(\phi) := \frac 18(1-\phi^2)^2$.
A solution $\bs\phi$ of \eqref{eq:csf} satisfies $\lim_{t\to \infty}\|\partial_t \phi(t)\|_{L^2}^2 = 0$ if and only if $\bs \phi$ or
$-\bs\phi$ is a kink cluster.
\end{proposition}


Our main result is the determination
of the asymptotic behaviour of any kink cluster.
\begin{mainthm}
\label{thm:asymptotics}
Let $\kappa > 0$ be given by Proposition~\ref{prop:prop-H}
and $M := E_p(H)$.
If $\bs\phi$ is a kink $n$-cluster, then there exist
continuously differentiable functions $a_1, \ldots, a_n: \bR \to \bR$
such that $g(t) := \phi(t) - H(\vec a(t))$ satisfies
\begin{equation}
\begin{aligned}
\lim_{t\to\infty}\bigg(&\max_{1\leq k < n}\bigg|\big(a_{k+1}(t) - a_k(t)\big) - \Big(2\log(\kappa t) - \log\frac{Mk(n-k)}{2}\Big)\bigg| \\
+ &\max_{1\leq k\leq n}|ta_k'(t) + (n+1 - 2k)|
+ t\|\partial_t g(t)\|_{L^2} + t\|g(t)\|_{H^1}\bigg) = 0.
\end{aligned}
\end{equation}
\end{mainthm}
The decomposition $\phi(t) = H(\vec a(t)) + g(t)$ used in the statement above
is clearly not unique. We will use a specific choice of $\vec a(t)$ determined
by the \emph{orthogonality conditions}
\begin{equation}
\label{eq:g-orth-intro}
\int_{-\infty}^\infty\partial_x H(x - a_k(t)) g(t, x)\ud x = 0, \qquad\text{for all }k \in \{1, \ldots, n\}.
\end{equation}
This way, whenever $\phi(t)$ is close to a multi-kink configuration,
the uniquely determined number $a_k(t)$ indicates the ``position'' of the $k$-th kink.

Our next result concerns the problem of existence of kink $n$-clusters.
We prove that, for any choice of $n$ points on the line sufficiently distant from each other,
there exists a kink $n$-cluster such that the initial positions of the (anti)kinks
are given by the $n$ chosen points.
The result is inspired by the work
of Maderna and Venturelli~\cite{MaVe09} on the Newtonian $n$-body problem.
Before we give the precise statement, we introduce the following notion.
\begin{definition}[Distance to a multi-kink configuration]
\label{def:d-def}
For all $\bs \phi_0 \in \cE_{1, (-1)^n}$, the distance from $\bs \phi_0 = (\phi_0, \dot\phi_0)$
to the set of multi-kink configurations is defined by
\begin{equation}
\label{eq:d-def}
\delta(\bs \phi_0) := \inf_{\vec b \in \bR^n}\Big(\|\dot \phi_0\|_{L^2}^2 + \|\phi_0 - H(\vec b)\|_{H^1}^2 + \sum_{k=1}^{n-1} \eee^{-(b_{k+1} - b_k)}\Big).
\end{equation}
\end{definition}
Note that, by Proposition~\ref{prop:close-to-H}, if $\bs\phi$ is a kink $n$-cluster,
then $\lim_{t\to\infty}\delta(\bs\phi(t)) = 0$.
We stress that the position parameters determined by the orthogonality
conditions do not necessarily achieve the infimum above,
but they do achieve it up to a constant, see Lemma~\ref{lem:static-mod}.
\begin{mainthm}
\label{thm:any-position}
There exist $C_0, L_0 > 0$ such that the following is true.
If $L \geq L_0$ and $\vec a_0 \in \bR^n$ satisfies $a_{0, k+1} - a_{0, k} \geq L$ for all
$k \in \{1, \ldots, n-1\}$, then there exists $\bs g_0 = (g_0, \dot g_0) \in \cE$ satisfying $\|\bs g_0\|_{\cE}^2 \leq C_0 \eee^{-L}$ and the orthogonality conditions
\begin{equation}
\label{eq:g0-orth}
\int_{-\infty}^\infty \partial_x H(x - a_{0, k})g_0(x)\ud x = 0 \qquad\text{for all }k \in \{1, \ldots, n\}
\end{equation}
such that the solution of \eqref{eq:csf} corresponding to the initial data
\begin{equation}
\bs\phi(0) = \bs \phi_0 := \big(H(\vec a_0) + g_0, \dot g_0\big)
\end{equation}
is a kink cluster and satisfies $\delta(\bs \phi(t)) \leq C_0/(e^L + t^2)$ for all $t \geq 0$.
\end{mainthm}
\begin{remark}
We expect that for a given choice of $\vec a_0$ there is actually a \emph{unique}
$\bs g_0$ in a~small ball of $\cE$ leading to a kink $n$-cluster.
This is clearly true for $n = 1$. In the case $n = 2$,
uniqueness of $\bs g_0$ can be obtained as
a consequence of our work with Kowalczyk~\cite{JKL1}.
Partial uniqueness results for $n > 2$ will be proved in our future work.
\end{remark}
\begin{remark}
In the case of equation \eqref{eq:sg}, which is completely integrable,
it is in principle possible to obtain explicit kink clusters.
For $n \in \{2, 3\}$, such examples of kink clusters where given in \cite{Tomasz}. Nevertheless, Theorems~\ref{thm:asymptotics}
and \ref{thm:any-position} are new even for the sine-Gordon equation.
\end{remark}

Finally, Section~\ref{sec:profiles} is devoted to the 
role of the kink clusters
as universal profiles for the formation/collapse of a multi-kink configuration.
\begin{mainthm}
\label{thm:unstable}
Let $\eta > 0$ be sufficiently small and let $\bs\phi_m$ be a sequence of solutions of \eqref{eq:csf} defined on time intervals $[0, T_m]$
satisfying the following assumptions:
\begin{enumerate}[(i)]
\item $\lim_{m\to \infty}\delta(\bs\phi_m(T_m)) = 0$,
\item $\delta(\bs \phi_m(t)) \leq \eta$ for all $t \in [0, T_m]$,
\item $\delta(\bs \phi_m(0)) = \eta$.
\end{enumerate}
Then, after extraction of a subsequence, there exist $0 = n^{(0)} < n^{(1)} < \ldots < n^{(\ell)} = n$,
finite energy states $\bs P_0^{(1)}, \ldots, \bs P_0^{(\ell)}$ and sequences of real numbers $(X_m^{(1)})_m, \ldots, (X_m^{(\ell)})_m$ such that
\begin{enumerate}[(i)]
\item for all $j \in \{1, \ldots, \ell\}$, the solution $\bs P^{(j)}$ of \eqref{eq:csf}
for the initial data $\bs P^{(j)}(0) = \bs P_0^{(j)}$ is a cluster of $n^{(j)} - n^{(j-1)}$ kinks,
\item for all $j\in \{1, \ldots, \ell-1\}$, $\lim_{m\to \infty}\big( X_m^{(j+1)} - X_m^{(j)}\big) = \infty$,
\item $
\lim_{m\to \infty} \Big\|\bs \phi_m(0) - \Big(\bs 1 + \sum_{j = 1}^\ell(-1)^{n^{(j-1)}} \big(\bs P_0^{(j)}(\cdot - X_m^{(j)}) - \bs 1\big)\Big)\Big\|_\cE = 0.
$
\end{enumerate}
\end{mainthm}

Theorem~\ref{thm:unstable} can be understood to mean that kink clusters have properties similar to the stable/unstable manifolds of a hyperbolic stationary state. This analogy is most easily understood in the case $n =2$, which we explain here.

If we artificially extended the phase space by a state $\bs H^\infty$ corresponding
to the limit of $(H(a_1, a_2), 0)$ as $a_2 - a_1 \to \infty$,
then the function $\delta$ gives a distance to $\bs H^\infty$ and the $2$-kink clusters satisfy $ \lim_{t \to \infty} \delta( \bs \phi(t)) = 0$, in other words they form the \emph{stable manifold} of $\bs H^\infty$.

In this language, Theorem~\ref{thm:unstable} characterizes the trajectories in the phase space that \emph{enter} (or in reverse time, \emph{exit}), a small neighbourhood of the ``critical point'' $\bs H^\infty$, by affirming that a such a trajectory,
while still far away from the critical point,
must be close to its (un)stable manifold.
For hyperbolic critical points, this property is a consequence
of the Hartman-Grobman theorem.
In our case, the soliton interactions play an analogous role
as exponential (in)stability in the hyperbolic case.

The analogy described above carries over to $n>2$, but is slightly more complicated, since at the ``exit'' time $t =0$ the solution $\bs\phi_m(0)$ is close to a superposition of well-separated kink clusters, rather than to a single one. Intuitively, for $n > 2$ it can happen
that only some of the neighbouring kinks ``collapse'', while
the distances between other neighbouring kinks remain large.

\subsection{Structure of the paper and main ideas}
In Section~\ref{sec:static}, we recall the basic properties of the stationary solutions
and compute the first non-trivial term in the asymptotic expansion
of the potential energy of a multi-kink configuration, $E_p(H(\vec a))$,
as the distances between the kinks tend to infinity.

Section~\ref{sec:cauchy} is devoted to a brief presentation of the well-posedness theory
of the equation \eqref{eq:csf}.

In Section~\ref{sec:mod}, we implement the \emph{modulation method},
also called the ``method of collective coordinates'' in the physics literature.
The idea is to rewrite \eqref{eq:csf} as a coupled system of equations for
the \emph{modulation parameters} $\vec a(t)$ and the \emph{remainder} $\bs g(t)$.

If we consider an isolated system of $n$ points on the real line, located at $a_1, a_2, \ldots, a_n$, whose masses are equal to $M > 0$ and the total potential energy of the system is given by $E_p(H(a_1, \ldots, a_n))$,
then the principles of Newtonian mechanics assert that the positions of the masses evolve according to the system of differential equations
\begin{equation}
\label{eq:newton-ode}
p_k(t) = M a_k'(t), \qquad p_k'(t) = F_k(a_1(t), \ldots, a_n(t)) = -\partial_{a_k}E_p(H(a_1, \ldots, a_n)).
\end{equation}
The main conclusion of Section~\ref{sec:mod} is Lemma~\ref{lem:ref-mod},
which states that the modulation parameters of a kink cluster
\emph{approximately} satisfy this system of ODEs.
This step relies on appropriately defined \emph{localised momenta},
a method first used in \cite{J-18p-gkdv} in a similar context,
and inspired by \cite[Proposition 4.3]{RaSz11}.

After replacing $E_p(H(\vec a))$ by its leading term computed in Section~\ref{sec:static},
the system \eqref{eq:newton-ode} becomes an $n$-body problem with \emph{attractive exponential
nearest-neighbour interactions}
\begin{equation}
\label{eq:attractive-toda}
p_k(t) = Ma_k'(t), \qquad p_k'(t) = 2\kappa^2\big(e^{-(a_{k+1}(t) - a_k(t))} - e^{-(a_k(t) - a_{k-1}(t))}\big),
\end{equation}
where by convention $a_0(t) := -\infty$ and $a_{n+1}(t) := \infty$.
In the case of repulsive interactions, we would have the well-known \emph{Toda system} introduced in \cite{Toda70}, so we shall call \eqref{eq:attractive-toda} the ``attractive Toda system''.
Section~\ref{sec:n-body} is devoted to the study of the long-time behaviour of
solutions of this $n$-body problem. The solutions corresponding to kink clusters
are the ones for which the distances between the masses tend to $\infty$
and their momenta tend to $0$ as $t \to \infty$.
Such solutions are referred to as \emph{parabolic motions}, see \cite{Saari2, MaVe09}.
One can check that \eqref{eq:attractive-toda} has an exact solution
(defined up to translation, the mass centre being chosen arbitrarily)
\begin{equation}
\label{eq:explicit-parabolic}
a_{k+1}(t) - a_k(t) = 2\log(\kappa t) - \log\frac{Mk(n-k)}{2}, \qquad p_k(t) = M\frac{n+1-2k}{t}.
\end{equation}
In Section~\ref{sec:n-body}, we prove that the leading order
of the asymptotic behaviour of $a_{k+1}(t) - a_k(t)$ and $p_k(t)$ for any parabolic motion
coincides with \eqref{eq:explicit-parabolic}.
Our analysis is sufficiently robust to be valid also in the presence of
the error terms obtained from Lemma~\ref{lem:ref-mod},
which leads to a proof of Theorem~\ref{thm:asymptotics}.

Let us mention that H\'enon \cite{Henon74} found $n$ independent conserved quantities
for the Toda system (both in the repulsive and in the attractive case).
For parabolic motions, all these quantities are equal to $0$,
which allows to reduce the problem to a system of $n$ equations of 1st order.
Probably, this approach could lead to some simplifications in determining
the asymptotic behaviour of the parabolic motions of the attractive Toda system,
and perhaps also of the approximate system satisfied by the modulation parameters.
Our arguments do not explicitly rely on the conservation laws related to the complete
integrability of the Toda system, and we expect that part of the analysis
will be applicable also in the cases where the modulation equations are not related
to any completely integrable system of ODEs.

Theorem~\ref{thm:any-position} is proved in Section~\ref{sec:any-position}.
The overall proof scheme is taken from Martel~\cite{Martel05},
see also the earlier work of Merle~\cite{Merle90}, and contains two steps:
\begin{itemize}
\item for any $T > 0$, prove existence of a solution $\bs \phi$ satisfying
the conclusions of Theorem~\ref{thm:any-position}, but only on the finite time interval $t \in [0, T]$,
\item take a sequence $T_m \to \infty$ and consider a weak limit
of the solutions $\bs \phi_m$ obtained in the first step with $T = T_m$.
\end{itemize}

The first step relies on a novel application of the Poincar\'e-Miranda theorem,
which is essentially a version of Brouwer's fixed point theorem.
We choose data close to a multi-kink configuration at time $t = T$ and control
how it evolves backwards in time.
It could happen that the multi-kink collapses before reaching the time $t = 0$.
For this reason, we introduce an appropriately defined ``exit time'' $T_1$.
The mapping which assigns the positions of the (anti)kinks at time $T_1$
to their positions at time $T$ turns out to be continuous and, for topological reasons,
surjective in the sense required by Theorem~\ref{thm:any-position}.

In the second step, it is crucial to dispose of some \emph{uniform} estimate on the sequence $\bs\phi_m$. In our case, the relevant inequality is $\delta(\bs\phi_m(t)) \lesssim (\eee^L + t^2)^{-1}$ with a universal constant.
The existence of such a uniform bound is related to what we would call the ``ejection property''
of the system. Intuitively, once $\delta(\bs\phi(t))$ starts to grow,
it has to continue growing at a definite rate until the multi-kink configuration collapses.

In Section~\ref{sec:profiles}, we give a proof of Theorem~\ref{thm:unstable}.
The identification of the clusters presents no difficulty: the positions of any two consecutive (anti)kinks at time $t = 0$, after taking a subsequence in $m$, either remain at a bounded distance
or separate with their distance growing to infinity as $m\to\infty$.
This dichotomy determines whether
they fall into the same cluster or to distinct ones.
The next step is to again make use of the ejection property in order to obtain
bounds on $\delta(\bs\phi(t))$ independent of $m$, for any $t \geq 0$.
By standard localisation techniques involving the finite speed of propagation,
these bounds are inherited by each of the clusters.
We mention that the proof of strong convergence in Theorem~\ref{thm:unstable} (iii)
is based on a novel application of the well-known principle from the Calculus of Variations
affirming that, for a strictly convex functional $\cF$, if $\bs g_m \wto \bs g$
and $\cF(\bs g_m) \to \cF(\bs g)$, then $\bs g_m \to \bs g$.

\subsection{Other related results}
The first construction of a two-soliton solution with trajectories having asymptotically vanishing velocities was obtained by Krieger, Martel and Rapha\"el \cite{KrMaRa09},
see also \cite{MaRa18, Vinh17} for other constructions
and \cite{WadOhk} for related computations in the completely integrable setting.
Existence of strongly interacting multi-solitons with an arbitrary number of solitons
was obtained by Lan and Wang \cite{LaWa22p} for the generalized Benjamin-Ono equation.

The particle-like character of solitons is a well-known phenomenon, see \cite[Chapter 1]{MS}
for a historical account.
The question of justification that the positions of solitons satisfy
an approximate $n$-body law of motion was considered for instance in \cite{Stuart,GuSi06,DuMa,OvSi}.

In their work on blow-up for nonlinear waves, Merle and Zaag~\cite{MeZa12-AJM}
obtained a system of ODEs with exponential terms like in \eqref{eq:attractive-toda},
but which was a \emph{gradient flow} and not an $n$-body problem.
The dynamical behaviour of solutions of this system was described by C\^ote and Zaag \cite{CoteZaag}.

In relation with our proof of Theorem~\ref{thm:any-position},
we note that Brouwer's theorem was previously used in constructions of multi-solitons,
but for a rather different purpose, namely in order to avoid the growth of linear unstable modes, see \cite{CMM11, CoteMunoz}.

Determining universal profiles of soliton collapse played an important
role in several works on dispersive equations related to the problem of Soliton Resolution.
We mention the study of centre-stable manifolds of ground states for various nonlinear wave equations, see for instance \cite{NaSc11-1,NaSc11-2,KrNaSc15},
as well as the earlier work \cite{DM08}.

Let us stress again that the main object of our study are
solutions approaching multi-soliton configurations in the strong energy norm, in other words we address the question of interaction of solitons \emph{in the absence of radiation}.
Allowing for a radiation term seems to be currently out of reach,
the question of the asymptotic stability of the kink being
still unresolved, see for example~\cite{DelortMasmoudi, KMM, GermainPusateri, HN2, CLL, LuhrmannSchlag} for recent results
on this and related problems.

Finally, we emphasize that our definition of kink clusters
concerns \emph{only one time direction}, and our study
does not address the question of the behaviour of kink clusters as $t \to {-}\infty$, which goes by the name of the \emph{kink collision problem}.
We refer to
\cite{KevrekidisEtc}
for an overview,
and to \cite{Abdon22p2, Abdon22p3} for recent rigorous results
in the case of the $\phi^6$ model.



\subsection{Notation}
\label{ssec:notation}
Even if $v(x)$ is a function of one variable $x$, we often write $\partial_x v(x)$
instead of $v'(x)$ to denote the derivative. The prime notation is only used
for the time derivative of a function of one variable $t$
and for the derivative of the potential $U$.

If $\vec a, \vec b \in \bR^n$, then $\vec a\cdot \vec b := \sum_{k=1}^n a_k b_k$.
If $u$ and $v$ are (real-valued) functions, then $\la u, v\ra := \int_{-\infty}^\infty u(x)v(x)\ud x$.

Boldface is used for pairs of values
(which will usually be a pair of functions
forming an element of the phase space).
A small arrow above a letter indicates a vector
with any finite number of components
(which will usually be the number of kinks
or the number of kinks diminished by 1).
If $\bs u = (u, \dot u) \in L^2(\bR)\times L^2(\bR)$ and $\bs v = (v, \dot v) \in L^2(\bR)\times L^2(\bR)$, we write
$\la \bs u, \bs v\ra := \la u, v\ra + \la \dot u, \dot v\ra$.

We will have to manipulate finite sequences and sums. In order to make the formulas reasonably compact,
we need to introduce appropriate notation, some of which is not completely standard.
If $\vec w = (w_j)_j$ is a vector and $f: \bR \to \bR$ is a function, we denote $f(\vec w)$ the vector with components $f(w_j)$,
for example $\eee^{\vec w}$ will denote the vector $(\eee^{w_j})_j$, $\vec w\,^2$ the vector $(w_j^2)_j$ and $\log\vec w$ the vector $(\log w_j)_j$.
If $\vec v = (v_j)_j$ is another vector, we denote $\vec w\vec v := (w_jv_j)_j$.
We also denote $w_{\min} := \min_j w_j$.

When we write $\simeq$, $\lesssim$ or $\gtrsim$, it should be understood that the constant is allowed to depend only on $n$.
Our use of the symbol $\sh f \sim f$ is non-standard and indicates that $\sh f - f$ is a negligible
quantity (the meaning of ``negligible'' will be specified in each case),
without requiring that $\sh f / f$ be close to $1$.

The energy space is denoted $\cE := H^1(\bR) \times L^2(\bR)$.
We also use local energies and energy norms defined as follows.
For $-\infty \leq x_0 < x_0' \leq \infty$ and $\phi_0: [x_0, x_0']\to \bR$, we denote
\begin{align}
E_p(\phi_0; x_0, x_0') &:= \int_{x_0}^{x_0'}\Big(\frac 12(\partial_x \phi_0(x))^2 + U(\phi_0)\Big)\ud x, \\
E(\bs \phi_0; x_0, x_0') &:= \int_{x_0}^{x_0'}\Big(\frac 12 (\dot \phi_0(x))^2 + \frac 12(\partial_x \phi_0(x))^2 + U(\phi_0)\Big)\ud x, \\
\|\phi_0\|_{H^1(x_0, x_0')}^2 &:= \int_{x_0}^{x_0'}\big((\partial_x \phi_0(x))^2 + \phi_0(x)^2\big)\ud x, \\
\|\bs \phi_0\|_{\cE(x_0, x_0')}^2 &:= \int_{x_0}^{x_0'}\big((\dot \phi_0(x))^2 + (\partial_x \phi_0(x))^2 + \phi_0(x)^2\big)\ud x.
\end{align}

The open ball of center $c$ and radius $r$ in a normed space $A$ is denoted $B_A(c, r)$.
We denote $\vD$ and $\vD^2$ the first and second Fr\'echet derivatives of a functional.

We write $x_+ := \max(0, x)$.

We take $\chi: \bR \to [0, 1]$ to be a decreasing $C^\infty$ function
such that $\chi(x) = 1$ for $x \leq \frac 13$
and $\chi(x) = 0$ for $x \geq \frac 23$.

Proofs end with the sign $\boxvoid$. Statements given without proof end with the sign $\boxslash$.

\section{Kinks and interactions between them}
\label{sec:static}
\subsection{Stationary solutions}
A stationary field $\phi(t, x) = \psi(x)$ is a solution of \eqref{eq:csf} if and only if
\begin{equation}
\label{eq:psi4}
\partial_x^2\psi(x) = U'(\psi(x)),\qquad\text{for all }x\in \bR.
\end{equation}
We seek solutions of \eqref{eq:psi4} having finite potential energy $E_p(\psi)$.
Since $U(\psi) \geq 0$ for $\psi \in \bR$, the condition $E_p(\psi) < \infty$ implies
\begin{align}
\label{eq:psi4-H1}
&\int_{-\infty}^{+\infty}\frac 12 (\partial_x \psi(x))^2 \ud x < \infty, \\
\label{eq:psi4-U}
&\int_{-\infty}^{+\infty}U(\psi(x)) \ud x < \infty.
\end{align}
From \eqref{eq:psi4-H1} we have $\psi \in C(\bR)$,
so \eqref{eq:psi4} and $U \in C^\infty(\bR)$ yield $\psi \in C^\infty(\bR)$.
Multiplying \eqref{eq:psi4} by $\partial_x \psi$ we get
\begin{equation}
\partial_x\Big(\frac 12 (\partial_x \psi)^2 - U(\psi)\Big)
= \partial_x \psi\big(\partial_x^2 \psi - U'(\psi)\big) = 0,
\end{equation}
so $\frac 12 (\partial_x \psi(x))^2 - U(\psi(x)) = k$ is a constant.
But then \eqref{eq:psi4-H1} and \eqref{eq:psi4-U} imply $k = 0$.
We obtain first-order autonomous equations, called the Bogomolny equations,
\begin{equation}
\label{eq:bogom-eq}
\partial_x\psi(x) = \sqrt{2U(\psi(x))}\quad\text{or}\quad \partial_x\psi(x) = -\sqrt{2U(\psi(x))},\quad\text{for all }x \in \bR,
\end{equation}
which can be integrated in the standard way, see for instance \cite[Section 2]{JKL1}.
We conclude that:
\begin{itemize}
\item the only stationary solution of \eqref{eq:csf} belonging to $\cE_{1, 1}$ is the vacuum
state $\bs 1 := (1, 0)$ and the only stationary solution of \eqref{eq:csf} belonging to $\cE_{-1, -1}$ is the vacuum
state $-\bs 1 := (-1, 0)$,
\item the only stationary solutions of \eqref{eq:csf} belonging to $\cE_{-1, 1}$
are the translates of $\bs H := (H, 0)$, where the function $H$ is defined by
\begin{equation}
\label{eq:H-def}
H(x) = G^{-1}(x), \qquad\text{with}\ \ G(\psi) := \int_{0}^\psi \frac{\ud y}{\sqrt{2U(y)}}\ \ \text{for all }\psi \in ({-}1, 1),
\end{equation}
and the solutions of \eqref{eq:psi4} belonging to $\cE_{1, -1}$
are the translates of $-\bs H := (-H, 0)$.

\end{itemize}

The asymptotic behaviour of $H(x)$ for $|x|$ large is essential for our analysis.
We have the following result, see \cite[Proposition 2.1]{JKL1}.
\begin{proposition}
\label{prop:prop-H}
The function $H(x)$ defined by \eqref{eq:H-def} is odd, of class $C^\infty(\bR)$
and there exist constants $\kappa > 0$ and $C > 0$ such that for all $x \geq 0$
\begin{equation}
\label{eq:H-asym-p}
\begin{aligned}
\big|H(x)-1+ \kappa \eee^{- x}\big| 
+ \big|\partial_x H(x) -  \kappa \eee^{- x}\big| 
+ \big|\partial_x^2 H(x) + \kappa \eee^{- x}\big| \leq C\eee^{-2 x},
\end{aligned}
\end{equation}
and for all $x \leq 0$
\begin{gather}
\label{eq:H-asym-m}
\big|H(x)+1- \kappa \eee^{ x}\big| + \big|\partial_x H(x) -  \kappa \eee^{ x}\big| + \big|\partial_x^2 H(x) -  \kappa \eee^{ x}\big| \leq C\eee^{2 x}.
\end{gather}
\vspace{-1cm}\qedno
\end{proposition}

We denote
\begin{align}
M := \|\partial_x H\|_{L^2}^2 = 2\int_0^1\sqrt{2U(y)}\ud y
= 2\int_{-\infty}^\infty U(H(x))\ud x = E_p(H),
\label{eq:M-def}
\end{align}
all these equalities following from \eqref{eq:bogom-eq} and the change of variable $y = H(x)$.
In the context of Special Relativity,
one can think of $M$ as the (rest) mass of the kink.

We will also use the fact, checked in \cite[Section 2.1]{JKL1}, that
\begin{align}
\label{eq:reduced-force}
\int_{-\infty}^\infty \partial_x H(x)\left(U''(H(x)) - 1\right)\eee^{x} \ud x = -2 \kappa.
\end{align}

Finally, we recall the following \emph{Bogomolny trick} from \cite{Bogom76}.
If $\phi_0(x_0) \leq \phi_0(x_0')$ (understood as limits if $x_0$ or $x_0'$ is infinite), then
\begin{equation}
\label{eq:bogom}
\begin{aligned}
E_p(\phi_0; x_0, x_0') &= \frac 12\int_{x_0}^{x_0'}\Big(\big(\partial_x \phi_0 - \sqrt{2U(\phi_0)}\big)^2 + 2\phi_0'\sqrt{2U(\phi_0)}\Big)\ud x \\
&= \int_{\phi_0(x_0)}^{\phi_0(x_0')}\sqrt{2U(y)}\ud y + \frac 12\int_{x_0}^{x_0'}\big(\partial_x \phi_0 - \sqrt{2U(\phi_0)}\big)^2\ud x.
\end{aligned}
\end{equation}
Analogously, if $\phi_0(x_0) \geq \phi_0(x_0')$, then
\begin{equation}
\label{eq:bogom-2}
E_p(\phi_0; x_0, x_0') = \int_{\phi_0(x_0')}^{\phi_0(x_0)}\sqrt{2U(y)}\ud y + \frac 12\int_{x_0}^{x_0'}\big(\partial_x \phi_0 + \sqrt{2U(\phi_0)}\big)^2\ud x.
\end{equation}
Hence, restrictions of kinks and antikinks to (bounded or unbounded) intervals are minimisers of the potential
energy among all the functions connecting two given values in $(-1, 1)$.

\subsection{Interaction of the kinks}
Our next goal is to compute the potential energy and the interaction forces
of a given chain of transitions between vacua $1$ and $-1$.
Since we consider only two vacua, such a chain is composed of alternating kinks and antikinks.
Without loss of generality, we assume that the leftmost transition is an antikink.

Let $\vec a := (a_1, \ldots, a_n) \in \bR^n$ be the positions of the transitions (following \cite[Chapter 5]{MS}, we use the letter
$a$ for the translation parameter; it should not be confused
with ``acceleration'' which will be given no symbol in this paper).
We always assume $a_1 \leq a_2 \leq \ldots \leq a_n$.
It will be convenient to abbreviate $H_k(x) := H(x - a_k)$ for $k \in \{1, \ldots, n\}$. We also denote
\begin{align}
\label{eq:Ha-def}
H(\vec a; x) := 1 + \sum_{k=1}^{n} (-1)^k(H_k(x) + 1).
\end{align}
For example, if $n = 0$, then $\vec a$ is an empty vector and $H(\vec a)$ is the vacuum $1$.
If $n = 1$ and $\vec a = (a_1)$, then $H(\vec a) = -H(\cdot - a_1)$ is an antikink.
If $n = 2$ and $\vec a = (a_1, a_2)$ with $a_2 - a_1 \gg 1$, then
$H(\vec a) = 1 - H(\cdot - a_1) + H(\cdot + a_2)$ has the shape of an antikink near $x = a_1$,
and of a kink near $x = a_2$. These are the \emph{kink-antikink pairs}, which we studied with Kowalczyk in \cite{JKL1}.

We define $\Phi \in C^\infty(\bR^n)$ by
\begin{equation}
\Phi(w_1, \ldots, w_n) := U'\Big( 1 + \sum_{k=1}^n(-1)^k(w_k + 1) \Big) - \sum_{k=1}^n (-1)^k U'(w_k),
\end{equation}
so that
\begin{equation}
\label{eq:Phi-def}
\begin{aligned}
\vD E_p(H(\vec a)) &= -\partial_x^2 H(\vec a) + U'(H(\vec a))
= {-}\sum_{k=1}^n (-1)^k U'(H_k(x)) + U'(H(\vec a)) \\
&= \Phi(H_1, \ldots, H_n).
\end{aligned}
\end{equation}
If we treat $\vec a = (a_1, \ldots, a_n)$ as point masses,
then the force acting on $a_k$ should be given by
\begin{equation}
\label{eq:Fk-def}
F_k(\vec a) := -\partial_{a_k} E_p(H(\vec a)) = (-1)^k\la \partial_x H_k, \Phi(H_1, \ldots, H_n)\ra.
\end{equation}
Observe that the translation invariance of $E_p$ implies $\sum_{k=1}^n F_k(\vec a) = 0$, as one should expect in view of Newton's third law.

Before we begin the computation of $F_k(\vec a)$ and $E_p(H(\vec a))$,
we state the Lemma 2.5 from \cite{JKL1}, which will be frequently used below.
\begin{lemma}
\label{lem:exp-cross-term}
For any $a_1 < a_2$ and $\alpha_1, \alpha_2 > 0$ with $\alpha_1 \neq \alpha_2$ the following bound holds:
\begin{equation}
\int_{-\infty}^\infty\eee^{-\alpha_1(x - a_1)_+}\eee^{-\alpha_2(a_2 - x)_+}\ud x \lesssim_{\alpha_1, \alpha_2} \eee^{-\min(\alpha_1, \alpha_2)(a_2 - a_1)}.
\end{equation}
For any $\alpha > 0$, the following bound holds:
\begin{equation}
\int_{-\infty}^\infty\eee^{-\alpha(x - a_1)_+}\eee^{-\alpha(a_2 - x)_+}\ud x \lesssim_{\alpha} (1 + a_2 - a_1)\eee^{-\alpha(a_2 - a_1)}.
\end{equation}\vspace{-1cm}\qedno
\end{lemma}
\vspace{0.5cm}
\begin{lemma}
There exists $C$ such that for all $\vec w \in [-2, 2]^n$ and $k \in \{1, \ldots, n\}$
\begin{equation}
\label{eq:Phi-taylor}
\begin{aligned}
&\big|\Phi(\vec w) + (-1)^{k}(U''(w_k) - 1)((1+w_{k+1}) - (1 - w_{k-1}))\big| \leq \\
&\qquad C\big(\max_{j < k-1}|1-w_j| + \max_{j > k+1}|1+w_j| + \\
&\qquad\qquad+ |1-w_k||1+w_{k+1}|^2 +
|1+w_k||1-w_{k-1}|^2 + |1+w_{k+1}||1-w_{k-1}|\big),
\end{aligned}
\end{equation}
where by convention $w_0 := 1$, $w_{n+1} := -1$ and $\max_{j < 0}|1-w_j| = \max_{j > n+1}|1+w_j|= 0$.
\end{lemma}
\begin{proof}
Let $\vec v \in \bR^n$ be given by $v_j := 1$ for $j < k$, $v_j := -1$ for $j > k$ and $v_k := w_k$.
The Taylor formula yields
\begin{equation}
\label{eq:Phi-taylor-1}
\begin{aligned}
\Phi(\vec w) &= \Phi(\vec v) + (\vec w - \vec v)\cdot\grad \Phi(\vec v) \\
&+ \sum_{i, j = 1}^n (w_i - v_i)(w_j - v_j)\int_0^1 (1-t)\partial_{i}\partial_{j}\Phi((1-t)\vec v + t\vec w)\ud t.
\end{aligned}
\end{equation}
We compute and estimate all the terms, calling a quantity ``negligible''
if its absolute value is smaller than the right hand side of \eqref{eq:Phi-taylor}.
For all $i, j \in \{1, \ldots, n\}$ and $\vec u \in \bR^n$, we have
\begin{align}
\partial_{j}\Phi(\vec u) &= (-1)^j\Big(U''\Big(1 + \sum_{\ell=1}^n (-1)^\ell(u_\ell + 1)\Big) - U''(u_j)\Big), \\
\partial_{i}\partial_{j}\Phi(\vec u) &= (-1)^{i+j}U'''\Big(1 + \sum_{\ell=1}^n (-1)^\ell(u_\ell + 1)\Big), \qquad\text{if }i \neq j, \\
\partial_{j}^2 \Phi(\vec u) &= U'''\Big(1 + \sum_{\ell=1}^n (-1)^\ell(u_\ell + 1)\Big) - (-1)^jU'''(u_j).
\label{eq:dj2Phi}
\end{align}
Observe that
\begin{equation}
1 + \sum_{\ell=1}^n (-1)^\ell (v_\ell + 1) = 1 + 2\sum_{\ell=1}^{k-1}(-1)^\ell + (-1)^k(w_k + 1) = (-1)^k w_k,
\end{equation}
in particular, since $U''$ is even and $U''(1) = 1$, for all $j \neq k$ we have $\partial_{j}\Phi(\vec v) = (-1)^j(U''(w_k) - 1)$.
We thus obtain
\begin{equation}
\begin{aligned}
(\vec w - \vec v)\cdot\grad \Phi(\vec v) &\sim (w_{k-1}-1)\partial_{{k-1}}\Phi(\vec v) + (w_{k+1}+1)\partial_{{k+1}}\Phi(\vec v) \\
&= {-}(-1)^k(U''(w_k) - 1)((1+w_{k+1}) - (1 - w_{k-1})).
\end{aligned}
\end{equation}

Remains the second line of \eqref{eq:Phi-taylor-1}. The terms with $i \notin \{k-1, k+1\}$
or $j \notin \{k-1, k+1\}$ are clearly negligible,
as are the terms $(i, j) \in \{(k-1, k+1), (k+1, k-1)\}$,
hence it suffices to consider the cases $(i, j) \in \{(k-1, k-1), (k+1, k+1)\}$.
Since the latter is analogous to the former,
we only consider $i = j = k-1$,
in particular we assume $k \geq 2$.

Using \eqref{eq:dj2Phi}, the fact that $U'''$ is locally Lipschitz, and writing
\begin{equation}
1 + \sum_{\ell=1}^n(-1)^\ell(u_\ell + 1)
= \sum_{\ell=1}^{k-2}(-1)^\ell(u_\ell-1) + (-1)^{k-1}u_{k-1}
+ \sum_{\ell = k}^{n}(-1)^\ell(u_\ell + 1),
\end{equation}
we obtain
\begin{equation}
|\partial_{k-1}^2 \Phi(\vec u)| \lesssim \sum_{\ell = 1}^{k-2}|1-u_\ell| + \sum_{\ell = k}^n |1 + u_\ell|.
\end{equation}
Setting $\vec u := (1-t)\vec v + t\vec w$
and multiplying by $(w_{k-1} - v_{k-1})^2 = (1-w_{k-1})^2$,
we obtain a negligible term as claimed.
\end{proof}

We denote
\begin{equation}
\label{eq:y-first-def}
\begin{gathered}
\vec y = (y_1, \ldots, y_{n-1}), \quad y_k := a_{k+1} - a_k, \\
y_0 := +\infty, \quad y_n := +\infty,\quad y_{\min} := \min_{1 \leq k \leq n-1}y_k.
\end{gathered}
\end{equation}
\begin{lemma}
\label{lem:Fz}
There exists $C$ such that for every increasing $n$-tuple $\vec a$ and $k \in \{1, \ldots, n\}$
\begin{align}
\label{eq:Fz}
\big|F_k(\vec a) - 2\kappa^2 (\eee^{-y_k} - \eee^{-y_{k-1}})\big| \leq Cy_{\min}\eee^{-2 y_{\min}}, 
\end{align}
where $\kappa$ is defined in Proposition~\ref{prop:prop-H}.
\end{lemma}
\begin{proof}
By Proposition~\ref{prop:prop-H} and Lemma~\ref{lem:exp-cross-term}, we have
\begin{equation}
\label{eq:F-cross-est}
\begin{aligned}
\int_{-\infty}^{\infty} &|\partial_x H_k|\big(\max_{j < k-1}|1-H_j| + \max_{j > k+1}|1+H_j| + |1-H_k||1+H_{k+1}|^2 +\\
&\qquad +|1+H_k||1-H_{k-1}|^2 + |1+H_{k+1}||1-H_{k-1}|\big)\ud x \lesssim y_{\min}\eee^{-2y_{\min}}.
\end{aligned}
\end{equation}
Thus, \eqref{eq:Fk-def} and \eqref{eq:Phi-taylor} yield
\begin{equation}
\label{eq:F-cross-est-1}
\Big|F_k(\vec a) + \int_{-\infty}^{\infty}\partial_x H_k(U''(H_k) - 1)((1 + H_{k+1}) - (1 - H_{k-1}))\ud x \Big| \lesssim y_{\min}\eee^{-2y_{\min}}.
\end{equation}
Applying again Proposition~\ref{prop:prop-H}, we have
\begin{equation}
|1 + H_{k+1} - \kappa \eee^{x - a_{k+1}}| \lesssim
\begin{cases}
\eee^{-2(a_{k+1} - x)}\qquad&\text{if }x \leq a_{k+1} \\
\eee^{x - a_{k+1}}\qquad&\text{if }x \geq a_{k+1},
\end{cases}
\end{equation}
thus, taking into account that $|\partial_x H_k(x)| + |U''(H_k(x)) - 1| \lesssim \eee^{-|x - a_k|}$,
\begin{equation}
\begin{aligned}
&\bigg|\int_{-\infty}^\infty \partial_x H_k(U''(H_k) - 1)(1+H_{k+1})\ud x - 
\kappa\int_{-\infty}^\infty \partial_x H_k(U''(H_k) - 1)\eee^{x - a_{k+1}}\ud x\bigg| \lesssim \\
&\qquad\lesssim \int_{-\infty}^{a_{k+1}} \eee^{-2|x - a_k|}\eee^{-2(a_{k+1} - x)}\ud x + \int_{a_{k+1}}^\infty \eee^{-2(x-a_k)}\eee^{x - a_{k+1}}\ud x.
\end{aligned}
\end{equation}
The first integral on the right hand side is $\lesssim y_{\min}\eee^{-2y_{\min}}$ by Lemma~\ref{lem:exp-cross-term},
and the second equals $\eee^{-2(a_{k+1} - a_k)}$.
Hence, \eqref{eq:reduced-force} yields
\begin{equation}
\bigg|\int_{-\infty}^\infty \partial_x H_k(U''(H_k) - 1)(1+H_{k+1})\ud x + 2\kappa^2 \eee^{-(a_{k+1} - a_k)}\bigg| \lesssim y_{\min}\eee^{-2y_{\min}}.
\end{equation}
Similarly, 
\begin{equation}
\bigg|\int_{-\infty}^\infty \partial_x H_k(U''(H_k) - 1)(1-H_{k-1})\ud x + 2\kappa^2 \eee^{-(a_{k} - a_{k-1})}\bigg| \lesssim y_{\min}\eee^{-2y_{\min}}.
\end{equation}
These two bounds, together with \eqref{eq:F-cross-est-1}, yield \eqref{eq:Fz}.
\end{proof}
\begin{remark}
We see from Lemma~\ref{lem:Fz} that the interaction between consecutive kink and antikink is attractive.
\end{remark}

\begin{lemma}
\label{lem:U-pyth}
There exists $C$ such that for all $\vec w \in [-2, 2]^n$
\begin{gather}
\label{eq:U-pyth}
\Big|U\Big( 1 + \sum_{k=1}^n(-1)^k(w_k + 1) \Big) - \sum_{k=1}^n U(w_k) \Big| \leq C \max_{1 \leq i < j \leq n}|1 - w_i||1 + w_j|, \\
\label{eq:U'-pyth}
\Big|U'\Big( 1 + \sum_{k=1}^n(-1)^k(w_k + 1) \Big) - \sum_{k=1}^n (-1)^k U'(w_k) \Big| \leq C \max_{1 \leq i < j \leq n}|1 - w_i||1 + w_j|.
\end{gather}
\end{lemma}
\begin{proof}
We proceed by induction with respect to $n$. For $n = 1$, both sides of both estimates equal $0$.

Let $n > 1$ and set
\begin{equation}
v := 1 + \sum_{k=1}^{n-1}(-1)^k(w_k + 1) = \sum_{k=1}^{n-1}(w_k - 1) - (-1)^n.
\end{equation}
Integrating the bound $|U'(w + (-1)^n(w_n + 1)) - U'(w)| \lesssim |w_n +1|$
for $w$ between $-(-1)^n$ and $v$, and using $U({-}(-1)^n) = 0$ as well as $U((-1)^n w_n) = U(w_n)$, we get
\begin{equation}
\begin{aligned}
\label{eq:U-pyth-step}
|U(v + (-1)^n(w_n + 1)) - U(w_n) - U(v)| &\leq C|(-1)^n+v||1+w_n| \\
&\leq C|1+w_n|\sum_{k=1}^{n-1}|1-w_k|,
\end{aligned}
\end{equation}
which finishes the induction step for \eqref{eq:U-pyth}.

Integrating the bound $|U''(w + (-1)^n(w_n + 1)) - U''(w)| \lesssim |w_n +1|$
for $w$ between $-(-1)^n$ and $v$, and using $U'({-}(-1)^n) = 0$ as well as $U'((-1)^n w_n) = (-1)^nU'(w_n)$, we get
\begin{equation}
\begin{aligned}
\label{eq:U-pyth-step-bis}
|U'(v + (-1)^n(w_n + 1)) - (-1)^nU'(w_n) - U'(v)| &\leq C|(-1)^n+v||1+w_n| \\
&\leq C|1+w_n|\sum_{k=1}^{n-1}|1-w_k|,
\end{aligned}
\end{equation}
which finishes the induction step for \eqref{eq:U'-pyth}. 
\end{proof}
\begin{lemma}
\label{lem:interactions}
There exists $C > 0$ such that for every increasing $n$-tupple $\vec a$
\begin{align}
\bigg|E_p(H(\vec a)) - nM + 2\kappa^2\sum_{k=1}^{n-1}\eee^{-y_k}\bigg| &\leq C y_{\min}\eee^{-2y_{\min}}.
\label{eq:EpHX-a}
\end{align}
\end{lemma}
\begin{proof}
From \eqref{eq:U-pyth} and Lemma~\ref{lem:exp-cross-term}, we have
\begin{equation}
\int_{-\infty}^\infty \Big| U(H(\vec a)) - \sum_{k=1}^n U(H_k)\Big|\ud x \lesssim y_{\min}\eee^{-y_{\min}}.
\end{equation}
Invoking again Lemma~\ref{lem:exp-cross-term}, we also have
\begin{equation}
\label{eq:EpHX-a-1}
\int_{-\infty}^\infty \Big| \frac 12(\partial_x H(\vec a))^2 - \frac 12\sum_{k=1}^n (\partial_x H_k)^2\Big|\ud x \leq \int_{-\infty}^\infty\sum_{1 \leq i < j \leq n}|\partial_x H_i||\partial_x H_j|\ud x \lesssim y_{\min}\eee^{-y_{\min}}.
\end{equation}
Combining these two bounds, we obtain
\begin{equation}
\big|E_p(H(\vec a)) - n M\big| \lesssim y_{\min}\eee^{-y_{\min}}.
\end{equation}
For all $s \geq 0$, set $\vec a(s) := (a_1 + s, a_2 + 2s, \ldots, a_n + ns)$.
Applying the last estimate with $\vec a(s)$ instead of $\vec a$, we get
\begin{equation}
\lim_{s \to \infty}E_p(H(\vec a(s))) = nM.
\end{equation}
By the Chain Rule, \eqref{eq:Fk-def} and \eqref{eq:Fz}, we have
\begin{equation}
\dd s E_p(H(\vec a(s))) = {-}\sum_{k=1}^n k F_k(\vec a(s)) = 2\kappa^2 \sum_{k=1}^{n-1} \eee^{-y_k-s} + O\big((y_{\min}+s)\eee^{-2(y_{\min}+s)}\big).
\end{equation}
An integration in $s$ yields \eqref{eq:EpHX-a}.
\end{proof}
\begin{lemma}
\label{lem:sizeDEp}
There exists $C > 0$ such that for every increasing $n$-tuple $\vec a$
\begin{align}
\|\vD E_p(H(\vec a))\|_{L^2} \leq C\sqrt{y_\tx{min}}\eee^{-y_\tx{min}},
\label{eq:sizeDEp} \\
\max_{1 \leq k \leq n} \big\|\partial_x H_k\big(U''(H(\vec a)) - U''(H_k)\big)\big\|_{L^2} \leq C\sqrt{y_\tx{min}}\eee^{-y_\tx{min}}.
\label{eq:size-diff-pot}
\end{align}
\end{lemma}
\begin{proof}
Since $\partial_x^2 H = U'(H)$, \eqref{eq:U'-pyth} yields
\begin{equation}
\Big| \partial_x^2 H(\vec a) - U'(H(\vec a))\Big| = \Big| \sum_{k=1}^n (-1)^kU(H_k) - U'(H(\vec a))\Big| \lesssim \max_{1 \leq i < j \leq n}
|1 - H_i||1 + H_j|.
\end{equation}
After taking the square, integrating over $\bR$ and applying Lemma~\ref{lem:exp-cross-term}, we obtain \eqref{eq:sizeDEp}.

In order to prove \eqref{eq:size-diff-pot}, we write
\begin{equation}
H(\vec a) = \sum_{\ell=1}^{k-1}(-1)^\ell(H_\ell-1) + (-1)^{k}H_k
+ \sum_{\ell = k+1}^{n}(-1)^\ell(H_\ell + 1).
\end{equation}
Since $U''$ is locally Lipschitz and $|\partial_x H| \lesssim \min(|1 - H|, |1+H|)$, we obtain
\begin{equation}
\big| \partial_x H_k \big(U''(H(\vec a)) - U''(H_k)\big)\big| \lesssim \max_{1 \leq \ell < k}
|1 - H_\ell||1 + H_k| + \max_{k < \ell \leq n}|1 - H_k||1 + H_\ell|,
\end{equation}
and we conclude as above.
\end{proof}

\subsection{Schr\"odinger operator with multiple potentials}
\label{ssec:schrod}
We define
\begin{equation}
L := \vD^2 E_p(H) = -\partial_x^2 + U''(H) = -\partial_x^2 + 1 + (U''(H) - 1).
\end{equation}
Differentiating $\partial_x^2 H(x - a) = U'(H(x-a))$ with respect to $a$ we obtain
\begin{equation}
\label{eq:Lc-ker}
\big({-}\partial_x^2 + U''(H(\cdot - a))\big)\partial_x H(\cdot - a) = 0,
\end{equation}
in particular for $a = 0$ we have $L(\partial_x H) = 0$.
Since $\partial_x H$ is a positive function, $0$ is a simple eigenvalue of $L$,
which leads to the following coercivity estimate
(see \cite[Lemma 2.3]{JKL1} for the exact same formulation, as well as \cite{HPW82} for a similar result).
\begin{lemma}
There exist $\nu, C > 0$ such that for all $h \in H^1(\bR)$ the following inequality holds:
\begin{equation}
\label{eq:vLv-coer}
\la h, Lh\ra \geq \nu \|h\|_{H^1}^2 -C \la \partial_x H, h\ra^2.
\end{equation}\vspace{-1cm}\qedno
\end{lemma}

For any increasing $n$-tuple $\vec a$,
we denote $L(\vec a) := \vD^2 E_p(H(\vec a))$,
which is the Schr\"odinger operator on $L^2(\bR)$ given by
\begin{equation}
\label{eq:schrod-op}
(L(\vec a)h)(x) := {-}\partial_x^2 h(x) + U''(H(\vec a; x))h(x).
\end{equation}
\begin{lemma}
\label{lem:D2H}
There exist $y_0, \nu, C > 0$ such that the following holds.
Let $\vec a \in \bR^n$ satisfy $a_{k+1} - a_k \geq y_0$ for all $k \in \{1, \ldots, n-1\}$
and let $h \in H^1(\bR)$. Then
\begin{equation}
\label{eq:D2H-coer}
\begin{aligned}
\la h, L(\vec a) h\ra \geq \nu \|h\|_{H^1}^2 -C\sum_{k=1}^n \la\partial_x H_k, h\ra^2.
\end{aligned}
\end{equation}
\end{lemma}
\begin{proof}
We adapt \cite[Proof of Lemma 2.4]{JKL1}.
We set
\begin{equation}
\begin{aligned}
\chi_1(x) &:= \chi\Big(\frac{x - a_1}{a_2 - a_1}\Big), \\
\chi_k(x) &:= \chi\Big(\frac{x - a_k}{a_{k+1} - a_k}\Big)
- \chi\Big(\frac{x - a_{k-1}}{a_k - a_{k-1}}\Big),\qquad\text{for }k \in \{2, \ldots, n-1\}, \\
\chi_n(x) &:= 1 - \chi\Big(\frac{x - a_{n-1}}{a_n - a_{n-1}}\Big)
\end{aligned}
\end{equation}
and we let
\begin{equation}
h_k := \chi_k h, \qquad\text{for }k \in \{1, \ldots, n\}.
\end{equation}
We have $h = \sum_{k=1}^n h_k$, hence
\begin{equation}
\la h, L(\vec a) h\ra = \sum_{k=1}^n \la h_k, L(\vec a) h_k\ra + 2\sum_{1 \leq i < j \leq n}\la h_i, L(\vec a) h_j\ra,
\end{equation}
so it suffices to prove that
\begin{gather}
\label{eq:coer-1}
\la h_k, L(\vec a) h_k\ra \geq \nu\|h_k\|_{H^1}^2 -  C\la\partial_x H_k, h_k\ra^2
-{o(1)\|h\|_{H^1}^2}, \\
\label{eq:coer-2}
\la h_i, L(\vec a) h_j\ra {\geq {-}o(1)\|h\|_{H^1}^2} \qquad \text{whenever }i \neq j, \\
\label{eq:coer-3}
\big|\la\partial_x H_k, h_k\ra^2 - \la \partial_x H_k, h\ra^2\big| {\leq o(1)\|h\|_{H^1}^2},
\end{gather}
{where $\nu>0$ is the constant in \eqref{eq:vLv-coer} and $o(1)\to 0$ as $y_0\to \infty$.}

We first prove \eqref{eq:coer-1}.
Without loss of generality we can assume $a_k = 0$. We then have
\begin{equation}
L(\vec a) = L + V, \qquad V := U''(H(\vec a)) - U''(H),
\end{equation}
thus
\begin{equation}
\la h_k, L(\vec a) h_k\ra = \la h_k, L h_k\ra + \la h_k, Vh_k\ra \geq \nu \|h_k\|_{H^1}^2 - C\la \partial_x H, h_k\ra^2
+ \int_{-\infty}^\infty \chi_k^2 V h_k^2\ud x.
\end{equation}
We only need to check that $\|\chi_k^2 V\|_{L^\infty} \ll 1$. If $x \geq \frac 23 a_{k+1}$
or $x \leq \frac 23 a_{k-1}$, then $\chi_k(x) = 0$.
If $\frac 23 a_{k-1} \leq x \leq \frac 23 a_{k+1}$, then $|1 - H_j(x)| \ll 1$
for all $j < k$ and $|1 + H_j(x)| \ll 1$ for all $j > k$, hence
\begin{equation}
\big|H(\vec a; x) - (-1)^k H_k(x)\big| = \Big|\sum_{j=1}^{k-1}(-1)^j(H_j(x) - 1) + \sum_{j=k+1}^n(-1)^j(H_j(x) + 1)\Big| \ll 1,
\end{equation}
which implies $|V(x)| \ll 1$.

Next, we show \eqref{eq:coer-2}.
Observe that
\begin{equation}
\label{eq:chiichij}
\chi_i\chi_jU''(H(\vec a; x)) \geq 0, \qquad \text{for all }x \in \bR.
\end{equation}
Indeed, if $j \neq i+1$, then $\chi_i(x) \chi_j(x) = 0$ for all $x$.
If $j = i+1$, then $\chi_i(x) \chi_j(x) \neq 0$ only if $\frac 23 a_i + \frac 13 a_{i+1} \leq x \leq \frac 13 a_i + \frac 23 a_{i+1}$, which implies $|1 - H_\ell(x)| \ll 1$ for all $\ell \leq i$
and $|1 + H_\ell(x)| \ll 1$ for all $\ell > i$, hence $|H(\vec a; x) - (-1)^i| \ll 1$,
in particular $U''(H(\vec a; x)) > 0$.
Using \eqref{eq:chiichij} and the fact that $\|\partial_x \chi_k\|_{L^\infty} \ll 1$
for all $k$, we obtain
\[
\begin{aligned}
\la h_i, L(\vec a) h_j\ra &= \int_{-\infty}^\infty \Big(\partial_x(\chi_i h)\partial_x(\chi_j h)
+U''(H(\vec a)) \chi_i\chi_j h^2\Big)\ud x\\
&\geq \int_{-\infty}^\infty\chi_i\chi_j (\partial_x h)^2\ud x - o(1)\|h\|_{H^1}^2\geq {-}o(1)\|h\|_{H^1}^2.
\end{aligned}
\]

Finally, we have
\begin{equation}
\begin{aligned}
\big|\la\partial_x H_k, h_k\ra^2 - \la \partial_x H_k, h\ra^2\big| &\leq (\|h_k\|_{L^2} + \|h\|_{L^2})
\big|\la\partial_x H_k, h_k - h\ra\big| \\ &\leq 2\|h\|_{L^2}^2\sum_{j \neq k}\|\chi_j\partial_x H_k\|_{L^2} \leq o(1)\|h\|_{L^2}^2,
\end{aligned}
\end{equation}
hence \eqref{eq:coer-3} follows.
\end{proof}

\section{Preliminaries on the Cauchy problem}
\label{sec:cauchy}
Since the local well-posedness theory for \eqref{eq:csf-2nd} presents no serious difficulties,
we only briefly resume it, leaving some details to the Reader.
\begin{definition}
If $I \subset \bR$ is an open interval, $\iota_-, \iota_+ \in \{-1, 1\}$ and $\bs\phi = (\phi, \dot \phi): I \times \bR \to \bR\times\bR$, then we say that $\bs\phi$ is a solution of \eqref{eq:csf} in the energy sector $\cE_{\iota_-, \iota_+}$ if $\bs \phi \in C(I; \cE_{\iota_-, \iota_+})$
and \eqref{eq:csf} holds in the sense of distributions.
\end{definition}

Choose any $(\xi, 0) \in \cE_{\iota_-, \iota_+}$ such that $\xi \in C^\infty$ and $\partial_x \xi$ is of compact support,
and write $(\phi, \dot\phi) = (\xi + \psi, \dot\psi)$. Then $\bs \phi \in C(I; \cE_{\iota_-, \iota_+})$ is equivalent to $\bs\psi \in C(I; \cE)$
and \eqref{eq:csf} becomes
\begin{equation}
\label{eq:cauchy-psi}
\dd t\begin{pmatrix} \psi \\ \dot\psi \end{pmatrix} = \begin{pmatrix}\dot \psi \\ \partial_x^2 \psi - \psi - \big(U'(\xi + \psi) - U'(\xi) - \psi\big)
+ \big(\partial_x^2 \xi - U'(\xi)\big)\end{pmatrix},
\end{equation}
which is the linear Klein-Gordon equation for $\psi$ with the forcing term ${-}\big(U'(\xi + \psi) - U'(\xi) - \psi\big) + \big(\partial_x^2 \xi - U'(\xi)\big)
\in C(I; \cE)$. By the uniqueness of weak solutions of linear wave equations, see \cite{Friedrichs54}, we have that a weak solution of \eqref{eq:cauchy-psi}
is in fact a strong solution given by the Duhamel formula.

If $\bs\phi_{1, 0}, \bs \phi_{2, 0}, \ldots$ and $\bs\phi_0$
are states of bounded energy, then we write
$\bs \phi_{m, 0} \to \bs\phi_0$ if $\bs \phi_{m, 0} - \bs\phi_0 \to 0$ in $\cE$, and
$\bs \phi_{m, 0} \wto \bs\phi_0$ if $\dot \phi_{m, 0} \wto \dot \phi_0$
in $L^2(\bR)$ and $\phi_{m, 0} - \phi_0 \to 0$ in $L^\infty_\tx{loc}(\bR)$. Observe that if $\bs\phi_{m, 0} - \bs \phi_0$ is bounded in $\cE$,
then $\bs \phi_{m, 0} \wto \bs\phi_0$ is equivalent to the weak convergence to $0$ in the Hilbert space $\cE$ of the sequence $\bs \phi_{m, 0} - \bs \phi_0$. However, our notion of weak convergence
is more general, in particular the topological class is not necessarily
preserved under weak limits as defined above.
A similar notion of weak convergence was used in Jia and Kenig \cite{JK}.
\begin{proposition}
\label{prop:cauchy}
\begin{enumerate}[(i)]
\item \label{it:cauchy-en}
For all $\iota_-, \iota_+ \in \{-1, 1\}$ and $\bs\phi_0 \in \cE_{\iota_-, \iota_+}$ and $t_0 \in \bR$,
there exists a unique solution $\bs\phi: \bR \to \cE_{\iota_-, \iota_+}$ of \eqref{eq:csf} such that $\bs\phi(t_0) = \bs\phi_0$.
The energy $E(\bs \phi(t))$ does not depend on $t$.
\item \label{it:cauchy-H2}
If $\partial_x \phi_0 \in H^1(\bR)$ and $\dot\phi_0 \in H^1(\bR)$, then $\partial_x \phi \in C(\bR; H^1(\bR)) \cap C^1(\bR; L^2(\bR))$, $\dot\phi \in C(\bR; H^1(\bR)) \cap C^1(\bR; L^2(\bR))$
and \eqref{eq:csf-1st} holds in the strong sense in $H^1(\bR) \times L^2(\bR)$.
\item \label{it:cauchy-strong}
Let $\bs \phi_{1,0}, \bs \phi_{2,0}, \ldots \in \cE_{\iota_-, \iota_+}$ and $\bs\phi_{m, 0} \to \bs\phi_0 \in \cE_{\iota_-, \iota_+}$.
If $(\bs \phi_m)_{m=1}^\infty$ and $\bs\phi$ are the solutions of \eqref{eq:csf} such that $\bs \phi_m(t_0) = \bs\phi_{m,0}$ and $\bs\phi(t_0) = \bs\phi_0$,
then $\bs \phi_m(t) \to \bs \phi(t)$ for all $t \in \bR$, the convergence being uniform on every bounded time interval.
\item \label{it:cauchy-speed}
If $\wt{\bs\phi}_0 \vert_{[x_1, x_2]} = {\bs\phi}_0 \vert_{[x_1, x_2]}$, then $\wt{\bs\phi}(t) \vert_{[x_1 + |t - t_0|, x_2 - |t - t_0|]} = {\bs\phi}(t) \vert_{[x_1 + |t - t_0|, x_2 - |t - t_0|]}$ for all $t \in \big[t_0 - \frac 12(x_2 - x_1), t_0 + \frac 12(x_2 - x_1)\big]$.
\item \label{it:cauchy-weak}
Let $\bs \phi_{1,0}, \bs \phi_{2,0}, \ldots \in \cE_{\iota_-, \iota_+}$ and $\bs\phi_{m, 0} \wto \bs\phi_0 \in \cE_{\iota_-, \iota_+}$.
If $(\bs \phi_m)_{m=1}^\infty$ and $\bs\phi$ are the solutions of \eqref{eq:csf} such that $\bs \phi_m(t_0) = \bs\phi_{m,0}$ and $\bs\phi(t_0) = \bs\phi_0$,
then $\bs \phi_m(t) \wto \bs \phi(t)$ for all $t \in \bR$.
\end{enumerate}
\end{proposition}
\begin{proof}
Statements \ref{it:cauchy-en}, \ref{it:cauchy-H2} and \ref{it:cauchy-strong} follow from energy estimates and Picard iteration, see for example \cite{GiVe85} or \cite[Section X.13]{ReedSimon}
for similar results.
Since the linear Klein-Gordon equation has propagation speed equal to 1, each Picard iteration satisfies the finite propagation speed property, and \ref{it:cauchy-speed} follows by passing to the limit.
We skip the details.

We sketch a proof of \ref{it:cauchy-weak}.

In the first step, we argue that it can be assumed without loss of generality that $\|\bs \phi_{m, 0} - \bs \phi_0\|_\cE$ is bounded and $\|\phi_{m, 0} - \phi_0\|_{L^\infty} \to 0$.
To this end, let $x_m \to -\infty$ and $x_m' \to \infty$ be such that
\begin{equation}
\lim_{m\to\infty}\sup_{x \in [x_m, x_m']}|\phi_{m, 0}(x) - \phi_0(x)| = 0.
\end{equation}
We define a new sequence $\wt \phi_{m, 0}: \bR \to \bR$
by the formula
\begin{equation}
\label{eq:weak-conv-cutoff}
\wt\phi_{m, 0}(x) := \begin{cases}
\iota_- & \text{for all }x \leq x_m - 1, \\
(x_m-x)\iota_- + (1-x_m+x)\phi_{m, 0}(x_m)& \text{for all }x \in [x_m - 1, x_m], \\
\phi_{m, 0}(x) &\text{for all }x \in [x_m, x_m'], \\
(x - x_m')\iota_+ + (1-x+x_m')\phi_{m, 0}(x_m') & \text{for all }x \in [x_m', x_m' + 1], \\
\iota_+ & \text{for all }x \geq x_m' + 1.
\end{cases}
\end{equation}
Taking into account that
\begin{equation}
\lim_{m\to\infty}\sup_{x \leq x_m}\big(|\phi_0(x) - \iota_-|
+\sup_{x \geq x_m'}|\phi_0(x) - \iota_+|\big) = 0,
\end{equation}
we have that $\|\wt\phi_{m, 0} - \phi_0\|_{L^\infty} \to 0$.
In particular, if we take $R \gg 1$, then
\begin{equation}
\limsup_{m\to\infty}\sup_{x \leq -R}|\wt\phi_{m, 0}(x) - \iota_-| \ll 1,
\end{equation}
thus
\begin{equation}
\limsup_{m\to\infty}\int_{-\infty}^{-R}|\wt\phi_{m, 0}(x) - \iota_-|^2\ud x \lesssim \limsup_{m\to\infty}\int_{-\infty}^{-R}U(\wt\phi_{m, 0}(x))\ud x < \infty,
\end{equation}
and similarly for $x \geq R$. Hence, $\limsup_{m\to\infty}\|\wt \phi_{m, 0} - \phi_0\|_{H^1} < \infty$.

Let $\wt{\bs\phi}_m$ be the solution of \eqref{eq:csf} such that
$\wt{\bs \phi}_m(t_0) = (\wt \phi_{m, 0}, \dot \phi_{m, 0})$.
By property \ref{it:cauchy-speed}, it suffices to prove that
$\wt{\bs\phi}_m(t) \wto \bs\phi(t)$ for all $t$,
which finishes the first step. In the sequel, we write $\phi_{m, 0}$
instead of $\wt\phi_{m, 0}$.

We choose $\xi$ as in \eqref{eq:cauchy-psi} and write $\bs \phi_m = (\xi, 0) + \bs\psi_m$, $\bs\phi = (\xi, 0) + \bs\psi$.
It suffices to prove that, for every $t\in \bR$, any subsequence of $\bs\psi_m(t)$ has a subsequence weakly converging to $\bs\psi(t)$.

Let $\bs h_m$ be the solution of \eqref{eq:free-kg} with initial data $\bs h_m(t_0) = \bs\psi_m(t_0)$,
$\bs h$ the solution of the same problem with initial data $\bs h(t_0) = \bs\psi(t_0)$,
and $\bs g_m$ the solution of
\begin{equation}
\label{eq:cauchy-gm}
\dd t\begin{pmatrix} g_m \\ \dot g_m \end{pmatrix} = \begin{pmatrix}\dot g_m \\ \partial_x^2 g_m - g_m - \big(U'(\xi + \psi_m) - U'(\xi) - \psi_m\big)
+ \big(\partial_x^2 \xi - U'(\xi)\big)\end{pmatrix},
\end{equation}
with the initial data $\bs g_m(t_0) = 0$.
We thus have $\bs\psi_m = \bs h_m + \bs g_m$ for all $m$.

By the continuity of the free Klein-Gordon flow, we have $\bs h_m(t) \wto \bs h(t)$ for all $t \in \bR$.
By the energy estimates, $\bs g_m$ is bounded in $C^1(I; \cE)$ for any bounded open interval $I$.
Since the weak topology on bounded balls of $\cE$ is metrizable, the Arzel\`a-Ascoli theorem yields $\bs g \in C(\bR; \cE)$
and a subsequence of $(\bs g_m)_m$, which we still denote $(\bs g_m)_m$, such that $\bs g_m(t) \wto \bs g(t)$ for every $t \in \bR$.
Let $\wt{\bs \psi} := \bs h + \bs g$, so that $\bs\psi_m(t) \wto \wt{\bs\psi}(t)$ for all $t \in \bR$.
In particular, $\psi_m(t, x) \to \wt\psi(t, x)$ for all $(t, x)$, hence by the dominated convergence theorem
\begin{equation}
U'(\xi + \psi_m) - U'(\xi) - \psi_m \to U'(\xi + \wt\psi) - U'(\xi) - \wt\psi
\end{equation}
in the sense of distributions. We can thus pass to the distributional limit in \eqref{eq:cauchy-psi} and conclude that $\wt{\bs \psi}$
is a solution of \eqref{eq:csf} with initial data $\wt{\bs \psi}(t_0) = \bs\psi(t_0)$.
By the uniqueness of weak solutions, $\wt{\bs\psi} = \bs\psi$.
\end{proof}
Finally, we have local stability of vacuum solutions.
\begin{lemma}
\label{lem:loc-vac-stab}
There exist $\eta_0, C_0 > 0$ having the following property.
If $\bs\iota \in \{\bs 1, -\bs 1\}$, $t_0 \in \bR$, $-\infty \leq x_1 + 1 < x_2 - 1 \leq \infty$, $E(\bs\phi_0) < 0$, $\|\bs \phi_0 - \bs\iota\|_{\cE(x_1, x_2)} \leq \eta_0$ and $\bs\phi$ is a solution of \eqref{eq:csf} such that $\bs\phi(0) = \bs\phi_0$, then for all $t \in \big[t_0 - \frac 12(x_2 - x_1)+1, t_0 + \frac 12(x_2 - x_1)-1\big]$
\begin{equation}
\label{eq:loc-vac-stab}
\|\bs \phi(t) - \bs \iota\|_{\cE(x_1 + |t - t_0|, x_2 - |t - t_0|)} \leq C_0 \|\bs \phi_0 - \bs\iota\|_{\cE(x_1, x_2)}.
\end{equation}
\end{lemma}
\begin{proof}[Sketch of a proof]
To fix ideas, assume $\iota = 1$ and $t > t_0 = 0$.
The positivity of $U$ and Green's formula in space-time imply that the function
\begin{equation}
\big[0, (x_2 - x_1)/2 - 1\big] \owns t \mapsto \int_{x_1 + t}^{x_2 - t}\Big( \frac 12(\dot \phi(t))^2 + \frac 12 (\partial_x \phi(t))^2 + U(\phi(t))\Big)\ud x
\end{equation}
is decreasing.
Set $\bs\psi(t) := \bs\phi(t) - \bs 1$.
For $|\psi|$ small, we have $U(1 + \psi) \simeq \psi^2$.
By a continuity argument and the fact that $\|\psi\|_{L^\infty(I)}\lesssim \|\psi\|_{H^1(I)}$
on any interval $I$ of length $\geq 2$, we obtain
$\|\psi(t)\|_{L^\infty(x_1 + t, x_2 - t)} \lesssim \eta_0$ and \eqref{eq:loc-vac-stab}.
\end{proof}

\section{Estimates on the modulation parameters}
\label{sec:mod}
Our present goal is to reduce the motion of a kink cluster to a system of ordinary
differential equations with sufficiently small error terms.
\subsection{Basic modulation}
\label{ssec:basic-mod}
We consider a solution of \eqref{eq:csf} which is close, on some open time interval $I$, to multi-kink configurations.
It is natural to express the solution as the sum of a multi-kink and a small error,
which can be done in multiple ways. A unique choice of such a decomposition is obtained
by imposing specific \emph{orthogonality conditions}.

If $\vec a \in C^1(I; \bR^n)$ and, using the notation \eqref{eq:Ha-def}, we decompose
\begin{equation}
\label{eq:g-def}
\bs \phi(t) = \bs H(\vec a(t)) + \bs g(t),
\end{equation}
then the Chain Rule implies that $\bs \phi$ solves \eqref{eq:csf} if and only if $\bs g$ solves
\begin{equation}
\label{eq:ls-g-eq}
\begin{aligned}
\partial_t \bs g(t) &= \bs J\vD E\big(\bs H(\vec a(t)) + \bs g(t)\big)
- {\vec a \,}'(t)\cdot \partial_{\vec a}\bs H(\vec a(t)).
\end{aligned}
\end{equation}
Similarly as in the previous section, we write $H_k(t, x) := H(x - a_k(t))$.
We impose the orthogonality conditions \eqref{eq:g-orth-intro}, which we rewrite as
\begin{equation}
\label{eq:g-orth}
\la \partial_x H_k(t), g(t)\ra = 0, \qquad\text{for all }k \in \{1, \ldots, n\}\text{ and }t \in I.
\end{equation}

\begin{definition}[Weight of modulation parameters]
For all $\vec a \in \bR^n$, we set
\begin{equation}
\label{eq:rhoa-def}
\rho(\vec a) := \sum_{k=1}^{n-1} \eee^{-(a_{k+1} - a_k)}.
\end{equation}
For $\vec a: I \to \bR^n$, we define $\rho : I \to (0, \infty)$ by
\begin{equation}
\label{eq:rhot-def}
\rho(t) := \rho(\vec a(t)) = \sum_{k=1}^{n-1} \eee^{-(a_{k+1}(t) - a_k(t))}.
\end{equation}\end{definition}
We recall that for all $\bs\phi_0 \in \cE_{1, (-1)^n}$ we set
\begin{equation}\label{eq:d-def-2}
\delta(\bs\phi_0) := \inf_{\vec a \in \bR^n}\big(\|\bs \phi_0 - \bs H(\vec a)\|_\cE^2 + \rho(\vec a)\big),
\end{equation}
see Definition~\ref{def:d-def}.
Since both $\bs \phi_0$ and $\bs H(\vec a)$ belong to $\cE_{1, (-1)^n}$,
it follows that $\bs \phi_0 - \bs H(\vec a) \in \cE$, so that $\delta(\bs \phi_0) < \infty$.

We have the following ``static'' modulation lemma.
\begin{lemma}
\label{lem:static-mod}
There exist $\eta_0, \eta_1, C_0 > 0$ having the following property.
For all $\bs \phi_0 \in \cE_{1, (-1)^n}$ such that $\delta(\bs \phi_0) < \eta_0$
there exists unique $\vec a = \vec a(\bs \phi_0) \in \bR^n$ such that
\begin{equation}
\label{eq:dist-bd-eta1}
\|\bs \phi_0 - \bs H(\vec a)\|_\cE^2 + \rho(\vec a) < \eta_1
\end{equation}
and
\begin{equation}
\label{eq:static-orth}
\la \partial_x H(\cdot - a_k), \phi_0 - H(\vec a)\ra = 0\qquad\text{for all }k \in \{1, \ldots, n\}.
\end{equation}
It satisifes
\begin{align}
\label{eq:dist-bd-C0}
\|\bs \phi_0 - \bs H(\vec a)\|_\cE^2 + \rho(\vec a) &\leq C_0 \delta(\bs \phi_0), \\
\label{eq:g-coer-stat}
\|\bs \phi_0 - \bs H(\vec a)\|_\cE^2 &\leq C_0(\rho(\vec a) + E(\bs \phi_0) - nM).
\end{align}
Moreover, the map $\cE_{1, (-1)^n} \owns \bs\phi_0 \mapsto \vec a(\bs \phi_0) \in \bR^n$ is of class $C^1$.
\end{lemma}
\begin{proof}
We first prove the existence of $\vec a$. Set $\eta := \delta(\bs\phi_0) \in (0, \eta_0)$.
By the definition of $\delta$, there exists $b \in \bR^n$ such that
\begin{equation}
\|\bs \phi_0 - \bs H(\vec b)\|_\cE^2 + \rho(\vec b) < 2\eta.
\end{equation}
We define $\vec\Gamma = (\Gamma_1, \ldots, \Gamma_n) \in C^1(\bR^n \times \cE(\bR); \bR^n)$ by
\begin{equation}
\begin{aligned}
\Gamma_k(\vec a, \bs h) := \la \partial_x H(\cdot - a_k), h + H(\vec b) - H(\vec a)\ra, \qquad k\in\{1, \ldots, n\}.
\end{aligned}
\end{equation}
We thus have
\begin{equation}
\label{eq:dajGk}
\partial_{a_j} \Gamma_k(\vec a, \bs h) =
\begin{cases}
- \la \partial_x^2 H(\cdot - a_k), h + H(\vec b) - H(\vec a)\ra + (-1)^k \|\partial_x H\|_{L^2}^2 &\quad\text{if }j = k, \\
(-1)^j\la \partial_x H(\cdot - a_j), \partial_x H(\cdot - a_k) \ra &\quad\text{if }j \neq k.
\end{cases}
\end{equation}
In particular, $\partial_{\vec a}\vec\Gamma(\vec b, \bs h)$ is invertible and has uniformly bounded
inverse if $\|\bs h\|_\cE$ and $|\vec a - \vec b|$ are small enough.

For $C_1 > 0$ to be chosen below, consider the map
\begin{equation}
\begin{gathered}
\vec\Phi: B_{\bR^n}(\vec b, C_1\sqrt{\eta_0}) \times B_{\cE}(0, 2\sqrt{\eta_0}) \to \bR^n, \\
\vec\Phi(\vec a, \bs h) := \vec a - \big[\partial_{\vec a}\vec \Gamma(\vec b, \bs h)\big]^{-1}\vec\Gamma(\vec a, \bs h).
\end{gathered}
\end{equation}
We have $|\vec\Phi(\vec b, \bs h)- \vec b| \lesssim \|\bs h\|_\cE < 2\sqrt \eta$,
hence we can choose $C_1$ so that
\begin{equation}
\label{eq:Phi-b}
|\vec \Phi(\vec b, \bs h)-\vec b| \leq \frac 13 C_1\sqrt{\eta} \leq \frac 13 C_1\sqrt{\eta_0}.
\end{equation}

The Fundamental Theorem of Calculus yields
\begin{equation}
\vec \Phi(\shvec{a}, \bs h) - \vec \Phi(\vec a, \bs h) = \bigg[ \tx{Id} - \big[\partial_{\vec a}\vec \Gamma(\vec b, \bs h)\big]^{-1}\int_0^1 \partial_{\vec a}\vec \Gamma((1-s)\vec a + s\shvec a, \bs h)\ud s\bigg](\shvec a - \vec a).
\end{equation}
We see from \eqref{eq:dajGk} that $\partial_{\vec a}\vec \Gamma$ is Lipschitz with respect to $\vec a$, hence $|\partial_{\vec a}\vec \Gamma((1-s)\vec a + s\shvec a, \bs h) - \partial_{\vec a}\vec \Gamma(\vec b, \bs h)| \lesssim C_1\sqrt{\eta_0}$ and
\begin{equation}
|\vec \Phi(\shvec{a}, \bs h) - \vec \Phi(\vec a, \bs h)| \lesssim C_1\sqrt{\eta_0}|\shvec a - \vec a|.
\end{equation}
If $\eta_0$ is small enough, we obtain
\begin{equation}
\label{eq:Phi-contr}
|\vec \Phi(\shvec{a}, \bs h) - \vec \Phi(\vec a, \bs h)| \leq \frac 13|\shvec a - \vec a|.
\end{equation}
Using \eqref{eq:Phi-b}, we see that $\Phi(\cdot, \bs h)$ is a strict contraction on
$B_{\bR^n}(\vec b, C_1\sqrt{\eta_0})$.

Let $\vec a = \vec a(\bs \phi_0)$ be the unique fixed point of $\vec \Phi(\cdot, \bs\phi_0 - \bs H(\vec b))$ in the ball $B_{\bR^n}(\vec b, C_1\sqrt{\eta_0})$.
From the definitions of $\vec\Phi$ and $\vec \Gamma$, we get \eqref{eq:static-orth}.
From \eqref{eq:Phi-b} and \eqref{eq:Phi-contr} we have $|\vec a - \vec b| \leq C_1 \sqrt{\eta}$,
hence for an appropriate choice of $C_0$
\begin{equation}
\|\bs\phi_0 - \bs H(\vec a)\|_\cE^2 \leq 2\big( \|\bs\phi_0 - \bs H(\vec b)\|_\cE^2+\|\bs H(\vec b) - \bs H(\vec a)\|_\cE^2 \big) \leq \frac 12 C_0\eta.
\end{equation}
Upon diminishing $\eta_0$, we also have $|a_k - b_k| \leq \frac 12\log 2$,
which implies $\rho(\vec a) \leq 2\rho(\vec b) \leq 4\eta$,
we thus get \eqref{eq:dist-bd-C0}.

We now prove \eqref{eq:g-coer-stat}. Set $\bs g := \bs \phi_0 - \bs H(\vec a)$.
We have the Taylor expansion
\begin{equation}
\begin{aligned}
E(\bs H(\vec a) + \bs g) &= E(\bs H(\vec a)) + \la \vD E(\bs H(\vec a)), \bs g\ra + \frac 12 \la \vD^2 E(\bs H(\vec a))\bs g, \bs g\ra \\
&+ \int_{-\infty}^\infty \Big(U(H(\vec a) + g) - U(H(\vec a)) - U'(H(\vec a))g - \frac 12 U''(H(\vec a))g^2\Big)\ud x.
\end{aligned}
\end{equation}
Lemma~\ref{lem:D2H} and \eqref{eq:g-orth} yield
\begin{equation}
\la \bs g, \vD^2 E(\bs H(\vec a))\bs g\ra  = \|\dot g\|_{L^2}^2 + \la g, \vD^2 E_p(H(\vec a)) g\ra \geq \|\dot g\|_{L^2}^2 + \nu\|g\|_{H^1}^2.
\end{equation}
The Sobolev embedding implies that the second line has absolute value $\lesssim \|g\|_{H^1}^3$, thus $ \leq \frac \nu 8 \|g\|_{H^1}^2$ if $\eta_0$ is small enough.
By Lemma~\ref{lem:sizeDEp}, we also have
\begin{equation}
|\la \vD E(\bs H(\vec a)), \bs g\ra| = |\la \vD E_p(H(\vec a)), g\ra| \lesssim \sqrt{y_{\min}}\eee^{-y_{\min}}\|g\|_{L^2},
\end{equation}
thus $|\la \vD E(\bs H(\vec a)), \bs g\ra| \leq \frac{\nu}{8}\|g\|_{H^1}^2 + Cy_{\min}\eee^{-2y_{\min}}$. 
Lemma~\ref{lem:interactions} implies
\begin{equation}
\label{eq:g-coer-pres}
\frac 12 \|\dot g\|_{L^2}^2 + \frac{\nu}{4}\|g\|_{H^1}^2 \leq 2\kappa^2 \sum_{k=1}^{n-1}\eee^{-y_k} + Cy_{\min}\eee^{-2y_{\min}} + E(\bs H(\vec a) + \bs g) - nM,
\end{equation}
in particular \eqref{eq:g-coer-stat}.

Continuous differentiability of $\bs \phi_0 \mapsto \vec a(\bs \phi_0)$ follows from $\vec \Phi \in C^1$, see \cite[Chapter 2, Theorem 2.2]{chow-hale}.

It remains to prove that, if $\eta_1$ is small enough,
then there is no $\vec a \notin B_{\bR^n}(\vec b, C_1\sqrt{\eta_0})$ satisfying \eqref{eq:dist-bd-eta1}. Suppose the contrary. Then there exist $\bs \phi_{0, m}$, $\vec b_m$ and
$\vec a_m$ such that
\begin{gather}
\|\bs \phi_{0, m} - \bs H(\vec b_m)\|_\cE^2 + \rho(\vec b_m) \leq 2\delta(\bs\phi_{0, m})\qquad\text{for all }m, \\
\lim_{m\to\infty} \big(\|\bs\phi_{0, m} - \bs H(\vec a_m)\|_\cE^2 + \rho(\vec a_m)\big) = 0, \\
|\vec a_m - \vec b_m| \geq C_1\sqrt{\eta_0} \qquad\text{for all }m.
\end{gather}
In particular, the first two conditions above yield
\begin{equation}
\lim_{m\to\infty} \big(\|\bs H(\vec a_m) - \bs H(\vec b_m)\|_\cE^2 + \rho(\vec a_m)+ \rho(\vec b_m)\big) = 0,
\end{equation}
which contradicts the third condition.
\end{proof}
\begin{lemma}
\label{lem:basic-mod}
There exist $\eta_0, C_0, C_1 > 0$ such that the following is true.
Let $I \subset \bR$ be an open interval and let $\bs \phi: I \to \cE_{1, (-1)^n}$ be a solution of \eqref{eq:csf} such that
\begin{equation}
\delta(\bs \phi(t)) \leq \eta_0, \qquad\text{for all }t \in I.
\end{equation}
Then there exist $\vec a \in C^1(I; \bR^n)$ and $\bs g \in C(I; \cE)$
satisfying \eqref{eq:g-def} and \eqref{eq:g-orth}.
In addition, 
for all $t \in I$ such that
\begin{equation}
\label{eq:delta-geq-ass}
\delta(\bs\phi(t)) \geq C_1(E(\bs \phi) - nM)
\end{equation}
the following bounds hold:
\begin{gather}
\label{eq:g-coer}
\|\bs g(t)\|_{\cE}^2 \leq C_0 \rho(t) \leq C_0^2 \delta(\bs\phi(t)), \\
\label{eq:ak'-est}
\big|Ma_k'(t) + (-1)^k \la \partial_x H_k(t), \dot g(t)\ra\big| \leq C_0\rho(t).
\end{gather}
Finally, if $\bs \phi_m$ is a sequence of solutions of \eqref{eq:csf} such that
\begin{equation}
\label{eq:phi-conv-en}
\lim_{m \to \infty} \|\bs \phi_m - \bs \phi\|_{C(I; \cE)} = 0,
\end{equation}
then the corresponding $\vec a_m$ and $\bs g_m$ satisfy
\begin{equation}
\label{eq:mod-basic-conv}
\lim_{m \to \infty}\big(\|\vec a_m - \vec a\|_{C^1(I; \bR^n)}
+ \|\bs g_m - \bs g\|_{C(I; \cE)}\big) = 0.
\end{equation}
\end{lemma}
\begin{remark}
In Sections~\ref{sec:n-body} and \ref{sec:any-position},
we will always have $E(\bs \phi) = nM$,
hence \eqref{eq:delta-geq-ass} will be automatically satisfied.
In Section~\ref{sec:profiles}, an additional argument will be necessary in order to ensure \eqref{eq:delta-geq-ass} on time intervals of interest, see the beginning of the proof of Theorem~\ref{thm:unstable}.
\end{remark}
\begin{proof}[Proof of Lemma~\ref{lem:basic-mod}]
\textbf{Step 1.} (Modulation equations.)
For every $t \in I$, let $\vec a(t) := \vec a(\bs\phi(t))$ be given by Lemma~\ref{lem:static-mod}, and let $\bs g(t)$ be given by \eqref{eq:g-def}.
The maps $t \mapsto \bs\phi(t) \in \cE_{1, (-1)^n}$ and $\cE_{1, (-1)^n} \owns \bs\phi_0 \mapsto \vec a(\bs\phi_0)$ are continuous, hence $\vec a \in C(I; \bR^n)$.
Moreover, \eqref{eq:phi-conv-en} implies
\begin{equation}
\label{eq:mod-basic-conv-0}
\lim_{m\to\infty}\big(\|\vec a_m\|_{C(I; \bR^n)} + \|\bs g_m - \bs g\|_{C(I; \cE)}\big) = 0.
\end{equation}

We claim that $\vec a \in C^1(I; \bR^n)$ and for any $k \in \{1, \ldots, n\}$ we have
\begin{equation}
\label{eq:ak-lin-syst}
\begin{aligned}
((-1)^kM - \la \partial_x^2 H_k(t), g(t)\ra)a_k'(t) + \sum_{j \neq k}\la \partial_x H_k(t), \partial_x H_j(t)\ra a_j'(t) \\
= {-}\la \partial_x H_k(t), \dot g(t)\ra.
\end{aligned}
\end{equation}

In order to justify \eqref{eq:ak-lin-syst}, first assume that $\bs \phi \in C^1(I; \cE_{1, (-1)^n})$
and \eqref{eq:csf} holds in the strong sense. Since the map $\bs\phi_0 \mapsto \vec a(\bs\phi_0)$
is of class $C^1$, we obtain $\vec a\in C^1(I; \bR^n)$.
The first component of \eqref{eq:ls-g-eq} yields
\begin{equation}
\label{eq:g-1st}
\partial_t g(t) = \dot g(t) + \sum_{k=1}^n (-1)^k a_k'(t)\partial_x H_k(t).
\end{equation}
Differentiating in time the relation \eqref{eq:g-orth}, we obtain \eqref{eq:ak-lin-syst}.

Consider now the general case (without the additional regularity assumptions).
By Proposition~\ref{prop:cauchy}, there exists a sequence $\bs\phi_m$ satisfying
\eqref{eq:phi-conv-en} and the additional regularity assumption stated above,
hence the corresponding parameters $\vec a_m$ satisfy \eqref{eq:mod-basic-conv-0}.
They also satisfy \eqref{eq:ak-lin-syst} with $\vec a$ replaced by $\vec a_m$,
which is a linear system for $\dd t\vec a_m(t)$.
The non-diagonal terms of its matrix are
$\lesssim y_{\min}\eee^{-y_{\min}} \ll \sqrt{\rho(\vec a)}$.
It follows from \eqref{eq:g-coer} that the diagonal terms differ from $(-1)^k M$ by a quantity of order at most $\sqrt{\rho(\vec a)}$.
In particular, the matrix is uniformly non-degenerate, hence we have uniform convergence
of $\dd t\vec a_m$, which implies that $\vec a \in C^1$ and \eqref{eq:ak-lin-syst} holds.

\textbf{Step 2.} (Coercivity.)
If we choose $C_1 := 2C_0$, where $C_1$ is the constant in \eqref{eq:delta-geq-ass}
and $C_0$ the constant in \eqref{eq:g-coer-stat}, then we have
\begin{equation}
\begin{aligned}
\|\bs g(t)\|_\cE^2 &\leq C_0\rho(\vec a(t)) + \frac 12C_1\big(E(\bs \phi) - nM\big) \\
&\leq C_0 \rho(\vec a(t)) + \frac 12 \delta(\bs \phi(t)) \\
&\leq C_0 \rho(\vec a(t)) + \frac 12 \big(\rho(\vec a(t)) + \|\bs g(t)\|_\cE^2\big),
\end{aligned}
\end{equation}
the last inequality resulting from \eqref{eq:d-def}.
This proves the first inequality in \eqref{eq:g-coer}, up to adjusting $C_0$.
The second inequality in \eqref{eq:g-coer} follows from \eqref{eq:dist-bd-C0}.

By standard estimates on the inverse of a diagonally dominant matrix, \eqref{eq:ak-lin-syst}
and \eqref{eq:g-coer} yield \eqref{eq:ak'-est}.

\textbf{Step 3.} (Continuous dependence of first derivatives.)
We have already observed that \eqref{eq:phi-conv-en} implies \eqref{eq:mod-basic-conv-0}.
It is clear that \eqref{eq:phi-conv-en} implies uniform convergence of the coefficients
of the system \eqref{eq:ak-lin-syst}, hence uniform convergence of $\dd t\vec a_m$,
and we conlude that \eqref{eq:phi-conv-en} implies \eqref{eq:mod-basic-conv}.
\end{proof}
\begin{remark}
We refer to
\cite[Proposition 3]{GuSi06}, \cite[Proposition 3.1]{MeZa12} or \cite[Lemma 3.3]{J-18-nonexist} for results and proofs similar to Lemmas~\ref{lem:static-mod} and \ref{lem:basic-mod}.
\end{remark}
The following more precise coercivity bound
will be useful in Section~\ref{sec:n-body}.
\begin{lemma}
\label{lem:impr-coer}
In the setting of Lemma~\ref{lem:basic-mod},
assume in addition that $E(\bs \phi) = nM$. Then
\begin{equation}
\label{eq:a'-est}
\|\partial_t g(t)\|_{L^2}^2 + \frac \nu2\|g(t)\|_{H^1}^2 + M|\vec a\,'(t)|^2
\leq 4\kappa^2 \rho(t) + C_0\rho(t)^{\frac 32},
\end{equation}
where $\nu > 0$ is the constant in Lemma~\ref{lem:D2H}.
\end{lemma}
\begin{proof}
From \eqref{eq:g-orth} and the Leibniz rule, we have
\begin{equation}
\la \partial_x H_k(t), \partial_t g(t)\ra = a_k'(t)\la \partial_x^2 H_k(t), g(t)\ra,
\end{equation}
thus \eqref{eq:g-coer} and \eqref{eq:ak'-est} yield
\begin{equation}
|\la \partial_x H_k(t), \partial_t g(t)\ra| \lesssim \rho(t).
\end{equation}
Using this bound, \eqref{eq:g-1st} yields
\begin{equation}
\label{eq:gdot-decomp}
\bigg|\|\dot g(t)\|_{L^2}^2 - \|\partial_t g\|_{L^2}^2 - \Big\|\sum_{k=1}^n (-1)^k a_k'(t)\partial_x H_k(t)\Big\|_{L^2}^2\bigg| \lesssim \rho(t)^{\frac 32}.
\end{equation}
%
Similarly as in \eqref{eq:EpHX-a-1}, we have
\begin{equation}
\int_{-\infty}^\infty\bigg(\bigg|\sum_{k=1}^n (-1)^k a_k'(t)\partial_x H_k(t)\bigg|^2 - \sum_{k=1}^n |a_k'(t)|^2|\partial_x H_k(t)|^2 \bigg)\ud x \lesssim y_{\min}\eee^{-2y_{\min}},
\end{equation}
thus we can rewrite \eqref{eq:gdot-decomp} as
\begin{equation}
\big|\|\dot g(t)\|_{L^2}^2 - \|\partial_t g\|_{L^2}^2 - M|\vec a\,'(t)|^2\big| \lesssim \rho(t)^{\frac 32}.
\end{equation}
Injecting this into \eqref{eq:g-coer-pres}, we obtain \eqref{eq:a'-est}.
\end{proof}

\subsection{Refined modulation}
\label{ssec:ref-mod}
In order to proceed with the analysis of the dynamics of the modulations parameters,
we follow an idea used in a similar context in \cite{J-18p-gkdv} and introduce \emph{localised momenta}, see also \cite[Proposition 4.3]{RaSz11}.

Recall that $\chi \in C^\infty$ is a decreasing function such that $\chi(x) = 1$
for all $x \leq \frac 13$ and $\chi(x) = 0$ for all $x \geq \frac 23$.
\begin{definition}[Localised momenta]
\label{def:p}
Let $I$, $\bs \phi$, $\vec a$ and $\bs g$ be as in Lemma~\ref{lem:basic-mod}.
We set
\begin{equation}
\begin{aligned}
\chi_1(t, x) &:= \chi\Big(\frac{x - a_1(t)}{a_2(t) - a_1(t)}\Big), \\
\chi_k(t, x) &:= \chi\Big(\frac{x - a_k(t)}{a_{k+1}(t) - a_k(t)}\Big)
- \chi\Big(\frac{x - a_{k-1}(t)}{a_k(t) - a_{k-1}(t)}\Big),\qquad\text{for }k \in \{2, \ldots, n-1\}, \\
\chi_n(t, x) &:= 1 - \chi\Big(\frac{x - a_{n-1}(t)}{a_n(t) - a_{n-1}(t)}\Big).
\end{aligned}
\end{equation}
We define $\vec p = (p_1, \ldots, p_n): I \to \bR^n$ by
\begin{equation}
\label{eq:p-def}
p_k(t) := \la {-}(-1)^k\partial_x H_k(t) + \chi_k(t)\partial_x g(t), \dot g(t)\ra.
\end{equation}
\end{definition}
\begin{lemma}
\label{lem:ref-mod}
There exists $C_0$ such that, under the assumptions of Lemma~\ref{lem:basic-mod},
$\vec p \in C^1(I; \bR^n)$ and the following bounds hold for all $t \in I$ and $k \in \{1, \ldots, n\}$:
\begin{align}
\label{eq:ak'} |Ma_k'(t) - p_k(t)| &\leq C_0 \rho(t), \\
\label{eq:pk'} |p_k'(t) - F_k(\vec a(t))| &\leq \frac{C_0\rho(t)}{-\log \rho(t)},
\end{align}
where $M$ and $F_k$ are defined by \eqref{eq:M-def} and \eqref{eq:Fk-def}.
\end{lemma}
\begin{remark}
If we think of $M$, $p_k$ and $F_k$ as the (rest) mass of the kink,
its momentum and the force acting on it, then \eqref{eq:ak'} and \eqref{eq:pk'}
yield (approximate) Newton's second law for the kink motion.
\end{remark}
\begin{proof}[Proof of Lemma~\ref{lem:ref-mod}]
The bound \eqref{eq:ak'} follows from \eqref{eq:p-def}, \eqref{eq:ak'-est} and \eqref{eq:g-coer}.

It is clear that $\vec p \in C(I; \bR^n)$.
We claim that $\vec p \in C^1(I; \bR^n)$ and for any
$k \in \{1, \ldots, n\}$ we have
\begin{equation}
\label{eq:pk'-id}
\begin{aligned}
p_k'(t) &= (-1)^k a_k'(t)\int_{-\infty}^{\infty}(1 - \chi_k(t))\partial_x^2 H_k(t)\dot g(t)\ud x \\
&- \sum_{j \neq k}(-1)^j a_j'(t)\int_{-\infty}^{\infty}\chi_k(t)\partial_x^2 H_j(t)\dot g(t)\ud x \\
&+ \frac 12 \int_{-\infty}^\infty \partial_x \chi_k(t)\big((\dot g(t))^2 + (\partial_x g(t))^2\big)\ud x - \int_{-\infty}^\infty \partial_t \chi_k(t)\dot g(t)\partial_x g(t)\ud x \\
&+ \int_{-\infty}^\infty \chi_k(t) \partial_x g(t)\big(U'(H(\vec a(t)) + g(t)) - \sum_{j}(-1)^j U'(H_j(t))\big)\ud x \\
&+ (-1)^k \int_{-\infty}^\infty \partial_x H_k(t)\Big(U'(H(\vec a(t)) + g(t)) \\
&\qquad\qquad\qquad\qquad\qquad- \sum_{j=1}^n (-1)^j U'(H_j(t)) - U''(H_k(t))g\Big)\ud x.
\end{aligned}
\end{equation}

In order to justify \eqref{eq:pk'-id}, first assume that
$\partial_x g, \dot g \in C(I; H^1(\bR)) \cap C^1(I; L^2(\bR))$,
that \eqref{eq:g-1st} holds in the strong sense in $H^1(\bR)$,
and that the second component of \eqref{eq:ls-g-eq}, which is
\begin{equation}
\label{eq:g-2nd}
\partial_t \dot g(t) = \partial_x^2 g(t) - U'(H(\vec a(t)) + g(t))+ \sum_{k=1}^n (-1)^k U'(H_k(t)),
\end{equation}
holds in the strong sense in $L^2(\bR)$.
From \eqref{eq:p-def} and the Leibniz rule, we obtain
\begin{equation}
\label{eq:pk'-1}
\begin{aligned}
p_k'(t) &= (-1)^k a_k'(t)\la \partial_x^2 H_k(t), \dot g(t)\ra
- \la \partial_t \chi_k(t)\partial_x g(t), \dot g(t)\ra \\
&- \Big\la \chi_k(t)\partial_x\Big(\dot g(t) + \sum_{j=1}^n (-1)^j a_j'(t)\partial_x H_j(t)\Big), \dot g(t)\Big\ra \\
&- \Big\la (-1)^k\partial_x H_k(t) + \chi_k(t)\partial_x g(t), \\
&\qquad\qquad\partial_x^2 g(t) - U'(H(\vec a(t)) + g(t))+ \sum_{j=1}^n (-1)^j U'(H_j(t))\Big\ra.
\end{aligned}
\end{equation}
We now observe that integrations by parts yield
\begin{equation}
\begin{aligned}
\la \chi_k(t)\partial_x \dot g(t), \dot g(t)\ra &= -\int_{-\infty}^\infty \frac 12 \partial_x \chi_k(t)(\dot g(t))^2\ud x, \\
\la \chi_k(t)\partial_x g(t), \partial_x^2 g(t)\ra &= -\int_{-\infty}^\infty \frac 12 \partial_x \chi_k(t)(\partial_x g(t))^2\ud x,
\end{aligned}
\end{equation}
and that \eqref{eq:Lc-ker} implies
\begin{equation}
\la \partial_x H_k(t), \partial_x^2 g(t)\ra = \int_{-\infty}^\infty \partial_x H_k(t) U''(H_k(t))g(t)\ud x.
\end{equation}
Inserting these relations into \eqref{eq:pk'-1} and rearranging
the terms, we obtain \eqref{eq:pk'-id}.

Consider now the general case (without the additional regularity assumptions). By Proposition~\ref{prop:cauchy},
there exists a sequence of solutions $\bs\phi_m$ satisfying \eqref{eq:phi-conv-en} and the additional regularity assumptions stated above.
Let $\vec p_m$ be the corresponding localised momentum.
By \eqref{eq:mod-basic-conv} and \eqref{eq:pk'-id},
the sequence $\big(\vec p_m\big)_m$ converges in $C^1$.

It remains to prove \eqref{eq:pk'}. In the computation below, we call a term ``negligible''
if its absolute value is smaller than the right hand side of \eqref{eq:pk'}.

By Proposition~\ref{prop:prop-H} and the definition of $\chi_k$, we have
\begin{equation}
\begin{aligned}
\int_{-\infty}^\infty \big((1 - \chi_k(t))\partial_x^2 H_k(t)\big)^2\ud x &\leq
2\int_{\frac 13 y_{\min}(t)}^{\infty}(\partial_x^2 H)^2\ud x \lesssim \eee^{-\frac 23 y_{\min}(t)} \lesssim \rho(t)^\frac 23.
\end{aligned}
\end{equation}
By a similar computation, for all $j \neq k$ we have
\begin{equation}
\begin{aligned}
\int_{-\infty}^\infty \big(\chi_k(t)\partial_x^2 H_j(t)\big)^2\ud x \lesssim \rho(t)^\frac 23.
\end{aligned}
\end{equation}
Applying \eqref{eq:g-coer}, \eqref{eq:ak'-est} and the Cauchy-Schwarz inequality,
we obtain that the first and second line of \eqref{eq:pk'-id} are at most of order $\rho(t)^\frac 43$,
hence negligible. Next, we observe that $\|\partial_x \chi_k\|_{L^\infty} \lesssim y_{\min}^{-1}
\lesssim ({-}\log \rho(t))^{-1}$, hence the first integral of the third line of \eqref{eq:pk'-id}
is negligible. The Chain Rule and \eqref{eq:ak'-est} yield $\|\partial_t \chi_k\|_{L^\infty}
\lesssim ({-}\log \rho(t))^{-1}\sqrt{\rho(t)}$, hence the second integral is negligible as well.

By Lemma~\ref{lem:sizeDEp} and the Cauchy-Schwarz inequality, the fourth line of \eqref{eq:pk'-id}
differs by a term of order $\sqrt{y_{\min}(t)}\eee^{-\frac 32 y_{\min}(t)} \ll \rho(t)({-}\log \rho(t))^{-1}$ from
\begin{equation}
\label{eq:pk'-4-alt}
\int_{-\infty}^\infty \chi_k(t) \partial_x g(t)\big(U'(H(\vec a(t)) + g(t)) - U'(H(\vec a(t)))\big)\ud x.
\end{equation}
We now transform the last integral of the right hand side of \eqref{eq:pk'-id}. By \eqref{eq:size-diff-pot}, we have
\begin{equation}
\int_{-\infty}^\infty \partial_x H_k(t) U''(H_k(t)) g(t) \ud x \sim \int_{-\infty}^\infty \partial_x H_k(t) U''(H(\vec a(t))) g(t) \ud x.
\end{equation}
Recalling the definition of $F_k(\vec a)$, see \eqref{eq:Fk-def}, we thus obtain
that the last integral of the right hand side of \eqref{eq:pk'-id} differs by a negligible term from
\begin{equation}
\label{eq:line-56}
\begin{aligned}
F_k(\vec a(t)) + (-1)^k\int_{-\infty}^\infty \partial_x H_k(t)\big(&U'(H(\vec a(t)) + g(t))
 \\ &- U'(H(\vec a(t))) - U''(H(\vec a(t)))g(t)\big)\ud x.
 \end{aligned}
\end{equation}
Applying the Taylor formula pointwise and using the fact that
$$\|\chi_k(t)\partial_x H(\vec a(t)) - (-1)^k \partial_x H_k(t)\|_{L^\infty} \lesssim \eee^{-\frac 13 y_{\min}(t)},$$
we see that up to negligible terms, in \eqref{eq:line-56} we can replace
$(-1)^k\partial_x H_k(t)$ by $\chi_k(t)\partial_x H(\vec a(t))$.
Recalling \eqref{eq:pk'-4-alt}, we thus obtain
\begin{equation}
\begin{aligned}
&p_k'(t) \sim F_k(\vec a(t)) \\&+ \int_{-\infty}^\infty \chi_k(t) \partial_x \big(U(H(\vec a(t)) + g(t)) - U(H(\vec a(t))) - U'(H(\vec a(t)))g(t)\big)\ud x,
\end{aligned}
\end{equation}
and an integration by parts shows that the second line is negligible.
\end{proof}
\section{Characterisations of kink clusters}
\label{sec:equiv-def}
\subsection{Kink clusters approach multi-kink configurations}
\label{ssec:close-to-H}
We now give a proof of Proposition~\ref{prop:close-to-H}.
The essential ingredient is the following lemma
yielding the strong convergence of minimising sequences,
similar to the results in \cite[Appendix A]{Abdon22p1}.
\begin{lemma}
\label{lem:min-conv}
\begin{enumerate}[(i)]
\item\label{it:min-conv-i}
If the sequences $x_m \in \bR$ and $\phi_m: ({-}\infty, x_m] \to \bR$ satisfy
\begin{equation}
\begin{gathered}
\lim_{x \to -\infty} \phi_m(x) = -1\quad\text{for all }m, \qquad
\lim_{m\to\infty} E_p(\phi_m; -\infty, x_m) = 0,
\end{gathered}
\end{equation}
then $\lim_{m\to\infty}\|\phi_m + 1\|_{H^1({-}\infty, x_m)} = 0$.
\item\label{it:min-conv-ii}
If the sequences $x_m' \in \bR$ and $\phi_m: [x_m', \infty) \to \bR$ satisfy
\begin{equation}
\begin{gathered}
\lim_{x \to \infty} \phi_m(x) = 1\quad\text{for all }m, \qquad
\lim_{m\to\infty} E_p(\phi_m; x_m', \infty) = 0,
\end{gathered}
\end{equation}
then $\lim_{m\to\infty}\|\phi_m - 1\|_{H^1(x_m', \infty)} = 0$.
\item\label{it:min-conv-iii}
If the sequence $\phi_m : \bR \to \bR$ satisfies
\begin{equation}
\begin{gathered}
\phi_m(0) = 0, \quad\lim_{x \to -\infty} \phi_m(x) = -1, \quad \lim_{x \to \infty} \phi_m(x) = 1\quad\text{for all }m, \\ \lim_{m \to \infty}E_p(\phi_m) = M,
\end{gathered}
\end{equation}
then $\lim_{m\to \infty}\|\phi_m - H\|_{H^1} = 0$.
\item\label{it:min-conv-iv}
If the sequences $x_m < 0$, $x_m' > 0$ and $\phi_m : [x_m, x_m'] \to \bR$ satisfy
\begin{equation}
\begin{gathered}
\phi_m(0) = 0, \quad \lim_{m \to -\infty} \phi_m(x_m) = -1, \quad \lim_{m \to \infty}\phi_m(x_m') = 1\quad\text{for all }m, \\ \lim_{m \to \infty}E_p(\phi_m; x_m, x_m') = M,
\end{gathered}
\end{equation}
then $\lim_{m \to \infty} x_m = -\infty$, $\lim_{m\to \infty}x_m' = \infty$ and
$\lim_{m \to \infty}\|\phi_m - H\|_{H^1(x_m, x_m')} = 0$.
\end{enumerate}
\end{lemma}
\begin{proof}
It follows from \eqref{eq:bogom} and \eqref{eq:bogom-2} that
\begin{equation}
\lim_{m\to \infty}\sup_{x \leq x_m}\int_{-1}^{\phi_m(x)}\sqrt{2U(y)}\ud y = 0,
\end{equation}
hence
\begin{equation}
\lim_{m\to \infty}\sup_{x \leq x_m}|\phi_m(x) + 1| = 0.
\end{equation}
Since $U''(-1) = 1$, for all $m$ large enough we have
\begin{equation}
U(\phi_m(x)) \geq \frac 14(\phi_m(x) + 1)^2\qquad\text{for all }x \leq x_m,
\end{equation}
which proves part \ref{it:min-conv-i}.

Part \ref{it:min-conv-ii} follows from \ref{it:min-conv-i} by symmetry.

In part \ref{it:min-conv-iii},
it suffices to prove that the conclusion holds for a subsequence of any subsequence.
We can thus assume that $\partial_x \phi_m \wto \partial_x \phi_0$ in $L^2(\bR)$
and $\phi_m \to \phi_0$ uniformly on every bounded interval.

From \eqref{eq:bogom}, we have
\begin{equation}
\label{eq:bogom-conv}
\lim_{m\to \infty} \int_{-\infty}^\infty\big(\partial_x \phi_m - \sqrt{2U(\phi_m)}\big)^2\ud x = 0.
\end{equation}
Restricting to bounded intervals and passing to the limit, we obtain $\partial_x \phi_0(x) = \sqrt{2U(\phi_0(x))}$ for all $x \in \bR$, hence $\phi_0 = H$.
Using again \eqref{eq:bogom-conv} restricted to bounded intervals,
we obtain $\lim_{m\to \infty}\|\phi_m - H\|_{H^1(-R, R)} = 0$ for every $R > 0$,
in particular $\lim_{m \to \infty} |E_p(\phi_m; -R, R) - E_p(H; -R, R)| = 0$.
Hence, there exists a sequence $R_m \to \infty$ such that
\begin{equation}
\label{eq:Rm-conv}
\lim_{m\to \infty}\big( \|\phi_m - H\|_{H^1(-R_m, R_m)} + |E_p(\phi_m; -R_m, R_m) - E_p(H; -R_m, R_m)|\big) = 0.
\end{equation}
Since $E_p(H; -R_m, R_m) \to M$ and $E_p(\phi_m) \to M$, we obtain
\begin{equation}
\lim_{m\to \infty}\big( E_p(\phi_m; -\infty, -R_m) + E_p(\phi_m; R_m, \infty)\big) = 0.
\end{equation}
Applying parts \ref{it:min-conv-i} and \ref{it:min-conv-ii}, we obtain
\begin{equation}
\lim_{m\to\infty}\big(\|\phi_m +1\|_{H^1(-\infty, -R_m)} + \|\phi_m - 1\|_{H^1(R_m, \infty)}\big) = 0.
\end{equation}
It is clear that $\lim_{m\to\infty}\big(\|H +1\|_{H^1(-\infty, -R_m)} + \|H - 1\|_{H^1(R_m, \infty)}\big) = 0$, thus \eqref{eq:Rm-conv} yields the conclusion.

In order to prove part \ref{it:min-conv-iv}, we define a new sequence $\wt \phi_m: \bR \to \bR$
by the formula, similar to \eqref{eq:weak-conv-cutoff},
\begin{equation}
\wt\phi_m(x) := \begin{cases}
-1 & \text{for all }x \leq x_m - 1, \\
-(x_m-x) + (1-x_m+x)\phi_m(x_m)& \text{for all }x \in [x_m - 1, x_m], \\
\phi_m(x) &\text{for all }x \in [x_m, x_m'], \\
(x - x_m') + (1-x+x_m')\phi_m(x_m') & \text{for all }x \in [x_m', x_m' + 1], \\
1 & \text{for all }x \geq x_m' + 1.
\end{cases}
\end{equation}
Then $\wt \phi_m$ satisfies the assumptions of part \ref{it:min-conv-iii}
and we obtain $\lim_{m\to \infty} \|\wt \phi_m - H\|_{H^1} = 0$.
In particular, $\wt\phi_m \to H$ uniformly, hence $H(x_m)\to -1$ and $H(x_m') \to 1$,
implying $x_m \to -\infty$ and $x_m' \to \infty$.
\end{proof}
\begin{proof}[Proof of Proposition~\ref{prop:close-to-H}]
If $\bs \phi$ is a solution of \eqref{eq:csf} satisfying \eqref{eq:close-to-H},
then Lemma~\ref{lem:interactions} yields $E(\bs \phi) = nM$.
Moreover, if $x_0(t), x_1(t), \ldots, x_n(t)$ are any functions such that
$$
\lim_{t\to\infty}\big(a_k(t) - x_{k-1}(t)\big) = \lim_{t\to\infty}\big(x_k(t) - a_{k}(t)\big) = \infty \qquad\text{for all }k \in \{1, \ldots, n\}
$$
then we have $\lim_{t\to\infty}\phi(t, x_k(t)) = (-1)^k$ for all $k \in \{0, 1, \ldots, n\}$.

In the opposite direction, assume that $\bs \phi$ is a kink cluster according to Definition~\ref{def:cluster}. It suffices to prove that $\lim_{t\to\infty}\delta(\bs\phi(t)) = 0$,
and apply Lemma~\ref{lem:basic-mod} to obtain a continuous choice of the positions of the kinks.

For all $k \in \{1, \ldots, n\}$ and $t$ sufficiently large, let $a_k(t) \in (x_{k-1}(t), x_k(t))$
be such that $\phi(t, a_k(t)) = 0$. From \eqref{eq:bogom} and \eqref{eq:bogom-2}, we obtain
\begin{equation}
\liminf_{t \to \infty} E_p(\phi(t); x_{k-1}(t), x_k(t)) \geq M, \qquad\text{for all }k \in \{1, \ldots, n\}.
\end{equation}
Thus, the condition $E(\bs \phi) \leq nM$ implies
\begin{gather}
\lim_{t \to \infty} \|\partial_t \phi(t)\|_{L^2} = 0, \\
\lim_{t \to \infty} E_p(\phi(t); x_{k-1}(t), x_k(t)) = M, \qquad\text{for all }k \in \{1, \ldots, n\}, \\
\lim_{t \to \infty} \big(E_p(\phi(t); -\infty, x_0(t)) + E_p(\phi(t); x_n(t), \infty)\big) = 0.
\end{gather}
Applying Lemma~\ref{lem:min-conv}, we obtain
\begin{gather}
\lim_{t \to \infty}(a_k(t) - x_{k-1}(t)) = \lim_{t \to \infty}(x_k(t) - a_k(t)) = \infty \qquad\text{for all }k \in \{1, \ldots, n\}
\end{gather}
and
\begin{equation}
\begin{aligned}
&\lim_{t\to \infty}\Big(\|\phi(t)\|_{H^1(-\infty, x_0(t))} + \|\phi(t)\|_{H^1(x_n(t),\infty)} \\
&\qquad+ \sum_{k=1}^n \|\phi(t) - (-1)^k H(\cdot - a_k(t))\|_{H^1(x_{k-1}(t), x_k(t))}\Big) = 0,
\end{aligned}
\end{equation}
thus
\begin{equation}
\lim_{t\to \infty}\big(\|\bs\phi(t) - \bs H(\vec a(t))\|_\cE^2 + \rho(\vec a(t))\big) = 0.\qedhere
\end{equation}
\end{proof}
\subsection{Asymptotically static solutions are kink clusters}
\label{ssec:asym-stat}
The present section is devoted to a proof of Proposition~\ref{prop:asym-stat}, which is inspired by some of the arguments in \cite{Cote15, JK}.
Recall that we restrict our attention to the case $U(\phi) = \frac 18(1-\phi^2)^2$,
so that $\pm \bs 1$ and $\pm \bs H(\cdot -a )$ are the only static states of \eqref{eq:csf}.
\begin{lemma}
\label{lem:kink-or-antikink}
Let $\bs\phi$ be a finite-energy solution of \eqref{eq:csf} such that $\lim_{t\to\infty}\|\partial_t \phi(t)\|_{L^2}^2 = 0$.
If $t_m \to \infty$ and $(a_m)_m$ is a sequence of real numbers such that
$\bs \phi(t_m, \cdot + a_m) \wto \bs \phi_0$ as $m \to \infty$,
then $\bs \phi_0$ is a static state.
\end{lemma}
\begin{proof}
Let $\wt{\bs \phi}$ be the solution of \eqref{eq:csf} for the initial data
$\wt{\bs\phi}(0) = \bs\phi_0$.
By Proposition~\ref{prop:cauchy} \ref{it:cauchy-weak}, for all $s \in [0, 1]$
we have $\partial_t \phi(t_m + s, \cdot + a_m) \wto \partial_t \wt\phi(s)$
in $L^2(\bR)$, thus for all $s \in [0, 1]$ we have
\begin{equation}
\|\partial_t \wt\phi(s)\|_{L^2} \leq \liminf_{m\to\infty}\|\partial_t \phi(t_m + s, \cdot + a_m)\|_{L^2} = \liminf_{m\to\infty}\|\partial_t \phi(t_m + s)\|_{L^2} = 0,
\end{equation}
implying that $\bs\phi_0$ is a static state.
\end{proof}

\begin{lemma}
\label{lem:asym-stat-Linf-conv}
Let $\bs\phi$ be a finite-energy solution of \eqref{eq:csf} such that
$\lim_{t\to\infty}\|\partial_t \phi(t)\|_{L^2}^2 = 0$
and let $t_m \to \infty$.
After extraction of a subsequence, there exist
$n \in \{0, 1, \ldots\}$, $\iota \in \{{-}1, 1\}$ and
$\vec a_{m} \in \bR^n$ for $m \in \{1, 2, \ldots\}$ such that
\begin{equation}
\label{eq:asym-stat-Linf-conv}
\lim_{m\to\infty}\big(\|\phi(t_m) - \iota H(\vec a_m)\|_{L^\infty}^2
+ \rho(\vec a_m)\big) = 0.
\end{equation}
\end{lemma}
\begin{proof}
Let $n$ be the maximal natural number such that there exist
a subsequence of $t_m$, and sequences $a_{m,1}, \ldots, a_{m,n}$
satisfying
$\lim_{m\to\infty}(a_{m,k+1} - a_{m,k}) = \infty$ for all $k \in \{1, \ldots, n-1\}$ and
\begin{equation}
\bs\phi(t_m, \cdot + a_{m, k}) \wto \bs\phi_k\neq \pm\bs 1\qquad\text{for all }k.
\end{equation}
By Lemma~\ref{lem:kink-or-antikink}, each $\bs \phi_k$ is a kink or an antikink. Upon adjusting the sequence $a_{m, k}$, we can assume
$\bs \phi_k = \pm \bs H$ for all $k$, and by an appropriate choice of $\iota$
we can reduce to $\bs\phi_1 = {-}\bs H$.

We claim that $\bs\phi_k = (-1)^k \bs H$ for all $k \in \{1, \ldots, n\}$.
This is true for $k = 1$. Suppose $k$ is such that $\bs \phi_k = \bs \phi_{k+1}$. Then there exists a sequence $x_m$ such that
$\lim_{m\to\infty}(x_m - a_{m, k}) = \lim_{m\to\infty}(a_{m, k+1} - x_m) = \infty$ and $\phi(t_m, x_m) = 0$ for all $m$.
Extracting a subsequence, we can assume that $\phi(t_m, \cdot + x_m)$
converges locally uniformly. The limit cannot be a vacuum,
contradicting the maximality of $n$ and finishing the proof of the claim.

Suppose there exists a sequence $x_m$ such that
\begin{equation}
\limsup_{m\to \infty} |\phi(t_m, x_m) - H(\vec a_m, x_m)| > 0.
\end{equation}
We then have $|a_{m, k} - x_m| \to \infty$,
thus $\bs \phi(t_m, \cdot + x_m)$ has (after extraction of a subsequence)
a non-constant weak limit, contradicting the maximality of $n$.
\end{proof}

The last auxiliary result which we need is the following
description of local in time stability of multi-kink configurations.
\begin{lemma}\label{lem:mkink-stab}
Let $n \in \{0, 1, \ldots\}$ and $\vec a_m \in \bR^n$ for $m \in \{1, 2, \ldots\}$ be such that $\lim_{m\to \infty}\rho(\vec a_m) = 0$.
Let $\bs g_{m, 0} \in \cE$ for $m \in \{1, 2, \ldots\}$
and let $\bs\phi_m: [-1, 1] \to \cE_{1, (-1)^n}$
be the solution of \eqref{eq:csf} for the initial data $\bs\phi_m(0) = \bs H(\vec a_m) + \bs g_{m, 0}$.
\begin{enumerate}[(i)]
\item \label{it:strong-mkink-stab}
If $\lim_{m\to\infty}\|\bs g_{m, 0}\|_\cE = 0$, then
\begin{equation}
\label{eq:loc-stab-1}
\lim_{m\to\infty}\sup_{t \in [-1, 1]}\|\bs\phi_m(t) - \bs H(\vec a_m)\|_{\cE} = 0.
\end{equation}
\item \label{it:weak-mkink-stab}
If $\limsup_{m\to\infty}\|g_{m, 0}\|_{H^1} < \infty$,
$\lim_{m\to\infty}\|\dot g_{m, 0}\|_{L^2} = 0$
and $\lim_{m\to\infty}\|g_{m, 0}\|_{L^\infty} = 0$, then
\begin{equation}
\label{eq:loc-stab-3}
\lim_{m\to\infty}\sup_{t \in [-1, 1]}\|\bs\phi_m(t) - \bs H(\vec a_m)
-\bs g_{m, \lin}(t)\|_{\cE} = 0,
\end{equation}
where $\bs g_{m, \lin}(t)$ is the solution of the free linear Klein-Gordon
equation for the initial data $\bs g_{m, \lin}(0) = \bs g_{m, 0}$.
\end{enumerate}
\end{lemma}
\begin{proof}
Part \ref{it:strong-mkink-stab} follows from part \ref{it:weak-mkink-stab} and the fact that $\|\bs g_{m, \lin}(t)\|_\cE = \|\bs g_{m, 0}\|_\cE$ for all $t$.





In order to prove \eqref{eq:loc-stab-3}, we set $\bs g_m(t):=  \bs\phi_m(t) - \bs H(\vec a_m)$ and $\bs h_m(t):= \bs g_m(t)  -\bs g_{m, \lin}(t)$. Then $\bs h_m(t) = (h_m(t), \partial_t h_m(t))$ solves the equation
$\partial_t^2 h_m - \partial_x^2 h_m + h_m = f_m$, where
\begin{equation}
\begin{aligned}
\label{eq:fm-gronwall}
f_m(t) := &- \big( U'(H(\vec a_m) + g_m(t)) - U'(H(\vec a_m)) - U''(H(\vec a_m)) g_m(t) \big)  \\
& - \Big( U'(H(\vec a_m)) - \sum_{k =1}^{n} (-1)^k U'( H( \cdot - a_{k, m})) \Big) \\
& - \big( U''(H(\vec a_m)) -1\big) g_m(t).
\end{aligned}
\end{equation}
Since $\bs h_m(0) = 0$, the standard energy estimate yields
\begin{equation}
\|\bs h_m(t)\|_\cE \leq \bigg|\int_0^t \|f_m(s)\|_{L^2}\ud s \bigg|\qquad\text{for all }t \in [-1, 1].
\end{equation}
By Gronwall's inequality, it suffices to verify that
\begin{equation}
\label{eq:fm-gronwall-bd}
\|f_m(t)\|_{L^2} \leq C\|\bs h_m(t)\|_{\cE} + \epsilon_m,\qquad\text{
with }\lim_{m\to\infty}\epsilon_m = 0.
\end{equation}
The first line of \eqref{eq:fm-gronwall} satisfies
\begin{align} 
&\| U'(H(\vec a_m) + g_m(t)) - U'(H(\vec a_m)) - U''(H(\vec a_m)) g_m(t) \|_{L^2} \\
&\qquad\lesssim \min(\|g_m(t)\|_{L^2}, \| g_m(t)^2 \|_{L^2}).
\end{align}
If $\|\bs h_m(t)\|_{\cE} \geq 1$, then
$$\|g_m(t)\|_{L^2} \leq \|h_m(t)\|_{L^2} + \|g_{m,\lin}(t)\|_{L^2} \leq \|\bs h_m(t)\|_\cE + \|\bs g_{m, 0}\|_\cE \lesssim \|\bs h_m(t)\|_\cE.$$
If $\|\bs h_m(t)\|_\cE \leq 1$, then
\begin{align}
\|g_m(t)^2\|_{L^2} &\leq \| g_m(t) \|_{L^\infty} \| g_m(t) \|_{L^2} \\
&\leq
(\|\bs h_m(t)\|_\cE + \|g_{m,\lin}(t)\|_{L^\infty})
(\|\bs h_m(t)\|_\cE + \|\bs g_{m,0}\|_\cE) \\
& \lesssim \|\bs h_m(t)\|_\cE + \|g_{m,\lin}(t)\|_{L^\infty}.
\end{align} 
The term on the second line of \eqref{eq:fm-gronwall} is equal to $-\textrm{D} E_p( H(\vec a_m))$, thus, using \eqref{eq:sizeDEp}, satisfies
\begin{align} 
\Big\| U'(H(\vec a_m)) - \sum_{k =1}^{n} (-1)^k U'( H( \cdot - a_{k, m})) \Big\|_{L^2} \lesssim \sqrt{{-}\log \rho(\vec a_m)} \rho(\vec a_m). 
\end{align} 
Finally, the third line of \eqref{eq:fm-gronwall} satisfies
\begin{align} 
\big\| \big( U''(H(\vec a_m)) -1\big) g_m(t) \big\|_{L^2} \leq \|U''(H(\vec a_m)) -1 \|_{L^2} \|g_m(t) \|_{L^{\infty}} \lesssim \|g_m(t) \|_{L^{\infty}}.
\end{align}
Gathering the estimates above and using the fact that,
by explicit formulas for the free Klein-Gordon equation,
we have
$\lim_{m\to\infty}\sup_{t\in[-1, 1]}\|g_{m,\lin}(t)\|_{L^\infty} = 0$, 
we obtain \eqref{eq:fm-gronwall-bd}.
\end{proof}
\begin{proof}[Proof of Proposition~\ref{prop:asym-stat}]
Without loss of generality, assume
\begin{equation}
\label{eq:asym-stat-class}
\lim_{x \to -\infty}\phi(t, x) = 1\qquad\text{for all }t.
\end{equation}
Suppose that either $n := M^{-1}E(\bs \phi) \notin \bN$, or
there exists a sequence $t_m \to \infty$ such that
\begin{equation}
\label{eq:asym-stat-contr}
\liminf_{m \to \infty}\delta(\bs \phi(t_m)) > 0,
\end{equation}
where $\delta$ defined by \eqref{eq:d-def-2} is the distance
to the set of $n$-kink configurations.
In both cases, for any $n' \in \bN$ and $\vec a_m \in \bR^{n'}$ we have
\begin{equation}
\label{eq:asym-stat-g-lbd}
\liminf_{m \to \infty}\big(\|\bs \phi(t_m) - \bs H(\vec a_m)\|_\cE^2
+ \rho(\vec a_m)\big) > 0.
\end{equation}
By Lemma~\ref{lem:asym-stat-Linf-conv}, after extracting a subsequence,
we can assume that \eqref{eq:asym-stat-Linf-conv} holds
for some $n' \in \bN$ and $\vec a_m \in \bR^{n'}$.
We have $\iota = 1$ due to \eqref{eq:asym-stat-class}.
We decompose
\begin{equation}
\bs \phi(t_m + s) = \bs H(\vec a_m) + \bs g_m(s), \qquad s \in [-1, 1].
\end{equation}
We claim that
\begin{equation}
\label{eq:g-large-en}
\liminf_{m\to\infty}\inf_{s \in [-1, 1]}\|\bs g_m(s)\|_\cE > 0.
\end{equation}
Otherwise, by Lemma~\ref{lem:mkink-stab} \ref{it:strong-mkink-stab},
we would have $\lim_{m\to\infty}\|\bs g_m(0)\|_\cE = 0$,
contradicting \eqref{eq:asym-stat-g-lbd}.

Let $\bs g_{m, \lin}: [-1, 1] \to \cE$ be the solution of the free
linear Klein-Gordon equation for the initial data $\bs g_{m, \lin}(0)
= \bs g_m(0)$. By Lemma~\ref{lem:mkink-stab} \ref{it:weak-mkink-stab},
we have
\begin{equation}
\lim_{m\to\infty}\sup_{s\in [-1, 1]}\|\bs g_m(s) - \bs g_{m, \lin}(s)\|_\cE = 0,
\end{equation}
thus
\begin{gather}
\lim_{m\to\infty}\sup_{s\in[-1, 1]}\|\partial_t g_{m, \lin}(s)\|_{L^2} = 0, \\
\liminf_{m\to\infty}\inf_{s\in[-1, 1]}\|g_{m, \lin}(s)\|_{H^1} > 0,
\end{gather}
the last inequality following from \eqref{eq:g-large-en}.
It is well-known that such behaviour is impossible
for the free Klein-Gordon equation.
For instance, it contradicts the identity
\begin{equation}
\begin{aligned}
&\int_{-\infty}^\infty \big(\partial_t g_{m, \lin}(1, x)g_{m, \lin}(1, x) - \partial_t g_{m, \lin}(-1, x)g_{m, \lin}(-1, x)\big)\ud x \\
&\qquad =\int_{-1}^1 \int_{-\infty}^\infty \big((\partial_t g_{m, \lin}(t, x))^2 - (\partial_x g_{m, \lin}(t, x))^2 - g_{m, \lin}(t, x)^2\big)\ud x\ud t,
\end{aligned}
\end{equation}
obtained by multiplying the equation by $g_{m, \lin}$ and integrating in space-time.
Hence, \eqref{eq:asym-stat-contr} is impossible.
\end{proof}

\section{Main order asymptotic behaviour of kink clusters}
\label{sec:n-body}
The current section is devoted to the analysis of the approximate $n$-body problem
associated with a kink cluster motion, obtained in Lemma~\ref{lem:ref-mod}.
\subsection{Linear algebra results}
\label{ssec:lin-alg}
We gather here a few auxiliary results needed later.
\begin{definition}
\label{def:laplacian}
For any $n \in \{1, 2, \ldots\}$, the (positive) discrete
Dirichlet Laplacian $\Delta^{(n)} = (\Delta_{jk}^{(n)})_{j, k=1}^{n-1} \in \bR^{(n-1)\times (n-1)}$ is defined by
$\Delta^{(n)}_{kk} = 2$ for $k = 1, \ldots, n-1$,
$\Delta^{(n)}_{k, k+1} = \Delta^{(n)}_{k+1, k} = -1$
for $k = 1, \ldots, n-2$ and $\Delta^{(n)}_{j, k} = 0$ for $|j - k| \geq 2$.
\end{definition}
We denote
\begin{equation}
\label{eq:1-and-sig-def}
\vec 1 := (1, \ldots, 1) \in \bR^{n-1}, \qquad
\vec \sigma = (\sigma_1, \ldots, \sigma_{n-1}), \ \sigma_k := \frac{k(n-k)}{2}.
\end{equation}
We have
\begin{equation}
\Delta^{(n)}\vec\sigma = \Big( k(n-k) - \frac 12((k-1)(n-k+1) + (k+1)(n-k-1)) \Big)_{k = 1}^{n-1} = \vec 1.
\end{equation}
We also denote
\begin{align}
\Pi &:= \{\vec z\in \bR^{n-1}: \vec \sigma\cdot \vec z = 0\}, \\
\mu_0 &:= \frac{1}{\vec \sigma\cdot \vec 1}
= \frac{12}{(n+1)n(n-1)}, \\
P_\sigma \vec y &:= \vec y - \frac{\vec \sigma\cdot\vec y}{\|\vec \sigma\|^2}\vec \sigma \quad\text{(the orthogonal projection of $\vec y \in \bR^{n-1}$ on $\Pi$)}, \\
P_1 \vec y &:= \vec y - \mu_0(\vec\sigma\cdot\vec y)\vec 1\quad\text{(the projection of $\vec y \in \bR^{n-1}$ on $\Pi$ along the direction $\vec 1$)}.
\end{align}
Observe that
\begin{equation}
P_1 \Delta^{(n)}\vec y = \Delta^{(n)} \vec y - \mu_0 (\vec\sigma\cdot (\Delta^{(n)}\vec y))\vec 1 = \Delta^{(n)}\vec y - \mu_0 (\vec 1\cdot \vec y) \vec 1.
\end{equation}
\begin{lemma}
\label{lem:matrix-D}
The matrix $\Delta^{(n)}$ is positive definite. There exists $\mu_1 > 0$ such that for all $\vec y \in \bR^{n-1}$
\begin{equation}
\label{eq:matrix-D}
\vec y\cdot (P_1 \Delta^{(n)}\vec y) = \vec y\cdot(\Delta^{(n)}\vec y) - \mu_0(\vec 1\cdot\vec y)^2 \geq \mu_1 | P_\sigma \vec y |^2.
\end{equation}
\end{lemma}
\begin{proof}
Since $\vec \sigma$ and $\vec 1$ are not orthogonal, it suffices to prove \eqref{eq:matrix-D}.

Consider the matrix $\wt \Delta = (\wt \Delta_{jk})$ given by $\wt \Delta_{jk} := 2\delta_{jk} + \mu_0 - \Delta_{jk}$.
The desired inequality is equivalent to
\begin{equation}
\label{eq:perron}
\vec y \cdot (\wt \Delta\vec y) \leq 2|\vec y|^2 - \mu_1 | P_\sigma \vec y |^2.
\end{equation}
We have
\begin{equation}
\wt \Delta\vec \sigma = 2\vec\sigma + \mu_0(\vec 1\cdot \vec \sigma)\vec 1 - \Delta^{(n)}\vec \sigma = 2\vec\sigma.
\end{equation}
The matrix $\wt \Delta$ and the vector $\vec \sigma$ have strictly positive entries.
By the Perron-Frobenius theorem, the largest eigenvalue of $\wt \Delta$ equals $2$ and is simple, which implies \eqref{eq:perron}.
\end{proof}

\begin{lemma}
\label{lem:zcr} \begin{enumerate}[(i)]
\item \label{it:zcr-1}
The function
\begin{equation}
\Pi \owns \vec z \mapsto \vec 1\cdot \eee^{-\vec z} \in (0, \infty)
\end{equation}
has a unique critical point $\vec z_\tx{cr}$, which is its global minimum.
\item\label{it:zcr-2}
There exist $C > 0$ and $\delta_0 > 0$ such that for all $\vec z \in \Pi$ and $\delta \leq \delta_0$
\begin{equation}
|P_\sigma \eee^{-\vec z}| \leq \delta\vec 1 \cdot\eee^{-\vec z}\quad\Rightarrow\quad|\vec z - \vec z_\tx{cr}| \leq C\delta.
\end{equation}
\item\label{it:zcr-3}
There exists $C > 0$ such that for all $\vec z \in \Pi$
\begin{equation}
C^{-1} |P_\sigma \eee^{-\vec z}|^2 \leq \vec 1\cdot\eee^{-\vec z}\big(\vec 1\cdot \eee^{-\vec z} - \vec 1\cdot \eee^{-\vec z_\tx{cr}}\big) \leq C |P_\sigma \eee^{-\vec z}|^2.
\end{equation}
\end{enumerate}
\end{lemma}
\begin{proof}
The function $\vec z \mapsto  \vec 1\cdot \eee^{-\vec z}$,
defined on the hyperplane $\Pi$, is strictly convex and tends to $\infty$ as $|\vec z| \to \infty$, thus it has no critical points other than its global minimum.
By the Lagrange multiplier method, this unique critical point $\vec z_\tx{cr}$ is determined by
\begin{equation}
\label{eq:zcr-lagr}
\eee^{-\vec z_\tx{cr}} = \lambda_\tx{cr}\vec \sigma \quad\Leftrightarrow\quad \vec z_\tx{cr} = -\log(\lambda_\tx{cr}\vec \sigma)
= -\log(\lambda_\tx{cr})\vec 1 - \log\vec\sigma,
\end{equation}
where $\lambda_\tx{cr} := \exp(-\mu_0\vec\sigma\cdot\log\vec\sigma)$ is the unique value for which $\vec z_\tx{cr} \in \Pi$.

We can rewrite the condition $|P_\sigma \eee^{-\vec z}| \leq\delta \vec 1\cdot\eee^{-\vec z}$ as follows:
\begin{equation}
\eee^{-\vec z} = \lambda \vec \sigma + \vec u, \qquad \lambda \in \bR,\ \vec u \in \Pi,\ |\vec u| \leq \delta \vec 1\cdot\eee^{-\vec z}.
\end{equation}
We thus have $\lambda = |\vec\sigma|^{-2} \vec\sigma\cdot \eee^{-\vec z}$, which implies $\lambda \simeq \vec 1\cdot \eee^{-\vec z}$, so
\begin{equation}
|{-}\vec z - \log(\lambda\vec \sigma)| = |\log(\lambda \vec \sigma + \vec u) - \log(\lambda\vec \sigma)| \lesssim \frac{|\vec u|}{\lambda} \lesssim \delta.
\end{equation}
By \eqref{eq:zcr-lagr} and the triangle inequality, we obtain
\begin{equation}
|\vec z - \vec z_\tx{cr} + (\log\lambda-\log\lambda_\tx{cr})\vec 1| \lesssim \delta.
\end{equation}
Taking the inner product with $\vec\sigma$, we get $|\log\lambda-\log\lambda_\tx{cr}| \lesssim \delta$, which proves \ref{it:zcr-2}.

Finally, we prove \ref{it:zcr-3}.
If $|\vec z - \vec z_\tx{cr}| \leq 1$, then $\vec 1\cdot \eee^{-\vec z} \simeq 1$
and $\vec 1\cdot \eee^{-\vec z} - \vec 1\cdot\eee^{-\vec z_\tx{cr}} \simeq |\vec z - \vec z_\tx{cr}|^2$, thus it suffices to verify that $|P_\sigma \eee^{-\vec z}| \simeq |z - z_\tx{cr}|$.
The inequality $\lesssim$ follows from $P_\sigma \eee^{-\vec z_\tx{cr}} = 0$
and the mean value theorem. The inequality $\gtrsim$ follows from \ref{it:zcr-2}.
Suppose now that $|\vec z - \vec z_\tx{cr}| \geq 1$. We then have
$\vec 1\cdot\eee^{-\vec z} - \vec 1 \cdot\eee^{-\vec z_\tx{cr}} \simeq \vec 1\cdot\eee^{-\vec z}$
and, invoking again \ref{it:zcr-2}, $|P_\sigma \eee^{-\vec z}| \simeq \vec 1\cdot\eee^{-\vec z}$.
\end{proof}

\subsection{Analysis of the $n$-body problem}
\label{ssec:analysis-mod}
Let $\vec y(t) = (y_1(t), \ldots, y_{n-1}(t))$ be defined by
\begin{equation}
\label{eq:y-def}
y_k(t) := a_{k+1}(t) - a_k(t), \qquad k \in \{1, \ldots, n-1\}
\end{equation}
and let
\begin{equation}
y_{\min}(t) := \min_{1\leq k<n} y_k(t).
\end{equation}
We thus have
\begin{equation}
\rho(t) = \sum_{k=1}^{n-1}\eee^{-y_k(t)} \simeq \eee^{-y_{\min}(t)}.
\end{equation}
Let $\vec q(t) = (q_1(t), \ldots, q_{n-1}(t))$
be defined by
\begin{equation}
\label{eq:q-def}
q_k(t) := M^{-1}\big(p_{k+1}(t) - p_{k}(t)\big).
\end{equation}
By \eqref{eq:ak'}, we have
\begin{equation}
\label{eq:y'-ineq}
|{\vec y \,}'(t) - \vec q(t)| \lesssim \eee^{-y_{\min}(t)},
\end{equation}
in particular \eqref{eq:ak'-est} yields
\begin{equation}
\label{eq:q-ineq}
|\vec q(t)| \lesssim \eee^{-\frac 12 y_{\min}(t)}.
\end{equation}
Let
\begin{equation}
\label{eq:A-def}
A:= \frac{\kappa\sqrt 2}{\sqrt M}.
\end{equation}
By Lemma~\ref{lem:Fz} and Lemma~\ref{lem:ref-mod}, we have
\begin{equation}
\label{eq:q'-ineq}
{\vec q\,}'(t) = -A^2\Delta^{(n)}\eee^{-\vec y(t)} + O(y_{\min}(t)^{-1} \eee^{-y_{\min}(t)}),
\end{equation}
where $\Delta^{(n)}$ is the matrix of the (positive) discrete Dirichlet Laplacian, see Definition~\ref{def:laplacian},
and we denote $\eee^{-\vec y(t)} := (\eee^{-y_1(t)}, \ldots, \eee^{-y_{n-1}(t)})$
as explained in Section~\ref{ssec:notation}.

%
\begin{lemma}
\label{lem:ymin-asym}
If $n > 1$ and $\bs\phi$ is a kink $n$-cluster, then
\begin{equation}
\sup_{t \geq 0}|y_{\min}(t) - 2\log t| < \infty.
\end{equation}
\end{lemma}
\begin{proof}
We first prove that $\sup_{t \geq T_0} y_{\min}(t) - 2\log t < \infty$.
By \eqref{eq:g-coer}, we have $|y_k'(t)| \lesssim \eee^{-\frac 12 y_{\min}(t)}$ for all $k \in \{1, \ldots, n\}$.
Therefore, $y_{\min}$ is a locally Lipschitz function, satisfying
\begin{equation}
\label{eq:ymin'}
|y_{\min}'(t)| \lesssim \eee^{-\frac 12 y_{\min}(t)},\qquad\text{for almost all }t \geq 0.
\end{equation}
By the Chain Rule, $|(\eee^{\frac 12 y_{\min}(t)})'| \lesssim 1$ almost everywhere.
After integration, we obtain $y_{\min}(t) - 2\log t \lesssim 1$.

Now we prove that $\sup_{t \geq 0} y_{\min}(t) - 2\log t > -\infty$.
Let $c_0 > 0$ be small enough and consider the function
\begin{equation}
\beta(t) := \vec q(t)\cdot \eee^{-\vec y(t)}- c_0\eee^{-\frac 32 y_{\min}(t)} = \sum_{k=1}^{n-1} q_k(t) \eee^{-y_k(t)} - c_0\eee^{-\frac 32 y_{\min}(t)}.
\end{equation}
Using \eqref{eq:y'-ineq} and \eqref{eq:q'-ineq} and $|q_k(t)| \lesssim \eee^{-\frac 12 y_{\min}(t)} \ll y_{\min}^{-1}$, we obtain
\begin{equation}
\begin{aligned}
\beta'(t) = &-\sum_{k=1}^{n-1} q_k(t)^2 \eee^{-y_k(t)}- A^2 \eee^{-\vec y(t)}\cdot (\Delta^{(n)}\eee^{-\vec y(t)}) \\
&+ \frac 32c_0y_{\min}'(t)\eee^{-\frac 32 y_{\min}(t)}
+ O(y_{\min}(t)^{-1}\eee^{-2y_{\min}(t)}).
\end{aligned}
\end{equation}
By Lemma~\ref{lem:matrix-D}, there is $c_1 > 0$ such that for all $t$ we have
$ \eee^{-\vec y(t)}\cdot(\Delta^{(n)}\eee^{-\vec y(t)}) \geq c_1\eee^{-2y_{\min}(t)}$.
If we recall \eqref{eq:ymin'}, we obtain that $\beta$ is a decreasing function
if $c_0$ is sufficiently small. By assumption, $\lim_{t\to\infty}y_{\min}(t) = \infty$, which, together with \eqref{eq:q-ineq}, implies $\lim_{t\to\infty}\beta(t) = 0$,
hence $\beta(t) \geq 0$ for all $t \geq T_0$, in other words
\begin{equation}
\label{eq:1stderlbound}
 \vec q(t)\cdot \eee^{-\vec y(t)} \geq c_0\eee^{-\frac 32 y_{\min}(t)}.
\end{equation}
Consider now the function $\rho(t)$ defined by \eqref{eq:rhot-def}.
From \eqref{eq:y'-ineq} and \eqref{eq:1stderlbound} we get, perhaps modifying $c_0$,
\begin{equation}
\label{eq:rho'-diff-ineq}
\rho'(t) = -\vec q(t)\cdot \eee^{-\vec y(t)}
+ O(\eee^{-2y_{\min}}) \leq -c_0 \rho(t)^\frac 32 \ \Leftrightarrow \ \dd t\big((\rho(t))^{-\frac 12}\big) \geq \frac{c_0}{2},
\end{equation}
so that, after integrating in time, $\rho(t) \lesssim t^{-2}$, yielding the conclusion.
\end{proof}
We have thus obtained upper and lower bounds on the minimal weighted distance between neighbouring kinks.
Our next goal is to obtain an upper bound on the \emph{maximal} distance.
From the lower bound, by integrating the modulation inequalities we get $|\vec y(t)| \lesssim \log t$.
However, because of exponential interactions, our analysis requires an estimate up to an \emph{additive} constant.
\begin{remark}
At this point, by means of energy estimates, we could improve bounds on the error to $\|\bs g(t)\|_\cE \leq t^{-2 + \epsilon}$ for $t \gg 1$.
However, such information does not seem to trivialise the analysis of the ODE,
which is not immediate even in the absence of \emph{any} error terms.
\end{remark}
\begin{lemma}
\label{lem:U-lbound}
If $\bs\phi$ is a kink $n$-cluster and $\rho$ is defined by \eqref{eq:rhot-def}, then
\begin{equation}
\label{eq:rho-lbound}
\limsup_{t\to \infty}\,\frac 12A^2t^2 \rho(t) \geq \vec 1\cdot \vec \sigma.
\end{equation}
\end{lemma}
\begin{proof}
The idea is to estimate from below the kinetic energy.

From the lower bound on the distance between the kinks, there exists $C$ such that for all $k \in \{1, \ldots, n \}$ and $t \geq T_0$ we have
\begin{equation}
|a_{n+1 - k}(t) - a_k(t)| \geq 2|n+1-2k|\log t - C,
\end{equation}
in particular
\begin{equation}
\label{eq:U-lbound-1}
\int_{T_0}^t \big(|a_{n+1 - k}'(s)| + |a_k'(s)|\big)\ud s \geq 2|n+1-2k|\log t - C.
\end{equation}
By Lemma~\ref{lem:ref-mod} and Lemma~\ref{lem:ymin-asym}, we have
\begin{equation}
\int_{T_0}^\infty |M^{-1}\vec p(s) - {\vec a \,}'(s)|\ud s < \infty,
\end{equation}
thus \eqref{eq:U-lbound-1} yields
\begin{equation}
M^{-1}\int_{T_0}^t (|p_k(s)| + |p_{n+1-k}(s)|)\ud s \geq 2|n+1-2k|\log t - C.
\end{equation}
Multiplying by $|n+1 - 2k|$ and taking the sum in $k$, we obtain
\begin{equation}
M^{-1}\sum_{k=1}^n \int_{T_0}^t |(n+1-2k)p_k(s)| \geq \bigg(\sum_{k=1}^n(n+1 - 2k)^2\bigg)\log t - C,
\end{equation}
hence, as $t \to \infty$, the quantity
\begin{equation}
\int_{T_0}^t\Big(s\sum_{k=1}^n|(n+1-2k)p_k(s)| - M\sum_{k=1}^n(n+1 - 2k)^2\Big)\frac{\vd s}{s}
\end{equation}
is bounded from below, which implies in particular
\begin{equation}
\limsup_{t\to\infty}\,t\sum_{k=1}^n |(n+1-2k)p_k(t)| \geq M\sum_{k=1}^n(n+1 - 2k)^2.
\end{equation}
The Cauchy-Schwarz inequality yields
\begin{equation}
\label{eq:p3-est}
\begin{aligned}
&\frac{1}{M^2}\limsup_{t\to \infty}\, t^2 \sum_{k=1}^n |p_k(t)|^2\sum_{k=1}^n(n+1-2k)^2 \geq \\
&\frac{1}{M^2}\bigg(\limsup_{t\to\infty}\,t\sum_{k=1}^n |(n+1-2k)p_k(t)|\bigg)^2
\geq \Big( \sum_{k=1}^n(n+1-2k)^2 \Big)^2,
\end{aligned}
\end{equation}
thus
\begin{equation}
\label{eq:p2-est}
\frac{1}{4M^2}\limsup_{t\to \infty}\, t^2 \sum_{k=1}^n |p_k(t)|^2
\geq \frac 14\sum_{k=1}^n(n+1 - 2k)^2 = \frac{(n-1)n(n+1)}{12} = \vec 1\cdot \vec \sigma.
\end{equation}
By \eqref{eq:a'-est} and \eqref{eq:ak'}, we have
\begin{equation}
\frac{1}{4M^2}|\vec p(t)|^2 \leq \big(\kappa^2M^{-1} + o(1)\big) \rho(t) = \Big(\frac 12 A^2 + o(1)\Big)\rho(t),
\end{equation}
thus \eqref{eq:p2-est} implies \eqref{eq:rho-lbound}.
\end{proof}

\begin{lemma}
\label{lem:y-bound-seq}
If $\bs\phi$ is a kink $n$-cluster, then there exists an increasing sequence $(t_m)_{m=1}^\infty$,
$\lim_{m\to\infty} t_m = \infty$, such that
\begin{equation}
\label{eq:y-bound-seq}
\lim_{m\to\infty}\big|\vec y(t_m) - \big(2\log(At_m)\vec 1 - \log(2\vec\sigma)\big)\big| = 0.
\end{equation}
\end{lemma}
\begin{proof}
Set $\conj\ell := \limsup_{t\to \infty} \frac 12 A^2 t^2 \rho(t)$.
Lemma~\ref{lem:U-lbound} yields $\conj\ell \geq \vec 1\cdot \vec \sigma$
and Lemma~\ref{lem:ymin-asym} yields $\conj\ell < \infty$.
Let $t_m$ be an increasing sequence such that $\lim_{m\to \infty} t_m = \infty$ and
\begin{equation}
\lim_{m\to \infty} \frac 12A^2 t_m^2 \rho(t_m) = \conj\ell.
\end{equation}
We will prove that $(t_m)$ satisfies \eqref{eq:y-bound-seq}, in particular $\conj\ell = \vec 1\cdot\vec \sigma$.
In order to make the formulas shorter, we denote $\vec y_m := \vec y(t_m)$ and $\vec q_m := \vec q(t_m)$.

Let $0 < \epsilon \ll 1$. The idea is to deduce \eqref{eq:y-bound-seq}
by analysing the evolution of $\vec y(t)$ for $(1-\epsilon)t_m \leq t \leq (1+\epsilon)t_m$.
In the computation which follows, the asymptotic notation $o$, $O$, $\ll$ etc.
is used for claims about the asymptotic behaviour of various quantities for $m$ large enough and $\epsilon$ small enough.

Lemma~\ref{lem:ymin-asym} yields $\rho(t) \lesssim t^{-2}$, hence $|{\vec y \,}'(t)| \lesssim t^{-1}$. We thus obtain
\begin{equation}
\label{eq:y-bound-eps}
|\vec y(t) - \vec y_m| \lesssim \epsilon,\qquad \text{for all }(1-\epsilon)t_m \leq t \leq (1+\epsilon)t_m,
\end{equation}
which implies
\begin{equation}
\label{eq:ey-bound-eps}
|\eee^{-\vec y(t)} - \eee^{-\vec y_m}| \lesssim \epsilon t_m^{-2},\qquad \text{for all }(1-\epsilon)t_m \leq t \leq (1+\epsilon)t_m.
\end{equation}
Using this bound and integrating \eqref{eq:q'-ineq} in time, we get
\begin{equation}
\vec q(t) = \vec q_m - (t-t_m)A^2 \Delta^{(n)} \eee^{-\vec y_m} + o(\epsilon t_m^{-1}),\qquad \text{for all }(1-\epsilon)t_m \leq t \leq (1+\epsilon)t_m.
\end{equation}
Integrating \eqref{eq:y'-ineq}, we get
\begin{equation}
\begin{aligned}
\vec y(t) &= \vec y_m + \int_{t_m}^t \vec q(\tau)\ud \tau + O(\epsilon t_m^{-1}) \\
&= \vec y_m + \int_{t_m}^t \big(\vec q_m - (\tau-t_m)A^2 \Delta^{(n)} \eee^{-\vec y_m}\big)\ud\tau + o(\epsilon^2) + O(\epsilon t_m^{-1}) \\
&= \vec y_m + (t - t_m)\vec q_m - \frac 12 (t - t_m)^2 A^2 \Delta^{(n)} \eee^{-\vec y_m} + o(\epsilon^2) + O(\epsilon t_m^{-1}).
\end{aligned}
\end{equation}
By taking $m$ large enough, we can have $t_m^{-1} \ll \epsilon$. The bound above allows us to compute the asymptotic expansion
of $t^2 \rho(t)$ up to order $o(\epsilon^2)$. Recalling from Section~\ref{ssec:notation} our notation for component-wise operations on vectors, we can write
\begin{equation}
\eee^{-\vec y(t)} = \eee^{-\vec y_m}\Big(1 - (t-t_m)\vec q_m+ \frac 12(t-t_m)^2 \big((\vec q_m)^2 + A^2 \Delta^{(n)} \eee^{-\vec y_m}\big)\Big) + o(\epsilon^2 t_m^{-2}),
\end{equation}
thus
\begin{equation}
\begin{aligned}
\rho(t) &= \rho(t_m) - (t - t_m) \vec q_m\cdot\eee^{-\vec y_m} \\
&+ \frac 12 (t-t_m)^2\big( (\vec q_m)^2\cdot \eee^{-\vec y_m}
+ A^2 \eee^{-\vec y_m}\cdot \Delta^{(n)} \eee^{-\vec y_m}\big) + o(\epsilon^2 t_m^{-2}).
\end{aligned}
\end{equation}
In particular, we have
\begin{equation}
\begin{aligned}
&((1\pm\epsilon)t_m)^2 \rho((1+\epsilon)t_m) = t_m^2 \rho(t_m) \pm \epsilon\big(2t_m^2 \rho(t_m) - t_m^3\vec q_m\cdot \eee^{-\vec y_m}\big) \\
&\quad+\epsilon^2 \Big(t_m^2 \rho(t_m) - 2t_m^3 \vec q_m\cdot \eee^{-\vec y_m} +\frac 12 t_m^4\big( (\vec q_m)^2\cdot \eee^{-\vec y_m}
+ A^2 \eee^{-\vec y_m}\cdot \Delta^{(n)} \eee^{-\vec y_m}\big)\Big) + o(\epsilon^2).
\end{aligned}
\end{equation}
From this estimate and the definition of the sequence $(t_m)$ we deduce that
\begin{equation}
\label{eq:1st-ord-eps}
\lim_{m\to \infty} \big(2t_m^2 \rho(t_m) - t_m^3 \vec q_m\cdot\eee^{-\vec y_m}\big) = 0
\end{equation}
and
\begin{equation}
\label{eq:2nd-ord-eps}
\limsup_{m\to\infty} \Big(t_m^2 \rho(t_m) - 2t_m^3 \vec q_m\cdot\eee^{-\vec y_m} +\frac 12 t_m^4\big((\vec q_m)^2\cdot \eee^{-\vec y_m}
+ A^2\eee^{-\vec y_m}\cdot \Delta^{(n)} \eee^{-\vec y_m}\big)\Big) \leq 0.
\end{equation}
By the Cauchy-Schwarz inequality,
\begin{equation}
\big(t_m^3\vec q_m\cdot \eee^{-\vec y_m}\big)^2 \leq t_m^2\vec 1\cdot\eee^{-\vec y_m} t_m^4 (\vec q_m)^2\cdot \eee^{-\vec y_m}
= t_m^2 \rho(t_m) t_m^4 (\vec q_m)^2\cdot \eee^{-\vec y_m},
\end{equation}
thus \eqref{eq:1st-ord-eps} yields
\begin{equation}
\label{eq:ord-eps-3}
\liminf_{m\to\infty} t_m^4  (\vec q_m)^2\cdot \eee^{-\vec y_m} \geq 4\lim_{m\to\infty}t_m^2 \rho(t_m).
\end{equation}
Injecting this estimate to \eqref{eq:2nd-ord-eps} and using again \eqref{eq:1st-ord-eps}, we obtain
\begin{equation}
\limsup_{m\to\infty} \big(t_m^4 A^2 \eee^{-\vec y_m}\cdot \Delta^{(n)}\eee^{-\vec y_m}\big) \leq 2\lim_{m\to\infty}t_m^2 \rho(t_m),
\end{equation}
hence, by Lemma~\ref{lem:matrix-D},
\begin{equation}
\limsup_{m\to\infty}\big(A^2 \mu_0 (t_m^2\rho(t_m))^2 + A^2 \mu_1|t_m^2 P_\sigma \eee^{-\vec y_m}|^2\big) \leq 2\lim_{m\to\infty}t_m^2 \rho(t_m).
\end{equation}
Recalling that $\lim_{m\to\infty} t_m^2 \rho(t_m) \geq 2A^{-2}\mu_0^{-1}$, we obtain that in the last inequality
there is in fact equality and, additionally, $\lim_{m\to\infty}|t_m^2 P_\sigma \eee^{-\vec y_m}| = 0$.
In other words, $\eee^{-\vec y_m} = \lambda_m\vec\sigma + o(t_m^{-2})$. The coefficient $\lambda_m$ is determined by
\begin{equation}
2(At_m)^{-2}\mu_0^{-1} + o(t_m^{-2}) = \rho(t_m) = \vec 1\cdot \eee^{-\vec y_m} = \lambda_m \vec 1\cdot \vec \sigma + o(t_m^{-2}),
\end{equation}
thus $\lambda_m = 2(At_m)^{-2} + o(t_m^{-2})$, so $\eee^{-\vec y_m} = 2(At_m)^{-2}\vec\sigma + o(t_m^{-2})$,
which yields \eqref{eq:y-bound-seq} after taking logarithms.
\end{proof}
\begin{remark}
From the proof, one can see that $\lim_{m\to \infty}|t_m\vec q(t_m) - 2\times \vec 1| = 0$.
Below, we conclude the modulation analysis without using this information.
\end{remark}

We decompose
\begin{equation}
\label{eq:yq-syst}
\begin{gathered}
\vec y(t) = r(t)\vec 1 + \vec z(t),\qquad \vec z(t) := P_1 \vec y(t), \\
\vec q(t) = b(t)\vec 1 + \vec w(t), \qquad \vec w(t) := P_1 \vec q(t).
\end{gathered}
\end{equation}
We thus have
\begin{equation}
\label{eq:y'q'-syst}
\begin{gathered}
{\vec y \,}'(t) = r'(t)\vec 1 + {\vec z \,}'(t),\qquad {\vec z \,}'(t) = P_1\vec y'(t), \\
{\vec q \,}'(t) = b'(t)\vec 1 + {\vec w \,}'(t), \qquad {\vec w \,}'(t) = P_1 {\vec q \,}'(t).
\end{gathered}
\end{equation}
We denote $z_{\min}(t) := \min_{1 \leq k \leq n-1}z_j(t)$. We note that $z_{\min}(t) \leq 0$ and $y_{\min}(t) = r(t) + z_{\min}(t)$.
Observe also that $|z_{\min}(t)| \simeq |\vec z(t)|$, since both quantities are norms on the hyperplane $\Pi$, and that $\eee^{-z_{\min}(t)} \simeq \vec 1\cdot \eee^{-\vec z(t)}$. Hence,
Lemma~\ref{lem:zcr} \ref{it:zcr-3} yields
\begin{equation}
\label{eq:Psigexp-equiv}
\begin{aligned}
\eee^{-r(t)}|P_\sigma \eee^{-\vec z(t)}|^2 &\simeq \eee^{z_{\min}(t) - y_{\min}(t)}\vec 1\cdot\eee^{-\vec z(t)}(\vec 1\cdot \eee^{-\vec z(t)} - \vec 1 \cdot \eee^{-\vec z_\tx{cr}}) \\
&\simeq \eee^{-y_{\min}(t)}(\vec 1\cdot \eee^{-\vec z(t)} - \vec 1 \cdot \eee^{-\vec z_\tx{cr}}).
\end{aligned}
\end{equation}
Taking the inner product of \eqref{eq:yq-syst} and \eqref{eq:y'q'-syst} with $\vec\sigma$, and using \eqref{eq:y'-ineq}--\eqref{eq:q'-ineq},
we obtain
\begin{equation}
\label{eq:rrho-syst}
\begin{aligned}
|r'(t) - b(t)| &\lesssim \eee^{-r(t)}\eee^{-z_{\min}(t)}, \\
|b'(t) + \mu_0 A^2 \eee^{-r(t)}  \vec 1\cdot \eee^{-\vec z(t)}| &\lesssim y_{\min}(t)^{-1}\eee^{-r(t)}\eee^{-z_{\min}(t)}.
\end{aligned}
\end{equation}
From this, again using \eqref{eq:y'-ineq}--\eqref{eq:q'-ineq}, we deduce
\begin{equation}
\label{eq:zw-syst}
\begin{aligned}
|{\vec z \,}'(t) - \vec w(t)| &\lesssim \eee^{-r(t)}\eee^{-z_{\min}(t)}, \\
|{\vec w \,}'(t) + \eee^{-r(t)}A^2P_1 \Delta^{(n)}\eee^{-\vec z(t)}| &\lesssim y_{\min}(t)^{-1}\eee^{-r(t)}\eee^{-z_{\min}(t)}.
\end{aligned}
\end{equation}
In Lemma~\ref{lem:y-bound-seq}, we proved that $\liminf_{t\to\infty}|\vec z(t) - \vec z_\tx{cr}| = 0$.
Our next goal is to improve this information to continuous time convergence, in other words we prove a no-return lemma.
\begin{proposition}
If $\bs\phi$ is a kink $n$-cluster, then
\begin{equation}
\label{eq:mod-conclusion}
\lim_{t \to \infty}\big(\big|\vec y(t) - \big(2\log(At)\vec 1 - \log(2\vec\sigma)\big)\big| + \big| t\vec q(t) - 2\times \vec 1 \big| \big) = 0.
\end{equation}
\end{proposition}
\begin{proof}
\textbf{Step 1.}
Consider the functions
\begin{equation}
\xi(t) := \vec 1\cdot \eee^{-\vec z(t)} - \vec 1\cdot \eee^{-\vec z_\tx{cr}}, \qquad \zeta(t) := \vec w(t)\cdot \eee^{-\vec z(t)} - c_0 t^{-1}\xi(t),
\end{equation}
where $c_0 > 0$ will be chosen below.
We know that $\liminf_{t\to\infty} \xi(t) = 0$, and we will improve this to $\lim_{t\to\infty}\xi(t) = 0$.

We claim that for every $\epsilon > 0$ there exist $c_0 = c_0(\epsilon) > 0$ and $t_0 = t_0(\epsilon)$ such that for all $t \geq t_0$
\begin{equation}
\xi(t) \geq \epsilon\ \text{and}\ \zeta(t) \leq 0 \ \Rightarrow \ \zeta'(t) \leq 0.
\end{equation}
Indeed, applying successively \eqref{eq:zw-syst}, \eqref{eq:matrix-D} and
\eqref{eq:Psigexp-equiv}, we get
\begin{equation}
\begin{aligned}
&\dd t \vec w(t)\cdot \eee^{-\vec z(t)} \\
&\ \leq - \vec w(t)^2\cdot \eee^{-\vec z(t)} - \eee^{-r(t)}A^2  \eee^{-\vec z(t)}\cdot (P_1 \Delta^{(n)}\eee^{-\vec z(t)}) + o(\eee^{-r(t)}\eee^{-2z_{\min}(t)}) \\
&\ \lesssim -\eee^{-r(t)}|P_\sigma \eee^{-\vec z(t)}|^2 + o(\eee^{-r(t)}\eee^{-2z_{\min}(t)})
\lesssim -\eee^{-y_{\min}(t)}(\xi(t) + o(\eee^{-z_{\min}(t)})).
\end{aligned}
\end{equation}
Now observe that $\xi(t) \geq \epsilon$ implies $\xi(t) \geq c_1\eee^{-z_{\min}(t)}$,
where $c_1 = c_1(\epsilon) > 0$. Hence, for $t$ large enough
\begin{equation}
\label{eq:wez-deriv}
\dd t \vec w(t)\cdot \eee^{-\vec z(t)}\lesssim -\eee^{-y_{\min}(t)}\xi(t) \lesssim -t^{-2}\xi(t).
\end{equation}

We have $-(c_0t^{-1}\xi(t))' = c_0t^{-2}\xi(t) - c_0t^{-1}\xi'(t)$.
If $c_0$ is small enough, then \eqref{eq:wez-deriv} allows to absorb the first term.
Regarding the second term, \eqref{eq:zw-syst}
and the assumption $\zeta(t) \leq 0$ yield
\begin{equation}
-\xi'(t) = \vec w(t)\cdot \eee^{-\vec z(t)} + O(\eee^{-r(t)}\eee^{-2z_{\min}(t)})
\leq c_0 t^{-1}\xi(t) + O(t^{-2}\xi(t)).
\end{equation}
For $c_0$ small enough and $t$ large enough, the term ${-}c_0t^{-1}\xi'(t)$ can thus be absorbed as well,
which finishes Step 1.

\textbf{Step 2.}
We are ready to prove that
\begin{equation}
\label{eq:z-conv}
\lim_{t\to\infty} \vec z(t) = \vec z_\tx{cr}.
\end{equation}
Suppose this is false, so there exists $\epsilon_0 \in (0, 1)$ and a sequence $t_m \to \infty$ such that
\begin{equation}
\xi(t_m) = \epsilon_0, \qquad \xi'(t_m) \geq 0
\end{equation}
(this is not the same sequence as in Lemma~\ref{lem:y-bound-seq},
but we denote it by the same symbol in order to simplify the notation).
Take $0 < \epsilon \ll \epsilon_0$, $c_0 = c_0(\epsilon)$ and $m$ large (depending on $\epsilon$).
Using \eqref{eq:zw-syst}, we have
\begin{equation}
\zeta(t_m) = -\xi'(t_m)- c_0 t_m^{-1}\xi(t_m) + O(\eee^{-r(t_m)}\eee^{-2z_{\min}(t_m)}) \lesssim -c_0\epsilon_0 t_m^{-1}.
\end{equation}
Let $\tau_m$ be the first time $\tau_m \geq t_m$ such that $\xi(\tau_m) = \epsilon$.
On the time interval $[t_m, \tau_m]$ we have $\xi(t) \geq \epsilon$, so Step 2. implies that $\zeta$ is decreasing,
provided $m$ is large enough, thus $\zeta(\tau_m) \lesssim -c_0\epsilon_0 t_m^{-1} \leq -c_0\epsilon_0 \tau_m^{-1}$.
On the other hand, $\xi'(\tau_m) \leq 0$, which, by the same computation as above, leads to $\zeta(\tau_m) \gtrsim -c_0\epsilon \tau_m^{-1}$, a contradiction.

\textbf{Step 3.}
We claim that
\begin{equation}
\label{eq:w-conv}
\lim_{t \to \infty} t|\vec w(t)| = 0.
\end{equation}
We use a Tauberian argument. From \eqref{eq:zw-syst}, we have $|{\vec z \,}'(t) - \vec w(t)|
+ |{\vec w \,}'(t)| \leq C t^{-2}$ for all $t$.
If there existed $\epsilon > 0$, $k \in \{1, \ldots, n-1\}$ and a sequence $t_m \to \infty$ such that $w_k(t_m) \geq \epsilon t_m^{-1}$ for all $m$,
then we would have $w_k(t) \geq \frac 12\epsilon t_m^{-1}$ for all $t \in \big[t_m, \big(1 + \frac{\epsilon}{2C}\big)t_m\big]$, hence for $m$ large enough $z_k'(t) \geq \frac 14\epsilon t_m^{-1}$ for all $t \in \big[t_m, \big(1 + \frac{\epsilon}{2C}\big)t_m\big]$,
contradicting the convergence of $z_k(t)$ as $t \to \infty$. The case $w_k(t_m) \leq {-}\epsilon t_m^{-1}$ is similar.

From \eqref{eq:zcr-lagr}, we have $\mu_0 \vec 1\cdot \eee^{-\vec z_\tx{cr}} = \lambda_\tx{cr}$.
Thus, \eqref{eq:rrho-syst} and \eqref{eq:z-conv} yield
\begin{equation}
\label{eq:rrho-syst-2}
|r'(t) - b(t)| \lesssim \eee^{-r(t)}, \qquad |b'(t) + A^2\lambda_\tx{cr} \eee^{-r(t)}| \ll \eee^{-r(t)}.
\end{equation}
From Lemma~\ref{lem:ymin-asym}, \eqref{eq:yq-syst} and \eqref{eq:z-conv} we have $\sup_{t\geq 0}|r(t) - 2\log t| < \infty$. Furthermore, \eqref{eq:q-ineq} and \eqref{eq:w-conv} yield
$\sup_{t\geq 0}t|b(t)| < \infty$. Hence,
\begin{equation}
\begin{aligned}
\dd t\Big( \frac 12 b(t)^2 - A^2\lambda_\tx{cr}\eee^{-r(t)} \Big)
&= b(t)\big(b'(t) + A^2 \lambda_\tx{cr}\eee^{-r(t)}\big) \\ 
&+ A^2\lambda_\tx{cr}(r'(t) - b(t))\eee^{-r(t)} = o(t^{-3}),
\end{aligned}
\end{equation}
and an integration in $t$ leads to
\begin{equation}
\label{eq:r-1st-ord}
\big(b(t) - A\sqrt{2\lambda_\tx{cr}}\eee^{-\frac 12 r(t)}\big)\big(b(t) + A\sqrt{2\lambda_\tx{cr}}\eee^{-\frac 12 r(t)}\big) =  b(t)^2 - 2A^2\lambda_\tx{cr}\eee^{-r(t)} = o(t^{-2}).
\end{equation}
The second bound in \eqref{eq:rrho-syst-2} yields
$b'(t) < 0$ for all $t$ large enough. Since $\lim_{t\to \infty} b(t) = 0$,
we have $b(t) > 0$ for all $t$ large enough, hence \eqref{eq:rrho-syst-2} and \eqref{eq:r-1st-ord} yield
\begin{equation}
\label{eq:r'-est}
r'(t) - A\sqrt{2\lambda_\tx{cr}}\eee^{-\frac 12 r(t)} = 
b(t) - A\sqrt{2\lambda_\tx{cr}}\eee^{-\frac 12 r(t)} + o(t^{-1}) = o(t^{-1}),
\end{equation}
which implies
\begin{equation}
\dd t \big(\eee^{\frac 12 r(t)}\big) = A\sqrt{\frac{\lambda_\tx{cr}}{2}} + o(1),
\end{equation}
thus
\begin{equation}
\eee^{\frac 12 r(t)} = t A\sqrt{\frac{\lambda_\tx{cr}}{2}}\big(1  + o(1) \big)
\end{equation}
and after taking the logarithm we finally obtain
\begin{equation}
r(t) = 2\log(At) + \log\big(\frac 12 \lambda_\tx{cr}\big) + o(1).
\end{equation}
Invoking again \eqref{eq:r'-est}, we also have $b(t) = 2t^{-1} + o(1)$.
Hence, \eqref{eq:mod-conclusion} follows from \eqref{eq:yq-syst}, \eqref{eq:zcr-lagr},
\eqref{eq:z-conv} and \eqref{eq:w-conv}.
\end{proof}
\begin{proof}[Proof of Theorem~\ref{thm:asymptotics}]
The bound on $a_{k+1}(t) - a_k(t)$ follows directly
from \eqref{eq:mod-conclusion}, and the definitions of $A$ and $\vec\sigma$.

Let $\conj a(t) := \frac 1n \sum_{j=1}^n a_j(t)$.
Then for all $k \in \{1, \ldots, n\}$, \eqref{eq:mod-conclusion} implies
\begin{equation}
\label{eq:ak'-final}
a_k'(t) = \conj a'(t) + \frac 1n \sum_j (a_k'(t) - a_j'(t))
= \conj a'(t) + (2k - n - 1)t^{-1} + o(t^{-1}),
\end{equation}
hence
\begin{equation}
|\vec a\,'(t)|^2 = (\conj a'(t))^2 + t^{-2}\sum_{k=1}^n(n+1-2k)^2
+ o(t^{-2}).
\end{equation}
An explicit computation yields
\begin{equation}
4\kappa^2 \rho(t) = Mt^{-2}\sum_{k=1}^n(n+1-2k)^2 + o(t^{-2}).
\end{equation}
Applying \eqref{eq:a'-est}, we obtain
\begin{equation}
\lim_{t\to\infty}t\big(\|\partial_t g(t)\|_{L^2} + \|g(t)\|_{H^1} + |\conj a'(t)|\big) = 0.
\end{equation}
The required bound on $a_k'(t)$ follows from the last inequality and \eqref{eq:ak'-final}.
\end{proof}

\section{Existence of a multi-kink for prescribed initial positions}
\label{sec:any-position}
The present section is devoted to a proof of Theorem~\ref{thm:any-position}.
The case $n = 1$ is clear, hence we assume $n > 1$.
We will use the following consequence of Brouwer's fixed point theorem,
known as the Poincar\'e--Miranda theorem.
\begin{theorem}\cite{Miranda}
\label{thm:miranda}
Let $L_1 < L_2$, $\vec y = (y_1, \ldots, y_d) \in \bR^d$ and $\vec \Psi = (\Psi_1, \ldots, \Psi_d) : [L_1, L_2]^d \to \bR^d$ be a continuous map such that
for all $k \in \{1, \ldots, d\}$ the following conditions are satisfied:
\begin{itemize}
\item $\Psi_k(\vec x) \leq y_k$ for all $\vec x = (x_1, \ldots, x_d) \in [L_1, L_2]^d$ such that $x_k = L_1$,
\item $\Psi_k(\vec x) \geq y_k$ for all $\vec x = (x_1, \ldots, x_d) \in [L_1, L_2]^d$ such that $x_k = L_2$.
\end{itemize}
Then there exists $\vec x \in [L_1, L_2]^d$ such that $\vec \Psi(\vec x) = \vec y$.
\end{theorem}
\begin{remark}
The result is often stated with $L_1 = -1$, $L_2 = 1$ and $\vec y = 0$,
which can be achieved by a straightforward change of coordinates.
\end{remark}

\begin{lemma}
\label{lem:data-at-T}
There exist $L_0, C_0 > 0$ such that the following is true.
Let $L \geq L_0$ and $\vec a_0 = (a_{0, 1}, \ldots, a_{0, n}) \in \bR^n$ be such that $a_{0, k+1} - a_{0, k}
\geq L$ for all $k \in \{1, \ldots, n-1\}$.
For any $T \geq 0$ there exists $(\vec a_T, \vec v_T) \in \bR^n \times \bR^n$
such that the solution $\bs \phi$ of \eqref{eq:csf} with
\begin{equation}
\label{eq:phiT}
\bs \phi(T) = \big(H(\vec a_T), \vec v_T \cdot \partial_{\vec a} H(\vec a_T)\big)
\end{equation}
has energy $nM$ and satisfies $\bs \phi(t) = \bs H(\vec a(t)) + \bs g(t)$ for all $t \in [0, T]$, where
\begin{enumerate}[(i)]
\item\label{it:data-at-T-1}
$\bs g(t)$ satisfies \eqref{eq:g-orth},
\item\label{it:data-at-T-2}
$\rho(t) \leq C_0/(e^L + t^2)$ for all $t \in [0, T]$, where $\rho(t)$ is defined by \eqref{eq:rhot-def},
\item\label{it:data-at-T-3}
$\vec a(0) = \vec a_0$.
\end{enumerate}
\end{lemma}
\begin{proof}
\textbf{Step 1.} (Preliminary observations.)
In order to have $\vec a(0) = \vec a_0$,
it suffices to guarantee that $a_{k+1}(0) - a_{k}(0) = a_{0, k+1} - a_{0, k}$,
and apply a translation if needed.
Set $y_{0, k} := a_{0, k+1} - a_{0, k}$ and let $0 < L_1 < L_2$ be chosen later
(depending on $T$ and $\vec y_0$).
The idea is to construct an appropriate function $\vec\Psi: [L_1, L_2]^{n-1} \to \bR^{n-1}$ satisfying the assumptions of Theorem~\ref{thm:miranda}.

For given $\vec y_T \in [L_1, L_2]^{n-1}$, let $\vec a_T \in \bR^n$ be such that $a_{T, k+1} - a_{T, k} = y_{T, k}$ for all $k \in \{1, \ldots, n-1\}$
and $\sum_{k=1}^n a_{T, k} = 0$
(the last condition is a matter of choice, we could just as well impose $a_{T, 1} = 0$
or any condition of similar kind), so that $\vec a_T = \vec a_T(\vec y_T)$ is continuous.
Let $\bs \phi = \bs \phi(T, \vec a_T, \vec v_T)$ be the solution of \eqref{eq:csf}
for the data \eqref{eq:phiT} at time $T$, and let $T_0 \geq 0$ be the minimal time
such that $\delta(\bs \phi(t)) \leq \eta_0$ for all $t \in [T_0, T]$ (we set $T_0 = 0$
if $\delta(\bs \phi(t)) \leq \eta_0$ for all $t \in [0, T]$).
We thus have well-defined modulation parameters $\vec a(t)$ for all $t \in [T_0, T]$.
We define $\vec p(t)$, $\vec y(t)$, $y_{\min}(t)$, $\rho(t)$ and $\vec q(t)$ as in Sections~\ref{sec:mod} and \ref{sec:n-body}.

\noindent
\textbf{Step 2.} (Choice of $\vec v_T$.)
Let $y_{T, \min} := \min_{1 \leq k < n} y_{T, k}$.
We define $\vec v_T = \vec v_T(\vec a_T)$ as the unique vector in $\bR^n$ such that
\begin{enumerate}[(i)]
\item \label{it:choice-vT-sum}
$\sum_{k=1}^n v_{T, k} = 0$,
\item \label{it:choice-vT-lambda}
there exists $\lambda > 0$ such that $v_{T, k+1} - v_{T, k} = \lambda\sqrt{\rho(\vec a_T)}$ for all $k \in \{1, \ldots, n-1\}$,
\item \label{it:choice-vT-en}
$E\big(H(\vec a_T), \vec v_T \cdot \partial_{\vec a} H(\vec a_T)\big) = nM$.
\end{enumerate}
Conditions \ref{it:choice-vT-sum} and \ref{it:choice-vT-lambda} are equivalent to
\begin{equation}
\label{eq:vT-form}
v_{T, k} = \frac{2k-n-1}{2}\lambda\sqrt{\rho(\vec a_T)}, \qquad\text{for all }k \in \{1, \ldots, n\}.
\end{equation}
Condition \ref{it:choice-vT-en} is equivalently written
\begin{equation}
\label{eq:vT-cond}
\|\vec v_T \cdot \partial_{\vec a} H(\vec a_T)\|_{L^2}^2 = 2nM - 2E(H(\vec a_T)).
\end{equation}
Consider the auxiliary function
\begin{equation}
f(\vec a_T) := \int_{-\infty}^\infty \Big| \sum_{k=1}^n \frac{2k-n-1}{2}\partial_x H(x - a_{T, k}) \Big|^2 \ud x.
\end{equation}
Then $f$ is smooth with respect to $\vec a_T$, and \eqref{eq:EpHX-a-1} yields
\begin{equation}
f(\vec a_T) = M\sum_{k=1}^n \Big(\frac{2k-n-1}{2}\Big)^2 + O(L_1\eee^{-L_1}) = M\mu_0^{-1} + O(L_1\eee^{-L_1})
\end{equation}
(the exact value of the first term on the right hand side is of no importance). Thus, \eqref{eq:vT-form}, \eqref{eq:vT-cond}
and \eqref{eq:EpHX-a} yield
\begin{equation}
\lambda^2 = \frac{2nM - 2E(H(\vec a_T))}{\rho(\vec a_T)f(\vec a_T)} = 4M^{-1} \kappa^2 \mu_0 + O(L_1 \eee^{-L_1}),
\end{equation}
implying that $\lambda$ is smooth with respect to $\vec a_T$ and $\lambda \simeq 1$.

\noindent
\textbf{Step 3.} (Definition and continuity of the exit time.)
Note that $\rho(t) \leq n\eee^{-L_1}$ if $y_k(t) \geq L_1$ for all $k$.
We claim that there exists $c_0 > 0$ depending only on $n$ such that
\begin{equation}
\label{eq:sec5-rho'}
\rho'(t) \leq {-}c_0\rho(t)^\frac 32 \qquad\text{for all }t \in [T_0, T].
\end{equation}
Let $\beta$ be defined as in the proof of Lemma~\ref{lem:ymin-asym}.
The same argument shows that $\beta(t)$ is decreasing for $t \in [T_0, T]$ if $c_0$ is small enough.

We claim that $\beta(T) > 0$, which will imply that $\beta(t) > 0$ for all $t \in [T_0, T]$.
The proof of \eqref{eq:rho'-diff-ineq} then applies without changes.
By the definition of $\beta$, it suffices to verify that
\begin{equation}
\label{eq:qkT-pos}
q_k(T) \gtrsim \sqrt{\rho(\vec a_T)}\qquad\text{for all }k \in \{1, \ldots, n-1\}.
\end{equation}
We have $\dot g(T) = {-}\sum_{k=1}^n (-1)^k v_{T, k}\partial_x H(\cdot - a_{T, k})$.
From \eqref{eq:p-def} we obtain $|p_k(T) - Mv_{T, k}| \lesssim \rho(\vec a_T)$,
hence $|q_k(T) - (v_{T, k+1} - v_{T, k})| \lesssim \rho(\vec a_T)$,
which implies \eqref{eq:qkT-pos} since $v_{T, k+1} - v_{T, k} = \lambda\sqrt{\rho(\vec a_T)}$ and $\lambda \simeq 1$.

Let $T_1 \in [T_0, T]$ be the minimal time such that $\rho(t) \leq 2n\eee^{-L_1}$
for all $t \in [T_1, T]$. Suppose that the last inequality holds for all $t \in [T_0, T]$.
Then in particular $\rho(T_0) \leq 2n\eee^{-L_1}$,
and Lemma~\ref{lem:basic-mod} yields $\delta(\bs\phi(T_0)) < \eta_0$ if $L_1$ is large enough,
thus $T_0 = 0$. In this case, we set $T_1 := 0$ as well.

We claim that, for $T$ fixed, $T_1 = T_1(\vec a_T)$ is a continuous function.
Let $\vec a_{T, m} \to \vec a_T$ and consider the corresponding sequence of solutions constructed above.

Assume first that $T_1(\vec a_T) > 0$, thus $\rho(T_1) = 2n\eee^{-L_1}$.
Let $\epsilon > 0$. Since $\rho$ is strictly decreasing by \eqref{eq:sec5-rho'},
we have $\rho(t) < 2n\eee^{-L_1}$ for all $t \in [T_1 + \epsilon, T]$.
By the continuity of the flow, the same inequality holds for the $m$-th solution of the sequence if $m$ is large enough,
hence $T_1(\vec a_{T, m}) \leq T_1 + \epsilon$ if $m$ is large enough.
Similarly, $\rho(T_1 - \epsilon) > 2n\eee^{-L_1}$, implying $T_1(\vec a_{T, m}) \geq T_1 - \epsilon$ for $m$ large enough.

Now assume that $T_1(\vec a_T) = T_0(\vec a_T) = 0$, thus $\rho(0) \leq 2n\eee^{-L_1}$.
Since $\rho$ is strictly decreasing, for any $\epsilon > 0$ we have $\rho(t) < 2n\eee^{-L_1}$ for all $t \in [\epsilon, T]$, and again the continuity of the flow yields $T_1(\vec a_{T, m}) \leq \epsilon$ for $m$ large enough.

\noindent
\textbf{Step 4.} (Application of the Poincar\'e--Miranda theorem.)
We set
\begin{equation}
\vec\Psi(\vec y_T) := \vec y(T_1).
\end{equation}
The preceding step together with the continuity of the flow imply that $\vec\Psi$ is continuous.
Let $k \in \{1, \ldots, n-1\}$ be such that $y_{T, k} = L_2$.
If $L_2$ is sufficiently large, then $\Psi_k(\vec y_T) \geq y_{0, k}$
(it suffices to integrate in time the bound $|y_k'(t)| \lesssim \sqrt{\rho(t)} \lesssim \eee^{-\frac 12 L_1}$).

Assume now that $y_{T, k} = L_1$, thus $\rho(T) \geq \eee^{-L_1}$.
By \eqref{eq:sec5-rho'}, we have
\begin{equation}
\label{eq:deriv-rho-12}
\dd t\big(\rho(t)^{-\frac 12}\big) \gtrsim 1, \qquad\text{for all }t \in [T_1, T].
\end{equation}
Since $\rho(T_1) \leq 2n\eee^{-L_1}$, we obtain
\begin{equation}
T - T_1 \lesssim \eee^{L_1/2}\big(1 - (2n)^{-\frac 12}\big).
\end{equation}
But we also have $|y_k'(t)| \lesssim \eee^{-\frac 12 L_1}$ for all $t \in [T_1, T]$,
hence $|y_k(T_1) - y_k(T)| \lesssim 1$, and it suffices to let $L_1 = L - C$ for $C$ sufficiently large.

By Theorem~\ref{thm:miranda}, there exists $\vec y_T \in [L_1, L_2]^{n-1}$ such that $\vec\Psi(\vec y_T) = \vec y_0$.
In particular, $\rho(T_1) \lesssim \eee^{-L} = \eee^{-C}\eee^{-L_1} \ll 2n\eee^{-L_1}$ if $C$ is large.
We thus have $T_0 = T_1 = 0$ and $\vec y(0) = \vec y_0$.

\noindent
\textbf{Step 5.} (Estimates on the solution.)
By \eqref{eq:deriv-rho-12}, we have
\begin{equation}
\rho(t)^{-\frac 12} - \rho(0)^{-\frac 12} \gtrsim t, \qquad\text{for all }t \in [0, T].
\end{equation}
Since $\rho(0) \lesssim \eee^{-L}$, we obtain $\rho(t) \lesssim (e^L + t^2)^{-1}$.
\end{proof}

\begin{proof}[Proof of Theorem~\ref{thm:any-position}]
Let $T_m$ be an increasing sequence tending to $\infty$ and let
\begin{equation}
\bs \phi_m(t) = \bs H(\vec a_m(t)) + \bs g_m(t)
\end{equation}
be the solution given by Lemma~\ref{lem:data-at-T} for $T = T_m$.
After extraction of a subsequence, we can assume that $\bs g_m(0) \wto \bs g_0 \in \cE$
and, using the Arzel\`a-Ascoli theorem, $\vec a_m \to \vec a$ locally uniformly.
Let $\bs\phi$ be the solution of \eqref{eq:csf} such that $\bs\phi(0) = \bs H(\vec a_0) + \bs g_0$.
We verify that $\bs \phi$ is the desired kink cluster.

By Lemma~\ref{lem:data-at-T} \ref{it:data-at-T-3}, for all $m$ we have $\vec a_m(0) = \vec a_0$, hence
$\la \partial_x H(\cdot - a_{0, k}), g_m(0)\ra = 0$. Taking the weak limit, we get \eqref{eq:g0-orth}.
Fix $t \geq 0$. Since $\bs \phi_m(0) \wto \bs\phi(0)$, Proposition~\ref{prop:cauchy}
implies $\bs \phi_m(t) \wto \bs \phi(t)$, hence 
\begin{equation}
\bs g_m(t)  = \bs\phi_m(t) - \bs H(\vec a_m(t)) \wto \bs \phi(t) - \bs H(\vec a(t)).
\end{equation}
From \eqref{eq:g-coer} and Lemma~\ref{lem:data-at-T} \ref{it:data-at-T-2}, we obtain
\begin{equation}
\|\bs g_m(t)\|^2 \leq C_0/(\eee^L + t^2) \qquad\text{for all }m,
\end{equation}
hence, by the weak compactness of closed balls in separable Hilbert spaces,
\begin{equation}
\|\bs \phi(t) - \bs H(\vec a(t))\|_\cE^2 \leq C_0/(\eee^L + t^2).
\end{equation}
We also have $\rho(\vec a(t)) = \lim_{m\to\infty} \rho(\vec a_m(t)) \leq C_0/(\eee^L + t^2)$,
yielding the required bound on $\delta(\bs\phi(t))$ by the definition of $\delta$, see \eqref{eq:d-def}.
\end{proof}

\section{Kink clusters as profiles of kink collapse}
\label{sec:profiles}
In this final section, we prove Theorem~\ref{thm:unstable}.
In the proof, we will need the following ``Fatou property''.
\begin{lemma}
\label{lem:fatou}
There exist $\eta_0, y_0 > 0$ such that the following holds. Let $\vec a_m = (a_{m, 1}, \ldots, a_{m, n}) \in \bR^n$
for all $m \in \{1, 2, \ldots\}$ be such that $a_{m, k+1} - a_{m, k} \geq y_0$ for all $m$ and $k$, and $\lim_{m \to \infty} \vec a_m = \vec a \in \bR^n$.
Let $\bs g_m \in \cE$, $\|\bs g_m\|_\cE \leq \eta_0$, $\bs g_m$ satisfy the orthogonality conditions \eqref{eq:g-orth}
with $(\vec a, \bs g)$ replaced by $(\vec a_m, \bs g_m)$, and assume that $\bs g_m \wto \bs g \in \cE$. Then
\begin{equation}
\label{eq:fatou}
E(\bs H(\vec a) + \bs g) \leq \liminf_{m \to \infty}E(\bs H(\vec a_m) + \bs g_m),
\end{equation}
and $E(\bs H(\vec a) + \bs g) = \lim_{m \to \infty}E(\bs H(\vec a_m) + \bs g_m)$ if and only if $\bs g_m \to \bs g$ strongly in $\cE$.
\end{lemma}
\begin{proof}
Let $\Pi \subset \cE$ be the codimension $n$ subspace defined by \eqref{eq:g-orth}.
Let $\wt{\bs g}_m$ be the orthogonal projection of $\bs g_m$ on $\Pi$.
Then $\lim_{m \to \infty}\|\wt{\bs g}_m - \bs g_m\|_\cE = 0$, hence we are reduced to the situation where $\vec a_m = \vec a$ for all $m$.
To simplify the notation, we write $\bs g_m$ instead of $\wt{\bs g}_m$.

We have the Taylor expansion around $\bs H(\vec a) + \bs g$:
\begin{equation}
\begin{aligned}
E(\bs H(\vec a) + \bs g_m) &= E(\bs H(\vec a) + \bs g) + \la \vD E(\bs H(\vec a) + \bs g), \bs g_m - \bs g\ra \\
&+ \frac 12 \la \vD^2 E(\bs H(\vec a) + \bs g)(\bs g_m - \bs g), \bs g_m - \bs g\ra + O(\|\bs g_m - \bs g\|_\cE^3).
\end{aligned}
\end{equation}
The second term of the right hand side converges to 0 by assumption.
The second line is $\gtrsim \|\bs g_m - \bs g\|_\cE^2$, by the coercivity estimate in Lemma~\ref{lem:D2H} and smallness of $\|\bs g_m\|_\cE$.
\end{proof}
\begin{proof}[Proof of Theorem~\ref{thm:unstable}]
Let $\vec a_m: [0, T_m] \to \bR^n$ be the modulation parameters corresponding to $\bs\phi_m$.
Let $C_1$ be the constant in \eqref{eq:delta-geq-ass}.
Since $\lim_{m\to\infty}E(\bs \phi_m) = nM$,
there exists a sequence $\wt T_m \in [0, T_m]$
so that $\delta(\bs\phi_m(t)) \geq C_1(E(\bs \phi_m) - nM)$ for all $m$
and $t \in [0, \wt T_m]$, but still $\lim_{m\to\infty}\delta(\bs\phi_m(\wt T_m)) = 0$.
Below, we write $T_m$ instead of $\wt T_m$.

Let $y_{m, k} := a_{m, k+1} - a_{m, k}$ and $\rho_m(t) := \rho(\vec a_m(t))$,
as in Sections~\ref{sec:mod} and \ref{sec:n-body}.
Let $\beta_m: [0, T_m] \to \bR$ be the function defined as in the proof of Lemma~\ref{lem:ymin-asym}, corresponding to the solution $\bs \phi_m$. By the same argument, $\beta(t)$
is decreasing for $t \in [0, T_m]$, hence
\begin{equation}
\beta(t) \geq -\epsilon_m, \qquad\text{for all }t \in [0, T_m],
\end{equation}
where $\epsilon_m := \beta(T_m) \to 0$ as $m \to \infty$.

Fix $t > 0$. For $m$ large enough and all $\tau \in [0, t]$,
we have $|\rho_m'(\tau)| \lesssim \rho_m(\tau)^\frac 32$, hence $\big|\dd t(\rho_m(\tau)^{-\frac 12})\big| \lesssim 1$. Since $\rho_m(0) \gtrsim \eta$, we obtain $\rho_m(\tau)^{-\frac 12} \lesssim \eta^{-\frac 12} + \tau$, hence
\begin{equation}
\rho_m(\tau) \gtrsim \frac{1}{\eta^{-1} + t^2}\qquad\text{for all }m\text{ large enough and }\tau \in [0, t].
\end{equation}
The proof of \eqref{eq:rho'-diff-ineq} yields
\begin{equation}
\dd t\big((\rho_m(\tau))^{-\frac 12}\big) - \frac{c_0}{2} \gtrsim {-}\epsilon_m\rho_m(\tau)^{-\frac 32} \gtrsim {-}\epsilon_m (\eta^{-\frac 32} + t^3).
\end{equation}
Integrating in time, we deduce that there is a constant $C_2$ such that
\begin{equation}
\label{eq:profiles-uni-d}
\limsup_{m\to\infty} \rho_m(t) \leq \frac{C_2}{\eta^{-1} + t^2}.
\end{equation}

Upon extracting a subsequence, we can assume that $\lim_{m \to \infty} (a_{m, k+1}(0) - a_{m, k}(0)) \in \bR \cup \{\infty\}$ exists for all $k$.
Observe that at least one of these limits has to be finite, due to the bound \eqref{eq:g-coer} and the fact that $\delta(\bs \phi_m(0)) = \eta$ for all $m$.
We set $n^{(0)} := 0$ and define inductively
\begin{equation}
n^{(j)} := \max \big\{k : \lim_{m\to \infty} a_{m, k}(0) - a_{m, n^{(j-1)}+1}(0) < \infty \big\},
\end{equation}
until we reach $n^{(\ell)} = n$ for some $\ell$. We set $X_m^{(j)} := a_{m, n^{(j)}}(0)$,
so that $\lim_{m \to \infty} \big(a_{m, n^{(j-1)} + k}(0) - X_m^{(j)}\big) \in \bR$
for all $k \in \big\{1, \ldots, n^{(j)} - n^{(j-1)}\big\}$.
By the definition of $n^{(j)}$, conclusion (ii) of the theorem holds.

Again extracting a subsequence and applying the Arzel\`a-Ascoli theorem,
we can assume that $a_{m, n^{(j-1)} + k}(t) - X_m^{(j)}$ converges uniformly
on every bounded time interval. Hence, we can define
\begin{gather}
\vec a\,^{(j)}: [0, \infty) \to \bR^{n^{(j)} - n^{(j-1)}}, \qquad a\,^{(j)}_k(t) := \lim_{m \to \infty} \big(a_{m, n^{(j-1)} + k}(t) - X_m^{(j)}\big).
\end{gather}
Note that $a\,^{(j)}_{n^{(j)} - n^{(j-1)}}(0) = 0$ for all $j$.
From the bound \eqref{eq:profiles-uni-d} we get $\rho\big(\vec a\,^{(j)}\big) \lesssim (\eta^{-1} + t^2)^{-1}$.

For $0 \leq j \leq \ell$, let $\iota^{(j)} := (-1)^{n^{(j)}}$, $\bs\iota^{(j)} := (\iota^{(j)}, 0)$ and for $1 \leq j < \ell$, let $I_m^{(j)} := \big[\frac 89 X_m^{(j)} + \frac 19 X_m^{(j+1)},
\frac 19 X_m^{(j)} + \frac 89 X_m^{(j+1)}\big]$.
For fixed $t > 0$, we have $\lim_{m\to \infty}\|H(\vec a_m(t)) - \iota^{(j)}\|_{H^1(I_m^{(j)})} = 0$,
thus \eqref{eq:profiles-uni-d} yields
\begin{equation}
\label{eq:Im-small-en}
\limsup_{m\to \infty} \|\bs\phi_m(t) - \bs\iota^{(j)}\|_{\cE(I_m^{(j)})}^2 \lesssim (\eta^{-1} + t^2)^{-1}.
\end{equation}
Let $J_m^{(j)} := \big[\frac 56 X_m^{(j)} + \frac 16 X_m^{(j+1)},
\frac 16 X_m^{(j)} + \frac 56 X_m^{(j+1)}\big]$.
By \eqref{eq:Im-small-en} and Lemma~\ref{lem:loc-vac-stab}, we obtain
\begin{equation}
\label{eq:Jm-small-en}
\limsup_{m\to \infty} \|\bs\phi_m(0) - \bs\iota^{(j)}\|_{\cE(J_m^{(j)})}^2 \lesssim (\eta^{-1} + t^2)^{-1}.
\end{equation}
Letting $t \to \infty$, we get
\begin{equation}
\label{eq:Jm-small-en-0}
\limsup_{m\to \infty} \|\bs\phi_m(0) - \bs\iota^{(j)}\|_{\cE(J_m^{(j)})}^2 = 0.
\end{equation}

Next, we divide the initial data $\bs\phi_m(0)$ into $\ell$ regions in the following way. We set
\begin{equation}
\begin{aligned}
\chi^{(1)}_m(x) &:= \chi\Big(\frac{x - X^{(1)}_m}{X^{(2)}_m - X^{(1)}_m}\Big), \\
\chi^{(j)}_m(x) &:= \chi\Big(\frac{x - X^{(j)}_m}{X^{(j+1)}_m - X^{(j)}_m}\Big)
- \chi\Big(\frac{x - X^{(j-1)}_m}{X^{(j)}_m - X^{(j-1)}_m}\Big),\qquad\text{for }j \in \{2, \ldots, \ell-1\}, \\
\chi^{(\ell)}_m(x) &:= 1 - \chi\Big(\frac{x - X^{(\ell)}_m}{X^{(\ell)} - X^{(\ell-1)}_m}\Big)
\end{aligned}
\end{equation}
(in the case $\ell = 1$, we set $\chi^{(1)}_m(x) := 1$).
We now define
\begin{align}
\label{eq:phim0j-def}
\bs\phi^{(j)}_{m, 0} &:= \bs\iota^{(j-1)}\sum_{i = 1}^{j-1}\chi^{(i)}_m + \chi^{(j)}_m \bs \phi_m(0)
+ \bs\iota^{(j)}\sum_{i=j+1}^{\ell}\chi^{(i)}_m, &\text{for }j \in \{1, \ldots, \ell\}
\end{align}
(note that the two sums above are telescopic sums).
From \eqref{eq:Jm-small-en-0}, we deduce
\begin{equation}
\label{eq:phim0j-out}
\begin{gathered}
\lim_{m\to\infty} \|\bs\phi_{m, 0}^{(j)} - \bs\iota^{(j-1)}\|_{\cE(-\infty,\frac 56 X_m^{(j)}+ \frac 16 X_m^{(j-1)})} = 0, \\
\lim_{m\to\infty} \|\bs\phi_{m, 0}^{(j)} - \bs\iota^{(j)}\|_{\cE(\frac 56 X_m^{(j)} + \frac 16 X_m^{(j+1)}, \infty)} = 0.
\end{gathered}
\end{equation}

Let $\wt {\bs \phi}_m^{(j)}$ be the solution of \eqref{eq:csf} for the initial data
$\wt{\bs \phi}_m^{(j)}(0) = \bs\phi_{m, 0}^{(j)}$.
Let
\begin{equation}
\begin{aligned}
K_m^{(1)} &:= \big({-}\infty, \frac 79 X_m^{(1)} + \frac 29 X_m^{(2)}\big], \\
K_m^{(j)} &:= \big[\frac 79 X_m^{(j)} + \frac 29 X_m^{(j-1)}, \frac 79 X_m^{(j)} + \frac 29 X_m^{(j+1)}\big]\qquad\text{for }j \in \{2, \ldots, \ell -1\}, \\
K_m^{(\ell)} &:= \big[\frac 79 X_m^{(\ell)} + \frac 29 X_m^{(\ell-1)}, \infty\big).
\end{aligned}
\end{equation}
By the finite speed of propagation, see Proposition~\ref{prop:cauchy} \ref{it:cauchy-speed},
and the definition of $\bs \phi_m^{(j)}$, we have
\begin{equation}
\wt{\bs \phi}_m^{(j)}(t)\vert_{K_m^{(j)}} = \bs\phi_m(t)\vert_{K_m^{(j)}}\qquad
\text{for any given }t \geq 0\text{ and }m\text{ large enough},
\end{equation}
hence \eqref{eq:profiles-uni-d} and \eqref{eq:g-coer} imply
\begin{equation}
\label{eq:phimjt-in}
\lim_{m\to\infty} \|\wt{\bs\phi}_{m}^{(j)}(t) - \iota^{(j-1)}\bs H(\vec a\,^{(j)}(t, \cdot - X_m^{(j)})\|_{\cE(K_m^{(j)})} \lesssim \frac{1}{\eta^{-1} + t^2}.
\end{equation}
For any given $t \geq 0$, Lemma~\ref{lem:loc-vac-stab} and \eqref{eq:phim0j-out} yield
\begin{equation}
\label{eq:phimjt-out}
\begin{gathered}
\lim_{m\to\infty} \|\wt{\bs\phi}_{m}^{(j)}(t) - \bs\iota^{(j-1)}\|_{\cE(-\infty,\frac 79 X_m^{(j)}+ \frac 29 X_m^{(j-1)})} = 0, \\
\lim_{m\to\infty} \|\wt{\bs\phi}_{m}^{(j)}(t) - \bs\iota^{(j)}\|_{\cE(\frac 79 X_m^{(j)} + \frac 29 X_m^{(j+1)}, \infty)} = 0.
\end{gathered}
\end{equation}
Invoking \eqref{eq:phimjt-in}, we obtain that for all $t \geq 0$
\begin{equation}
\label{eq:profiles-wtphim-bound}
\limsup_{m \to \infty}\big\|\iota^{(j-1)}\wt{\bs\phi}_m^{(j)}\big(t, \cdot + X_m^{(j)}\big) - \bs H\big(\vec a\,^{(j)}(t)\big) \big\|_\cE^2 \lesssim \frac{1}{\eta^{-1} + t^2}.
\end{equation}

After extraction of a subsequence, we can assume that, for all $j \in \{1, \ldots, \ell\}$,
\begin{equation}
\label{eq:P0j-def}
\iota^{(j-1)}\bs \phi_{m, 0}^{(j)}\big(\cdot + X_m^{(j)}\big) \wto \bs P_0^{(j)} \in \cE_{1, \iota^{(j-1)}\iota^{(j)}}.
\end{equation}
%
Let $\bs P^{(j)}$ be the solution of \eqref{eq:csf} such that $\bs P^{(j)}(0) = \bs P^{(j)}_0$.
By Proposition~\ref{prop:cauchy}
\ref{it:cauchy-weak} and \eqref{eq:profiles-wtphim-bound}, $\bs P^{(j)}$ is a kink cluster,
so conclusion (i) of the theorem holds.
In particular, $E(\bs P^{(j)}) = (n^{(j)} - n^{(j-1)})M$.
From \eqref{eq:Jm-small-en} and \eqref{eq:phim0j-def}, we have
\begin{equation}
\lim_{m\to\infty}\sum_{j=1}^\ell E(\bs\phi_{m, 0}^{(j)}) = \lim_{m\to\infty}E(\bs\phi_m) =
nM = \sum_{j=1}^\ell E(\bs P^{(j)}).
\end{equation}
By the last part of Lemma~\ref{lem:fatou}, we thus have strong convergence in \eqref{eq:P0j-def}
for all $j \in \{1, \ldots, \ell\}$.
Elementary algebra yields
\begin{equation}
\bs\phi_m(0) = \bs 1 + \sum_{j=1}^\ell \big(\bs\phi_{m, 0}^{(j)} - \bs \iota^{(j-1)}\big),
\end{equation}
implying conclusion (iii) of the theorem.
\end{proof}

\providecommand{\noopsort}[1]{}

\bigskip
\centerline{\scshape Jacek Jendrej}
\smallskip
{\footnotesize
 \centerline{CNRS and LAGA, Universit\'e  Sorbonne Paris Nord}
\centerline{99 av Jean-Baptiste Cl\'ement, 93430 Villetaneuse, France}
\centerline{\email{jendrej@math.univ-paris13.fr}}
} 
\medskip 
\centerline{\scshape Andrew Lawrie}
\smallskip
{\footnotesize
 \centerline{Department of Mathematics, Massachusetts Institute of Technology}
\centerline{77 Massachusetts Ave, 2-267, Cambridge, MA 02139, U.S.A.}
\centerline{\email{alawrie@mit.edu}}
} 


\begin{thebibliography}{10}

\bibitem{Bogom76}
E.~B. Bogomolny.
\newblock The stability of classical solutions.
\newblock {\em Sov. J. Nucl. Phys.}, 24(4):449--454, 1976.

\bibitem{CJ2}
G.~Chen and J.~Jendrej.
\newblock Kink networks for scalar fields in dimension $1+1$.
\newblock {\em Nonlinear Anal.}, 215(112643), 2022.

\bibitem{CLL}
G.~Chen, J.~Liu, and B.~Lu.
\newblock Long-time asymptotics and stability for the sine-{G}ordon equation.
\newblock {\em Preprint}, arXiv:2009.04260, 2020.

\bibitem{chow-hale}
S.-N. Chow and J.~K. Hale.
\newblock {\em {Methods of Bifurcation Theory}}, volume 251 of {\em Grundlehren
  der mathematischen Wissenschaften}.
\newblock Springer, 1982.

\bibitem{Cote15}
R.~C{\^o}te.
\newblock On the soliton resolution for equivariant wave maps to the sphere.
\newblock {\em Comm. Pure Appl. Math.}, 68(11):1946--2004, 2015.

\bibitem{CMM11}
R.~C{\^o}te, Y.~Martel, and F.~Merle.
\newblock Construction of multi-soliton solutions for the ${L}^2$-supercritical
  {gKdV} and {NLS} equations.
\newblock {\em Rev. Mat. Iberoam.}, 27(1):273--302, 2011.

\bibitem{CoteMunoz}
R.~C{\^o}te and C.~Mu{\~n}oz.
\newblock Multi-solitons for nonlinear {K}lein--{G}ordon equations.
\newblock {\em Forum Math. Sigma}, 2:e15, 38 pages, 2014.

\bibitem{CoteZaag}
R.~C{\^o}te and H.~Zaag.
\newblock Construction of a multisoliton blowup solution to the semilinear wave
  equation in one space dimension.
\newblock {\em Comm. Pure Appl. Math.}, 66(10):1541--1581, 2013.

\bibitem{DelortMasmoudi}
J.~M. Delort and N.~Masmoudi.
\newblock {\em Long-Time Dispersive Estimates for Perturbations of a Kink
  Solution of One-Dimensional Cubic Wave Equations}.
\newblock Memoirs of the European Mathematical Society. EMS Press, 2022.

\bibitem{DuMa}
W.~Dunajski and N.~S. Manton.
\newblock Reduced dynamics of {W}ard solitons.
\newblock {\em Nonlinearity}, 18:1677--1689, 2005.

\bibitem{DM08}
T.~Duyckaerts and F.~Merle.
\newblock Dynamics of threshold solutions for energy-critical wave equation.
\newblock {\em Int. Math. Res. Pap. IMRP}, 2008.

\bibitem{Friedrichs54}
K.~O. Friedrichs.
\newblock Symmetric hyperbolic linear differential equations.
\newblock {\em Comm. Pure Appl. Math.}, 7:345--392, 1954.

\bibitem{GermainPusateri}
P.~Germain and F.~Pusateri.
\newblock Quadratic {K}lein-{G}ordon equations with a potential in one
  dimension.
\newblock {\em Forum Math. Pi}, 10:1--172, 2022.

\bibitem{GiVe85}
J.~Ginibre and G.~Velo.
\newblock The global {C}auchy problem for the non linear {K}lein-{G}ordon
  equation.
\newblock {\em Math. Z.}, 189:487--505, 1985.

\bibitem{GuSi06}
S.~Gustafson and I.~M. Sigal.
\newblock Effective dynamics of magnetic vortices.
\newblock {\em Adv. Math.}, 199:448--498, 2006.

\bibitem{HN2}
N.~Hayashi and P.~I. Naumkin.
\newblock Quadratic nonlinear {K}lein-{G}ordon equation in one dimension.
\newblock {\em J. Math. Phys.}, 53(10):103711, 36 pages, 2012.

\bibitem{Henon74}
M.~H{\'e}non.
\newblock Integrals of the {T}oda lattice.
\newblock {\em Phys. Rev. B}, 9(4):1921--1923, 1974.

\bibitem{HPW82}
D.~Henry, J.~Perez, and W.~Wreszinski.
\newblock {Stability theory for solitary-wave solutions of scalar field
  equations}.
\newblock {\em Comm. Math. Phys.}, 85(3):351 -- 361, 1982.

\bibitem{J-18p-gkdv}
J.~Jendrej.
\newblock Dynamics of strongly interacting unstable two-solitons for
  generalized {K}orteweg-de {V}ries equations.
\newblock {\em Preprint}, arXiv:1802.06294, 2018.

\bibitem{J-18-nonexist}
J.~Jendrej.
\newblock Nonexistence of radial two-bubbles with opposite signs for the
  energy-critical wave equation.
\newblock {\em Ann. Sc. Norm. Super. Pisa Cl. Sci. (5)}, XVIII:1--44, 2018.

\bibitem{JKL1}
J.~Jendrej, M.~Kowalczyk, and A.~Lawrie.
\newblock Dynamics of strongly interacting kink-antikink pairs for scalar
  fields on a line.
\newblock {\em Duke Math. J.}, 171(18):3643--3705, 2022.

\bibitem{JK}
H.~Jia and C.~Kenig.
\newblock Asymptotic decomposition for semilinear wave and equivariant wave map
  equations.
\newblock {\em Amer. J. Math.}, 139(6):1521--1603, 2017.

\bibitem{KevrekidisEtc}
P.~G. Kevrekidis and J.~Cuevas-Maraver, editors.
\newblock {\em A~Dynamical Perspective on the $\phi^4$ Model}, volume~26 of
  {\em Nonlinear Systems and Complexity}.
\newblock Springer, 2019.

\bibitem{KMM}
M.~Kowalczyk, Y.~Martel, and C.~Mu\~{n}oz.
\newblock Kink dynamics in the {$\phi^4$} model: asymptotic stability for odd
  perturbations in the energy space.
\newblock {\em J. Amer. Math. Soc.}, 30(3):769--798, 2017.

\bibitem{KrMaRa09}
J.~Krieger, Y.~Martel, and P.~Rapha{\"e}l.
\newblock Two-soliton solutions to the three-dimensional gravitational
  {H}artree equation.
\newblock {\em Comm. Pure Appl. Math.}, 62(11):1501--1550, 2009.

\bibitem{KrNaSc15}
J.~Krieger, K.~Nakanishi, and W.~Schlag.
\newblock Center-stable manifold of the ground state in the energy space for
  the critical wave equation.
\newblock {\em Math. Ann.}, 361(1--2):1--50, 2015.

\bibitem{LaWa22p}
Y.~Lan and Z.~Wang.
\newblock Strongly interacting multi-solitons for generalized {B}enjamin-{O}no
  equations.
\newblock {\em Preprint}, arXiv:2204.02715, 2022.

\bibitem{LuhrmannSchlag}
J.~L{\"u}hrmann and W.~Schlag.
\newblock Asymptotic stability of the sine-{G}ordon kink under odd
  perturbations.
\newblock {\em Preprint}, arXiv:2106.09605, 2021.

\bibitem{MaVe09}
E.~Maderna and A.~Venturelli.
\newblock Globally minimizing parabolic motions in the newtonian {$N$}-body
  problem.
\newblock {\em Arch. Rat. Mech. Anal.}, 194:283--313, 2009.

\bibitem{MS}
N.~Manton and P.~Sutcliffe.
\newblock {\em Topological solitons}.
\newblock Cambridge Monographs on Mathematical Physics. Cambridge University
  Press, Cambridge, 2004.

\bibitem{Tomasz}
N.~S. Manton, K.~Ole{\'s}, T.~Roma{\'n}czukiewicz, and A.~Wereszczy{\'n}ski.
\newblock Kink moduli spaces: {C}ollective coordinates reconsidered.
\newblock {\em Phys. Rev. D}, 103(025024), 2021.

\bibitem{Martel05}
Y.~Martel.
\newblock Asymptotic ${N}$-soliton-like solutions of the subcritical and
  critical generalized {K}orteweg-de {V}ries equations.
\newblock {\em Amer. J. Math.}, 127(5):1103--1140, 2005.

\bibitem{MaRa18}
Y.~Martel and P.~Rapha{\"e}l.
\newblock Strongly interacting blow up bubbles for the mass critical {NLS}.
\newblock {\em Ann. Sci. {\'E}c. Norm. Sup{\'e}r.}, 51(3):701--737, 2018.

\bibitem{Merle90}
F.~Merle.
\newblock Construction of solutions with exactly $k$ blow-up points for the
  {S}chr{\"o}dinger equation with critical nonlinearity.
\newblock {\em Commun. Math. Phys.}, 129(2):223--240, 1990.

\bibitem{MeZa12-AJM}
F.~Merle and H.~Zaag.
\newblock Existence and classification of characteristic points at blowup for a
  semi-linear wave equation in one space dimension.
\newblock {\em Amer. J. Math.}, 134(3):581--648, 2012.

\bibitem{MeZa12}
F.~Merle and H.~Zaag.
\newblock Isolatedness of characteristic points at blowup for a 1-dimensional
  semilinear wave equation.
\newblock {\em Duke Math. J.}, 161(15):2837--2908, 2012.

\bibitem{Miranda}
C.~Miranda.
\newblock Un'osservazione su un teorema di {B}rouwer.
\newblock {\em Boll. Unione Mat. Ital.}, 3:5--7, 1940.

\bibitem{Abdon22p2}
A.~Moutinho.
\newblock Approximate kink-kink solutions for the $\phi^6$ model in the
  low-speed limit.
\newblock {\em Preprint}, arXiv:2211.09714, 2022.

\bibitem{Abdon22p1}
A.~Moutinho.
\newblock Dynamics of two interacting kinks for the $\phi^6$ model.
\newblock {\em Preprint}, arXiv:2205.04301, 2022.

\bibitem{Abdon22p3}
A.~Moutinho.
\newblock On the collision problem of two kinks for the $\phi^6$ model with low
  speed.
\newblock {\em Preprint}, arXiv:2211.09749, 2022.

\bibitem{NaSc11-1}
K.~Nakanishi and W.~Schlag.
\newblock Global dynamics above the ground state energy for the focusing
  nonlinear {K}lein-{G}ordon equation.
\newblock {\em J. Differential Equations}, 250(5):2299--2333, 2011.

\bibitem{NaSc11-2}
K.~Nakanishi and W.~Schlag.
\newblock Global dynamics above the ground state for the nonlinear
  {K}lein-{G}ordon equation without a radial assumption.
\newblock {\em Arch. Ration. Mech. Anal.}, 203(3):809--851, 2011.

\bibitem{Vinh17}
T.~V. Nguyen.
\newblock Strongly interacting multi-solitons with logarithmic relative
  distance for the {gKdV} equation.
\newblock {\em Nonlinearity}, 30(12):4614--4648, 2017.

\bibitem{OvSi}
Yu.~N. Ovchinnikov and I.~M. Sigal.
\newblock The {G}inzburg--{L}andau equation {III}. {V}ortex dynamics.
\newblock {\em Nonlinearity}, 11:1277--1294, 1998.

\bibitem{RaSz11}
P.~Rapha{\"e}l and J.~Szeftel.
\newblock Existence and uniqueness of minimal mass blow up solutions to an
  inhomogeneous ${L}^2$-critical {NLS}.
\newblock {\em J. Amer. Math. Soc.}, 24(2):471--546, 2011.

\bibitem{ReedSimon}
M.~Reed and B.~Simon.
\newblock {\em Methods of Modern Mathematical Physics}.
\newblock Academic Press, New York, 1975.

\bibitem{Saari2}
D.~G. Saari and N.~D. Hulkower.
\newblock On the manifolds of total collapse orbits and of completely parabolic
  orbits for the {$n$}-body problem.
\newblock {\em J. Differential Equations}, 41:27--43, 1981.

\bibitem{Skyrme}
T.~H.~R. Skyrme.
\newblock A unified field theory of mesons and baryons.
\newblock {\em Nuclear Phys.}, 31:556--569, 1962.

\bibitem{Stuart}
D.~M.~A. Stuart.
\newblock The geodesic approximation for the {Y}ang-{M}ills-{H}iggs equations.
\newblock {\em Comm. Math. Phys.}, 166:149--190, 1994.

\bibitem{Toda70}
M.~Toda.
\newblock Waves in nonlinear lattice.
\newblock {\em Prog. Theor. Phys. Suppl.}, 45:174--200, 1970.

\bibitem{vachaspati}
T.~Vachaspati.
\newblock {\em Kinks and Domain Walls: An Introduction to Classical and Quantum
  Solitons}.
\newblock Cambridge University Press, Cambridge, 2023.

\bibitem{WadOhk}
M.~Wadati and K.~Ohkuma.
\newblock Multiple-pole solutions of modified {K}orteweg-de {V}ries equation.
\newblock {\em J. Phys. Soc. Jpn.}, 51:2029--2035, 1982.

\end{thebibliography}
\end{document}